\documentclass[12pt]{amsart} 
\usepackage{verbatim, latexsym, amssymb, amsmath,color, mathrsfs}
\usepackage{enumitem}
\usepackage{epsfig}
\usepackage{xcolor}
\usepackage{euscript}
\usepackage{stmaryrd}

\makeatletter
\renewcommand{\tocsection}[3]{%
  \indentlabel{\@ifnotempty{#2}{\bfseries\ignorespaces#1 #2\quad}}\bfseries#3}
\renewcommand{\tocsubsection}[3]{%
  \indentlabel{\@ifnotempty{#2}{\ignorespaces#1 #2\quad}}#3}
\newcommand\@dotsep{4.5}
\def\@tocline#1#2#3#4#5#6#7{\relax
  \ifnum #1>\c@tocdepth % then omit
  \else
    \par \addpenalty\@secpenalty\addvspace{#2}%
    \begingroup \hyphenpenalty\@M
    \@ifempty{#4}{%
      \@tempdima\csname r@tocindent\number#1\endcsname\relax
    }{%
      \@tempdima#4\relax
    }%
    \parindent\z@ \leftskip#3\relax \advance\leftskip\@tempdima\relax
    \rightskip\@pnumwidth plus1em \parfillskip-\@pnumwidth
    #5\leavevmode\hskip-\@tempdima{#6}\nobreak
    \leaders\hbox{$\m@th\mkern \@dotsep mu\hbox{.}\mkern \@dotsep mu$}\hfill
    \nobreak
    \hbox to\@pnumwidth{\@tocpagenum{\ifnum#1=1\bfseries\fi#7}}\par% <-- \bfseries for \section page
    \nobreak
    \endgroup
  \fi}
\AtBeginDocument{%
\expandafter\renewcommand\csname r@tocindent0\endcsname{0pt}
}
\def\l@subsection{\@tocline{2}{0pt}{2.5pc}{5pc}{}}
\makeatother

\usepackage{geometry}
\usepackage{hyperref}

\DeclareMathOperator{\im}{Im}
\DeclareMathOperator{\End}{End}

\DeclareMathOperator{\Length}{Length}
\DeclareMathOperator{\Area}{Area}

\DeclareMathOperator{\dist}{dist}

\DeclareMathOperator{\Eren}{E_{\mathrm{ren}}}
\DeclareMathOperator{\E}{E}

\DeclareMathOperator{\Isom}{Isom}

\DeclareMathOperator{\blambda}{\boldsymbol{\lambda}}
\DeclareMathOperator{\Punc}{Punc}

\newtheorem{thm}{Theorem}[section]
\newtheorem{prop}[thm]{Proposition}
\newtheorem{lem}[thm]{Lemma}
\newtheorem{defn}[thm]{Definition}
\newtheorem{cor}[thm]{Corollary}

\theoremstyle{defn}
\newtheorem{rem}[thm]{Remark}

\begin{document}
\title[Random harmonic maps into spheres]
{Random harmonic maps into spheres}
%\title[Almost hyperbolic minimal surfaces in spheres]
%{Almost hyperbolic minimal surfaces in spheres}

\author{Antoine Song}
\address{California Institute of Technology\\ 177 Linde Hall, \#1200 E. California Blvd., Pasadena, CA 91125}
\email{aysong@caltech.edu}

%\AtEndDocument{\bigskip{\footnotesize%
  %\textsc{\\California Institute of Technology\ 177 Linde Hall, \\ \#1200 E. California Blvd., Pasadena, CA 91125}\par 
  %\textit{\\E-mail address}: \texttt{aysong@caltech.edu}
  %}}

\maketitle

\begin{abstract} 
Let $S$ be a punctured Riemann surface with Euler characteristic $\chi(S)<0$.
For any unitary representation $\rho: \pi_1(S) \to U(N)$, we introduce its renormalized energy and its harmonic representatives, which are equivariant harmonic maps from the universal cover of $S$ to the unit sphere in $\mathbb{C}^N$. 
Our main result is that if a sequence of unitary representations $\rho_j$ strongly converges, then their renormalized energies converge to $\frac{\pi}{4}|\chi(S)|$ and the shape of their harmonic representatives converges to a unique rescaled hyperbolic metric. 
Combining this statement with examples of strongly converging representations provided by random matrix theory, we derive the following  applications. 
\begin{itemize}
\item 
If $\pi_1(S)$ is a free group, then for a random $\rho: \pi_1(S) \to U(N)$, the shape of its harmonic representatives concentrates around a rescaled hyperbolic metric with high probability as $N\to \infty$.
\item
For any closed hyperbolic surface, a finite covering admits a harmonic immersion into some Euclidean unit sphere, which is almost isometric after rescaling. 
\item
There are closed, branched, minimal surfaces $\mathfrak{S}_j$ in some Euclidean unit spheres such that 
$\mathfrak{S}_j$ Benjamini-Schramm converges to a rescaled hyperbolic plane as $j\to \infty$, 
and the Gaussian curvature $K_j$ of $\mathfrak{S}_j$ satisfies
$$\lim_{j\to \infty} \frac{1}{\Area(\mathfrak{S}_j)}\int_{\mathfrak{S}_j} |K_j+8|=0.$$
\end{itemize}

\end{abstract}

%\setcounter{tocdepth}{1}
%\tableofcontents

\section*{Introduction}

A smooth map from a Riemann surface to a Riemannian manifold is \emph{harmonic} if it is a critical point of the Dirichlet energy functional. Since the seminal paper by J. Eells and J. Sampson \cite{ES64}, tremendous effort has been devoted to equivariant harmonic maps into nonpositively curved spaces coming from representations of surface groups: given a Riemann surface $S$ and a representation of its fundamental group into the isometry group of a Riemannian manifold $X$ with nonpositive curvature,
$$\rho:\pi_1(S)\to \Isom(X),$$
one can often construct by energy minimization a harmonic map from the universal cover of $S$ to $X$ which is $\pi_1(S)$-equivariant with respect to $\rho$. These equivariant harmonic maps are rigid, for instance they are unique up to natural equivalence \cite{Hartman67}. This type of rigidity plays a fundamental role in the non-abelian Hodge  correspondence \cite{Garcia-Prada09} \cite{Li19}\cite{Thomas24}, in Teichm\"{u}ller theory \cite{DW07}  \cite{JY09} \cite{Labourie19}, and has been a driving force in the development of geometric analysis \cite{SY97}\cite{Jost08}.

What happens when the target space $X$ is not nonpositively curved? 
Consider equivariant harmonic maps into Euclidean spheres.  While rigidity in the traditional sense does not hold, we find a new kind of rigidity which is more probabilistic in nature: when the dimensions of the spheres grow to infinity, the shapes of equivariant harmonic maps tend to concentrate around a unique hyperbolic geometry.

To formalize this statement, we combine harmonic map theory with the topic of unitary representations of surface groups.  This provides a playground for exploring the effect of randomness and high dimensionality on variational objects such as harmonic maps and minimal surfaces.  In particular, we establish a concrete connection between random matrix theory and harmonic map theory, which is the content of  the main theorem.
%: strong convergence of unitary representations of surface groups implies that the corresponding equivariant harmonic maps concentrate around a unique geometry. 

%We introduce methods from differential geometry---especially harmonic map theory---to the topic of unitary representations of surface groups.  This combination provides a playground for exploring the effect of randomness and high-dimensionality on variational objects, such as harmonic maps and minimal surfaces.  Notably, these ideas lead us to a connection between random matrix theory and harmonic map theory: roughly speaking, strong convergence of unitary representations implies that the shapes of corresponding equivariant harmonic maps concentrate around a unique geometry. 

%Although the interplay between geometry and probability has a long history, applications of probabilistic methods to variational objects in geometry like harmonic maps or minimal surfaces have been rare so far. In our setting, the connection can be roughly summarized as follows: strong convergence for unitary representations implies  the asymptotic rigidity of corresponding equivariant harmonic maps from surfaces into Euclidean spheres.

\subsection{Renormalized energy and equivariant harmonic maps}
Consider a punctured Riemann surface $S$, namely a closed oriented surface minus finitely many points and endowed with a conformal structure. Suppose that $S$ has negative Euler characteristic $\chi(S)<0$.
For $N\geq 1$, consider a unitary representation of its fundamental group:
$$\rho:\pi_1(S) \to U(N).$$ 
%Let $\mathbb{S}^{2N-1}$ be the unit sphere in $\mathbb{C}^N$, endowed with its standard Euclidean metric $g_{\mathrm{Eucl}}$. 
The representation $\rho$ induces an isometric action of $\pi_1(S)$ on the unit sphere $(\mathbb{S}^{2N-1},g_{\mathrm{Eucl}})$ in $\mathbb{C}^N$ with its standard Euclidean metric $g_{\mathrm{Eucl}}$.

One of the simplest geometric invariants for the pair $(S,\rho)$ is its \emph{energy}:
\begin{align*}
\E(S, \rho) := \inf\{& \frac{1}{2}\int_{\mathbf{D}_S} |d\varphi(x)|^2; \\
 & \text{$\varphi$ is a $\pi_1(S)$-equivariant smooth map from $\tilde{S}$ to $\mathbb{S}^{2N-1}$}\}
 \end{align*}
where $\tilde{S}$ is the universal cover of $S$, $\pi_1(S)$ acts  by deck transformations on $\tilde{S}$, $\mathbf{D}_S\subset \tilde{S}$ is any Borel fundamental domain of this action, and the $L^2$-norm of the differential of $\varphi$ is computed with respect to $g_{\mathrm{Eucl}}$ on $\mathbb{S}^{2N-1}$ and any Riemannian metric on $S$ compatible with it conformal structure. 

When $S$ has punctures, $\E(S, \rho) $ is typically  infinite. 
By applying a standard renormalization procedure that originated in the work of F. Bethuel, H. Brezis and F. H\'{e}lein \cite{BBH94}, we define our main invariant, the \emph{renormalized energy} of $(S,\rho)$:
\begin{align*}
\Eren(S, \rho) := \inf\{& \text{renormalized energy of $\varphi$}; \\
 & \text{$\varphi$ is a $\pi_1(S)$-equivariant smooth map from $\tilde{S}$ to $\mathbb{S}^{2N-1}$}\}
 \end{align*}
It is a finite number and is obtained by minimizing the energy after subtracting off the diverging  logarithmic energy contribution near the punctures (see Subsection \ref{Eren}). When $\E(S, \rho) $ is finite, we have $\E(S, \rho) =\Eren(S, \rho) $. 
%The renormalized energy is closely related to several well-studied invariants in analysis, representation theory, geometry and topology, as we will see in a moment.
%$\Eren(S,\rho)$ is always realized by a $\pi_1(S)$-equivariant  \emph{harmonic map}  from $\tilde{S}$ to $\mathbb{S}^{2N-1}$. 

The main geometric property of the renormalized energy $\Eren(S,\rho)$ is that it is always ``achieved'' by an optimal harmonic map 
$$\psi: \tilde{S}\to ( \mathbb{S}^{2N-1},g_{\mathrm{Eucl}})$$
which is $\pi_1(S)$-equivariant with respect to $\rho$ and equivariantly energy-minimizing, see Theorem \ref{realize}. 
Any such $\psi$ will be called a \emph{harmonic representative} of $(S,\rho)$. Its intrinsic shape is encoded by the pullback metric   
$\psi^*g_{\mathrm{Eucl}}$ on $\tilde{S}$, which descends by equivariance of $\psi$ to a Riemannian metric on $S$, still called $\psi^* g_{\mathrm{Eucl}}$.

\emph{How does the renormalized energy behave? What is the geometry of the  harmonic representatives?}
We will answer these questions in the large $N$ limit, when the representation $\rho$ is chosen at random.

\subsection{Applications}
Before stating our main result, here are some unexpected corollaries.

\textbf{Geometric concentration:}
%The Euler characteristic of $S$ is assumed to be negative, so
The unique conformal, finite-area, complete hyperbolic metric on $S$ is denoted by $g_{\mathrm{hyp}}$. 
Suppose that $\pi_1(S)$ is isomorphic to a nonabelian free group $F_k$ of rank $k\geq 2$.  Since $\chi(S)<0$, this occurs if and only if $S$  has at least one puncture. 
Unitary representations of $\pi_1(S)$ into $U(N)$ are then in one-to-one correspondence with $k$-tuple $(u_1,...,u_k)\in U(N)^k$, and the Haar measure on $U(N)^k$ provides a natural notion of random unitary representation of $\pi_1(S)$, see Subsection \ref{average}. Our first application is a geometric concentration result for the shape of harmonic representatives:
\begin{thm} \label{app1}
Let $\epsilon>0$ and let $K\subset S$ be a compact domain. If $N$ is large enough, then for a random unitary representation $\tau:\pi_1(S)\to U(N)$ and any harmonic representative $\psi$ of $(S,\tau)$, with probability at least $1-\epsilon$,
$$ |\Eren(S,\tau) -\frac{\pi}{4}(k-1)|<\epsilon\quad \text{and} \quad \|\psi^*g_{\mathrm{Eucl}} - \frac{1}{8}g_{\mathrm{hyp}}\|_{C^2(K)}<\epsilon.$$
%and $\psi^*g_{\mathrm{Eucl}}$ is $\epsilon$-close to $\frac{1}{8}g_{\mathrm{hyp}}$ in the $C^2$-topology on $K$.
\end{thm}
Here, the $C^2$-norm is computed with respect to $g_{\mathrm{hyp}}$. Thus, $\Eren(S,\tau)$ tends to concentrate around a value independent of the Riemann surface $S$, once the rank of $\pi_1(S)$ is fixed, while the pullback metric by harmonic representatives tends to concentrate around a rescaled hyperbolic metric which determines the Riemann surface $S$.
%For $N\geq 1$, set 
%$$\mathbb{E}_N(S) := \text{ average of $\Eren(S,\rho)$ over all unitary representations $\rho:\pi_1(S)\to U(N)$}$$
%where the average is with respect to the Haar measure on $U(N)^k$, see Subsection \ref{average}.
%\begin{thm} \label{app1}
%Let $S$ be a punctured Riemannian surface with $\pi_1(S)$ isomorphic to a free group of rank $k\geq2$. Then
%$$ \lim_{N\to \infty}\mathbb{E}_N(S)= \frac{\pi}{4}(k-1).$$
%\end{thm}
%This will follow from a much stronger geometric concentration result, Theorem \ref{concentration}.
%A key point is that the limit of $\mathbb{E}_N(S)$ is \emph{independent} of the Riemann surface $S$, once the rank of $\pi_1(S)$ is fixed.
%The quantity $\mathbb{E}_N$ descends to a function on the moduli space of surfaces, and by the theorem it is mostly close to a constant for large $N$.

\textbf{Special immersions of surfaces into Euclidean spaces:}
Our next corollary concerns an old theme in  differential geometry: 

\emph{What is the best way to immerse a closed surface inside a Euclidean space?} 

Consider a closed hyperbolic surface $(\Sigma,g_{\mathrm{hyp}})$.
%There is a long history of rigidity results that obstruct the existence of surfaces in Euclidean spaces and Euclidean spheres with special curvature conditions. For instance, BYChen, Yau, Bryant ...\cite{}. However, up until know, it was unclear how sharp those obstructions were.
Nash's embedding theorem provides a massive collection of isometric immersions of $(\Sigma,g_{\mathrm{hyp}})$ into Euclidean spaces \cite{Andrews02}.
%, which are ruled by the $h$-principle. 
%On the other hand, an elegant result of R. Bryant states that an isometric immersion of a hyperbolic surface can never also be a harmonic map into a Euclidean sphere \cite[Theorem 2.3]{Bryant85}. Combining this with \cite{..}, 
%such an isometric immersion can never also have parallel mean curvature or parallel normalized mean curvature.
On the other hand, it was later shown that such an isometric immersion of $(\Sigma,g_{\mathrm{hyp}})$ never satisfies other natural conditions: it cannot be a harmonic map into a Euclidean sphere \cite[Theorem 2.3]{Bryant85}, or have parallel mean curvature \cite[Theorem 2]{CL72} \cite[Classification Theorem]{CC73} \cite[Theorem 4]{Yau74} or have parallel normalized mean curvature \cite[Theorem 2]{Chen80}. Since then, determining the boundary between these two regimes of extreme flexibility and rigidity had been open. Our second application provides an essentially optimal answer: 
\begin{thm} \label{app2}
Let $\epsilon>0$. For any closed hyperbolic surface $(\Sigma,g_{\mathrm{hyp}})$, there exists a finite degree covering $(\Sigma',g_{\mathrm{hyp}})$ and a harmonic immersion into a Euclidean unit sphere
$$\psi:(\Sigma',g_{\mathrm{hyp}}) \to (\mathbb{S}^{n_\epsilon},g_{\mathrm{Eucl}})$$
 such that $\|\psi^* g_{\mathrm{Eucl}} - \frac{1}{8} g_{\mathrm{hyp}} \|_{C^2(\Sigma')} <\epsilon$.
 % $\psi^* g_{\mathrm{Eucl}} $ is $\epsilon$-close to $\frac{1}{8} g_{\mathrm{hyp}}$ in the $C^2$-topology. 
\end{thm}
%In particular, it implies that  for any $\epsilon>0$, there is a closed high-genus surface $\Sigma$ embedded in some Euclidean space such that its Gaussian curvature is almost constant and its mean curvature vector is almost parallel in the normal bundle.
%Note that $\frac{1}{8}$ is exactly the bottom of the Laplace spectrum on the rescaled hyperbolic plane of curvature $$.
%The factor $\frac{1}{8}$  might be minimal among all factors for which the conclusion holds, see \ref{...} for more speculations on the meaning of $\frac{1}{8}$.
Here the $C^2$-norm is computed with respect to $g_{\mathrm{hyp}}$. 
We speculate on the meaning of the factor $\frac{1}{8}$ in Remark \ref{meaning}.
These harmonic immersions are obtained using random constructions, and as of now there is no known construction of such maps based on more conventional, deterministic methods.
To our knowledge, this is the first application of the probabilistic method and random matrix theory to special immersions of surfaces.

%its Gaussian curvature $K$ and its mean curvature $\overrightarrow{H}$ satisfy $$\max_{\Sigma} |K+1|<\epsilon\quad \text{and} \quad \max_{\Sigma} |\nabla^N \overrightarrow{H}|<\epsilon.$$

\textbf{Almost hyperbolic minimal surfaces in spheres:}
For our third application, recall that a closed, branched, minimal surface is the image of a closed surface by a harmonic map which is weakly conformal. %This condition is much more delicate that being harmonic and almost weakly conformal.
The study of minimal surfaces in spheres has a long history. 
Minimal round spheres and minimal flat tori in Euclidean spheres exist, and are completely classified \cite{Calabi67} \cite{Kenmotsu76}\cite[Proposition 3.3]{Bryant85}. 
On the other hand, the obstruction of R. Bryant  \cite[Theorem 2.3]{Bryant85} states that minimal surfaces in Euclidean spheres can never be hyperbolic surfaces \cite[Theorem 2.3]{Bryant85}. This obstruction is one of the few existing results on the relation between minimal surfaces in spheres and negative Gaussian curvature. Curiously, the following problem from S.-T. Yau's 1982 list  \cite[Problem Section, Problem 101]{Yau82} is still open:  

\emph{Is there a closed, minimal surface $\mathfrak{S}$ in a Euclidean sphere with negative Gaussian curvature $K_\mathfrak{S}<0$?}

As a corollary of the main theorem, we show the existence of new minimal surfaces in high dimensional spheres, which are almost hyperbolic on average:
\begin{thm}\label{thm:almost hyperbolic}
There exist a sequence of closed, branched, minimal surfaces
$\mathfrak{S}_j$ in Euclidean unit spheres $(\mathbb{S}^{n_j},g_{\mathrm{Eucl}})$ such that 
$$\lim_{j\to \infty} \frac{1}{\Area(\mathfrak{S}_j)}\int_{\mathfrak{S}_j}|K_{\mathfrak{S}_j}+8|=0.$$
Moreover, $\mathfrak{S}_j$  Benjamini-Schramm converges to $(\mathbb{H}^2,\frac{1}{8}g_{\mathrm{hyp}})$ as $j\to \infty$.
\end{thm}
In this statement, $K_{\mathfrak{S}_j}$ is the Gaussian curvature of $\mathfrak{S}_j$ with respect to the induced metric, 
$(\mathbb{H}^2,\frac{1}{8}g_{\mathrm{hyp}})$ is the rescaled hyperbolic plane with Gaussian curvature $-8$, and Benjamini-Schramm convergence captures the asymptotic geometry of $\mathfrak{S}_j$ around most points (it will be defined in Subsection \ref{min surff}). Although Theorem \ref{thm:almost hyperbolic} is relevant to \cite[Problem Section, Problem 101]{Yau82}, its main point lies elsewhere: 
the qualitative properties of $\mathfrak{S}_j$ are drastically different from the familiar examples of minimal surfaces in $\mathbb{S}^3$ or $\mathbb{S}^4$ \cite{Lawson70} \cite{Bryant82} \cite{KPS88} \cite{KY10}\cite{CS15}. 
%These surfaces should converge in the ``Benjamini-Schramm topology'' to the hyperbolic plane rescaled by $\frac{1}{8}$ as $\epsilon\to 0$, see Remark \ref{.}. 
They  support the heuristic  that negative curvature should be ``typical'' for minimal surfaces in higher-dimensional Riemannian manifolds.

\subsection{Main result: strong convergence and harmonic maps}
Let $S$ be a punctured Riemann surface of genus $\mathbf{g}$ and with $\mathbf{n}\geq0$ punctures.
Suppose that $S$ has negative Euler characteristic $\chi(S)=2-2\mathbf{g}-\mathbf{n}<0$. Its unique conformal, finite-area, complete hyperbolic metric is denoted by $g_{\mathrm{hyp}}$. 
%Let $\rho:\pi_1(S)\to U(N)$ be a unitary representation. 
%As we will see in Theorem \ref{realize}, the renormalized energy $\Eren(S,\rho)$ is always ``achieved'' by an optimal map which is a $\pi_1(S)$-equivariant, equivariantly energy-minimizing, harmonic map 
%$$\psi: \tilde{S}\to ( \mathbb{S}^{2N-1},g_{\mathrm{Eucl}}).$$
%We call any such $\psi$ a \emph{harmonic representative} of $(S,\rho)$. 
%Given a harmonic representative $\psi: \tilde{S}\to ( \mathbb{S}^{2N-1},g_{\mathrm{Eucl}})$ of $(S,\rho)$, the pullback metric $\psi^* g_{\mathrm{Eucl}}$ on $\tilde{S}$ descends by equivariance of $\psi$ to a Riemannian metric on $S$, still denoted by $\psi^* g_{\mathrm{Eucl}}$.

Let us state the definition of strong convergence, an important property coming from free probability and random matrix theory. A sequence of unitary representations  $\rho_j:\pi_1(S)\to U(N_j)$ \emph{strongly converges} when 
for any $z\in \mathbb{C}[\pi_1(S)]$,
$$\lim_{j\to \infty} \|\rho_j(z)\| = \|\lambda_{\pi_1(S)}(z)\|$$
where $\|.\|$ is the operator norm, and $\lambda_{\pi_1(S)}:\pi_1(S) \to \End(\ell^2(\pi_1(S)))$ is the left regular representation (see Subsection \ref{strong cv} in the Appendix).

\begin{thm}[Main theorem] \label{thm:main}
Let $\rho_j:\pi_1(S)\to U(N_j)$ be a sequence of unitary representations and, for each $j\geq1$, let $\psi_j$ be a harmonic representative of $(S,\rho_j)$. 
If $\rho_j$ strongly converges, then
% to the regular representation, then
%$$\lim_{j\to \infty} E_{ren}(S,\rho_j) = \frac{1}{8}\Area(\Sigma,g_{\mathrm{hyp}}) = \frac{\pi}{4}(2\mathbf{g}+\mathbf{n}-2)$$
$$\lim_{j\to \infty} \Eren(S,\rho_j) =\frac{\pi}{4}|\chi(S)|$$
and  $\psi_j^* g_{\mathrm{Eucl}} $ converges to $\frac{1}{8} g_{\mathrm{hyp}}$ in the $C^\infty$-topology on compact subsets of $S$ as $j\to \infty$.
\end{thm}

The proof of Theorem \ref{thm:main} (see Subsection \ref{subsection:cv metric}) gives a stronger conclusion: the harmonic maps $\psi_j$ in the statement subsequentially converge to a limit \emph{minimal immersion}\footnote{Building on this paper, together with Riccardo Caniato and Xingzhe Li, we showed in \cite{CLS25} that this limit minimal surface is actually unique up to rotations of the sphere.} from a surface to an infinite dimensional Hilbert  unit sphere, whose pullback metric is equal to the \emph{rescaled hyperbolic metric $\frac{1}{8} g_{\mathrm{hyp}}$}. 
This theorem connects an analytic condition on norms of linear operators to the geometric behavior of harmonic maps: under this condition, the rescaled hyperbolic metric emerges as the unique limit of pullback metrics. I suspect that, under mild assumptions, the conclusion in this theorem should actually be equivalent to strong convergence. 
The assumption of Theorem \ref{thm:main} is often satisfied thanks to the ``universality'' of strong convergence: a plethora of models of sequences of random unitary representations all strongly converge.
There has been a flurry of activity in random matrix theory on this subject in recent years, starting with the fundamental work of U. Haagerup and S. Thorbj{\o}rnsen \cite{HT05}. Remarkably, there is no known deterministic construction of strongly converging sequences of finite dimensional unitary representations.
Here are a few results from this large body of work (see Subsection \ref{examples} in the Appendix and references in \cite{Magee24}): based on  \cite{HT05}, B. Collins and C. Male showed that sequences of unitary representations of free groups built from random Haar unitaries strongly converge \cite{CM14}; in an important paper, C. Bordenave and B. Collins established the strong convergence of sequences of unitary representations of free groups constructed from random permutations \cite{BC19}; more recently, C.F. Chen, J. Garza-Vargas, J. Tropp and R. van Handel  discovered that random ``stable representations'' of free groups strongly converge, a result which recovers previous results on free groups with new proof methods \cite{CGZTvH24}; 
L. Louder and M. Magee constructed sequences of unitary representations of surface groups which strongly converge \cite{LM25}.

%historically, the first known random sequence of strongly convergent unitary representations of free groups was obtained by HT \cite{} by exponentiating GUE matrices, a model different from the Haar unitary model.

%\begin{itemize}
%\item CM showed that sequences of random unitary representations of free groups, namely the Haar unitary model of random representations, strongly converge \cite{},
%\item BC found that sequences of unitary representations of free groups constructed from random permutations also strongly converge \cite{},
%\item CGVTVH \cite{} showed that random ``stable representations'' of free groups strongly converge, a result which recovers the previous ones with very different methods,
%\item LM constructed sequences of unitary representations of surface groups which strongly converge \cite{},
%\item historically, the first known random sequence of strongly convergent unitary representations of free groups was obtained by HT \cite{} by exponentiating GUE matrices, a model different from the Haar unitary model.
%\end{itemize}

\subsection{Related work and comments}

In the genealogy of harmonic map theory, our results belong to the broad theme of representations of surface groups into Lie groups 
and equivariant harmonic maps into symmetric spaces. The literature on this subject is enormous; we refer the reader to  \cite[Chapters III and XIII]{SY97} \cite{DW07} \cite{JY09} \cite[Chapter 9]{Jost08} \cite{Garcia-Prada09} \cite{Li19} \cite{Thomas24} and the references therein.
%\cite{Labourie13}\cite{BIW14}
%\cite{Donaldson87} \cite{Corlette88} \cite{Toledo89} \cite{Corlette91} \cite{Reznikov92}  \cite{Dunfield99} \cite{Francaviglia04} \cite{BCG07} \cite{Labourie17} \cite{PS17} \cite{DM21}... 
Essentially all previous results focused on nonpositively curved target spaces, where the rigidity of equivariant harmonic maps is leveraged to study representations and their moduli spaces, in a fixed dimension. In contrast, what we establish for representations into unitary groups is a ``probabilistic rigidity'' statement in the regime where the dimension tends to infinity. This eventually hinges on a rigidity result for certain equivariant minimal surfaces in infinite dimensional spheres, see Section \ref{section 2}. Are there analogous phenomena for representations into general Lie groups? 
%The basic idea is that given a Riemann surface $S$ and a representation of $\pi_1(S)$ into some Lie group $G$, which acts by isometries on some symmetric space $X$, one can construct and study a corresponding equivariant harmonic map from $\tilde{S}$ to $X$.
%There are numerous notions of areas, volumes and energies for representations of groups into isometry groups of nonpositively curved symmetric spaces which partly inspired our renormalized energy, see \cite{Toledo89} \cite{Corlette91} \cite{Reznikov92}  \cite{Dunfield99} \cite{Francaviglia04} \cite{BCG07} \cite{Labourie17} \cite{PS17}, etc. 
%Up until now, efforts have been almost exclusively devoted to harmonic maps into nonpositively curved target spaces $X$, due to their rigid behavior.
%rather than say Euclidean spheres. 
%The reason is that such harmonic maps are rigid. For instance they are essentially unique in equivariant homotopy classes \cite{.}. This type of rigidity is a basic ingredient of the nonabelian Hodge correspondence \cite{}. 
%they are for example essentially unique in each homotopy classes, and .  
%While such rigidity is hopeless for harmonic maps into a compact symmetric space, one of the main findings in this work is a new ``probabilistic rigidity'' phenomenon when $G=U(N)$ and $X=\mathbb{S}^{2N-1}$: in high dimensions, the shapes of equivariant harmonic maps into spheres concentrate around a unique hyperbolic geometry with high probability. 

Unitary representations of surface groups is a classical topic on its own \cite{Labourie13}. When the surface has punctures, the fundamental group is free. Its unitary representations are then in one-to-one correspondence with tuples of unitary matrices, which have been intensively studied in random matrix theory and free probability, see \cite{Speicher17} and the references in the Appendix.
When the surface has no punctures, the unitary representations of its fundamental group are related to stable holomorphic vector bundles by M. Narasimhan and C. Seshadri \cite{NS65}. Their moduli space carries a natural symplectic volume form by M. Atiyah and R. Bott \cite{AB83}, and W. Goldman \cite{Goldman84}, whose volume was computed by E. Witten \cite{Witten91} and A.N. Sengupta \cite{Sengupta03}. In \cite{Magee21} \cite{Magee22}, M. Magee computed the expansion of Wilson loops on this moduli space.

This article is part of a series \cite{Antoine23a} \cite{Antoine23b} \cite{Antoine24b} \cite{Antoine25} which initiates a study of orthogonal and unitary representations of groups using geometric concepts such as area, energy, minimal surfaces and harmonic maps. The central geometric objects  are triples $(\Gamma,\rho,\Sigma)$ where $\Gamma$ is a group, $\rho:\Gamma\to \End(V)$ is an orthogonal representation of $\Gamma$, and $\Sigma$ is a $\rho(\Gamma)$-invariant  $m$-dimensional minimal surface in the unit sphere of $V$. When do such triples exist? When are they rigid?
The variational problem for such triples is called the spherical Plateau problem. The harmonic maps we consider in this paper are closely connected, since they are solutions of a ``spherical Dirichlet problem'', see discussion in Subsection \ref{spb}.

The renormalized energy introduced in this paper is closely connected to several invariants in analysis, representation theory, geometry and topology (see \cite{Antoine25} for expanded explanations and related questions):\\
--\emph{Bethuel-Brezis-H\'{e}lein's renormalized energy:}
our renormalized energy is a direct descendant of the renormalized energy introduced by F. Bethuel, H. Brezis and F. H\'{e}lein \cite{BBH94} for harmonic maps from punctured planar domains to the circle. There are numerous generalizations relevant to us, including the generalization to Riemannian targets by A. Monteil, R. Rodiac, J. Van Schaftingen \cite{MRVS22} and the ``modified energy'' considered by G. Daskalopoulos, C. Mese in \cite{DM23}.  \\
--\emph{Spectral gaps:} the renormalized energy can be viewed as a nonlinear analogue of the first Laplace eigenvalue for sections of flat $\mathbb{C}^N$-bundles over surfaces \cite{Zargar22} \cite{Hide23}, and the corresponding harmonic maps into spheres are then nonlinear analogues of first Laplace eigensections.
\\
--\emph{Kazhdan constants:}
those invariants \cite[Remark 1.1.4]{BDLHV08} measure the distance between a  unitary representation $\rho_\Gamma$ of a discrete group $\Gamma$ and the trivial representation in the Fell topology. 
%They are defined by minimizing the ``length'' of unitary representations $\rho_\Gamma\circ \theta$ over all surjective morphisms $\theta$ from a free group $F_k$ to $\Gamma$.  
The renormalized energy could be used to define a notion of 2-dimensional Kazhdan constants.\\
--\emph{Energy for maps into nonpositively curved spaces:}
for nonpositively curved spaces, notions of energy or volume of representations analogous to the renormalized energy have been extensively studied \cite{Toledo89}\cite{Dunfield99} \cite{Francaviglia04} \cite{BCG07} \cite{Labourie17} \cite{PS17}...
Harmonic maps with infinite energy appear in \cite{Lohkamp90} \cite{Wolf91} \cite{JZ97} \cite{DM23} \cite{Sagman2023}...\\
%The ``modified energy'' employed in \cite{DM23} is essentially our renormalized energy in the context of $CAT(0)$ spaces. \\
--\emph{Besson-Courtois-Gallot's spherical volume:}
one of our original motivations for defining the renormalized energy is the spherical volume of G. Besson, G. Courtois and S. Gallot. It is a topological invariant for any closed manifold $M$, defined by minimizing the ``area'' of certain maps equivariant with respect to the regular representation of $\pi_1(M)$ \cite[Section 3]{BCG91}.

Minimal surfaces are a special case of harmonic maps. To our knowledge, there are roughly three general methods to construct large genus minimal surfaces in high dimensional spheres: area-minimization under symmetry assumptions \cite{SY79} \cite{SU82},
%using parametric methods \cite{SY79} \cite{SU82} or geometric measure theory \cite{DLSS18}, 
twistor map methods \cite{Bryant82} \cite{Bryant82b} \cite{Hano96} which produce superminimal surfaces, 
and $\lambda_1$-maximization methods \cite{Nadirashvili96} \cite{ESI00} \cite{Petrides14} (see also \cite{NS15} \cite{MS19}). 
The random minimal surfaces of Theorem \ref{thm:almost hyperbolic} are obtained via the first method. The second method is the most explicit and thus gives more information on the Gaussian curvature of the minimal surfaces. The third method yields minimal surfaces which might potentially\footnote{I learnt this possibility from Misha Karpukhin. It seems supported by an estimate of A. Ros \cite[Subsection 2.4]{Ros22}.}  become almost hyperbolic when the genus gets large, like the minimal surfaces in Theorem \ref{thm:almost hyperbolic}. 
Concerning \cite[Problem Section, Problem 101]{Yau82}, closed negatively curved minimal surfaces may well already exist in the round 4-sphere $\mathbb{S}^4$. Such surfaces do not exist in $\mathbb{S}^3$ by \cite{Lawson70} and they cannot be superminimal surfaces in $\mathbb{S}^4$ \cite{Bryant82} by \cite{Bryant24}.

%By an elementary result of T. Takahashi \cite{Takahashi66},  the coordinate functions on any minimal surface $\Sigma$ in a Euclidean unit sphere are Laplace eigenfunctions with eigenvalue $2$ on $\Sigma$ with respect to the induced metric. 
%Minimal surfaces immersed by first eigenfunctions naturally appear in optimization problems for the first Laplace eigenvalue $\lambda_1$ on surface \cite{Nadirashvili96} \cite[Theorem 1.1]{ESI00}.
%: $\lambda_1$-maximizing metrics on surfaces, if they exist,  are realized by branched minimal immersions in spheres \cite{Nadirashvili96} \cite[Theorem 1.1]{ESI00}. 
%A potential connection with our Theorem \ref{} is that 
%Those minimal surfaces perhaps become almost hyperbolic when the genus of $\Sigma$ gets large, like the minimal surfaces in Theorem \ref{}.\footnote{I learnt this possibility from M. Karpukhin. It seems supported by an estimate of A. Ros \cite[Subsection 2.4]{Ros22}.}  

Our work is closely connected to spectral geometry. Given a Riemann surface $S$ and a representation $\rho:\pi_1(S)\to U(N)$, there is a corresponding twisted flat $\mathbb{C}^N$-bundle $\mathcal{B}$ over $S$. Equivariant maps from $\tilde{S}$ to $\mathbb{S}^{2N-1}$ descend to sections of  $\mathcal{B}$ with constant norm $1$.  
Thus, the renormalized energy $\Eren(S,\rho)$ and harmonic representatives are nonlinear analogues of the first Laplace eigenvalue for sections of $\mathcal{B}$ and first Laplace eigensections of $\mathcal{B}$, respectively. 
One may connect the two sides via Ginzburg-Landau theory, which involves Dirichlet-type energies $E_\epsilon$ depending on a parameter $\varepsilon \in (0,\infty)$ (see \cite[Introduction, Equation (1)]{BBH94}). In the $\varepsilon\to0$ limit, critical points of $E_\epsilon$ correspond to harmonic representatives, while in the $\varepsilon\to\infty$ limit, they correspond to Laplace eigensections.
In view of Theorem \ref{thm:main}, what is the limit shape of first Laplace eigensections, for strongly convergent representations?
As for the $\lambda_1$-maximization methodology \cite{Nadirashvili96} \cite{ESI00} \cite{Petrides14} \cite{NS15} \cite{MS19}, what is the limit shape of  metrics on $S$ maximizing the first Laplace eigenvalue, for strongly convergent representations?

In a recent article \cite{HM23}, W. Hide and M. Magee settled an old open problem by showing that there are closed hyperbolic surfaces with arbitrarily large genus and with $\lambda_1$ arbitrarily close to $\frac{1}{4}$, by studying random coverings of hyperbolic surfaces. Their proof relies on resolvent methods and the strong convergence result of C. Bordenave and B. Collins \cite{BC19}, who treated an analogous problem for the spectral gaps of graphs.   More recent results \cite{Zargar22} \cite{Hide23}, based on the method of \cite{HM23}, imply that the first eigenvalue of the Laplace operator acting on sections of random flat bundles over a hyperbolic surface $(S,g_{\mathrm{hyp}})$ is at least $\frac{1}{4}-\varepsilon$ with high probability in large dimensions. Those techniques provide\footnote{I thank Michael Magee for explaining this to me.} sharp lower bounds for the renormalized energy $\Eren(S,\rho)$ for certain representations $\rho$ in large dimensions, see Remark \ref{magee}. On the other hand, the proof of Theorem \ref{thm:main}, especially the identification of the limit pullback metric, is based on arguments completely different from \cite{HM23}.

%Their proof heavily relies on the strong convergence result of C. Bordenave and B. Collins \cite{BC19}, who treated an analogous problem for the spectral gaps of graphs.  
%This result also plays a fundamental role in our Theorem \ref{thm:almost hyperbolic}, although we use \cite{BC19} in a completely different way compared to \cite{HM23}, because the harmonic map problem is not a spectral gap problem.  In our case, the hyperbolic metric emerges a posteriori in the limit.

%Recent results \cite{Zargar22} \cite{Hide23}, based on the same techniques as \cite{HM23}, imply that the first eigenvalue of the Laplace operator acting on sections of flat bundles $E_N$ over $(\Sigma_{0,3},g_{\mathrm{hyp}})$, corresponding to random representations $\rho_N$ of $F_2$, is at least $\frac{1}{4}-\varepsilon$ with high probability. The relation between these advances and  Theorem \ref{theorem:main}, is that our branched minimal immersions $\psi_n$ can be viewed as Laplace eigensections of the flat bundles $E_N$ over $(\Sigma_{0,3},g_N)$,
%with eigenvalue exactly $2$ and with constant norm, where $g_N$ are Riemannian metrics (possibly with conical singularities) converging to the rescaled hyperbolic metric $\frac{1}{8} g_{\mathrm{hyp}}$ as $N\to \infty$. Are those maps $\psi_n$ first Laplace eigensections? 

The conclusions of Theorems \ref{app2} and \ref{thm:almost hyperbolic} are reminiscent of several results about holomorphic curves (which are a special kind of harmonic maps and minimal surfaces) in complex geometry.
First, the analogue of \cite[Problem Section, Problem 101]{Yau82} for holomorphic curves in complex projective spaces is completely understood thanks to Kodaira embeddings \cite{Tian90}\cite{Mohsen22} and various obstructions \cite{Pereira11} \cite{Hulin00}: there are holomorphic curves in complex projective spaces whose induced metric is close to a hyperbolic metric after rescaling. In a similar vein, the period map for Riemann surfaces yields holomorphic curves inside complex tori whose induced metric is close to a hyperbolic metric after rescaling \cite{Kazhdan70} \cite{Rhodes93} \cite{McMullen13}. All those constructions are explicit, in contrast to Theorem \ref{thm:almost hyperbolic}. Is there an explicit construction of almost hyperbolic minimal surfaces in spheres?

Some earlier settings involved various different notions of ``random'' minimal surfaces.
The study of zeroes of random complex polynomials and random holomorphic sections, a special type of minimal submanifolds,  is surveyed in \cite{BCHM18} \cite{SZ23}. 
Minimal surfaces representing surface subgroups in hyperbolic manifolds and other locally symmetric spaces of nonpositive curvature, have been studied using dynamical and topological methods, see \cite{KM12} \cite{Hamenstadt15}\cite{CMN22} \cite{LN21}\cite{KMS23}. 
%These papers study the average behavior of natural sequences of minimal surfaces, whose existence is shown by combining dynamical and topological methods \cite{KM12}.
 Equidistribution and scarring results from min-max minimal surfaces \cite{MNS19} \cite{GG19} \cite{SZ21} \cite{Li23} shed light on the average behavior of minimal surfaces. See also the survey \cite{Antoine25a}.

\subsection{Outline of proof for the main theorem}

Roughly speaking, given two unitary representations $\tau_1$, $\tau_2$, we say that $\tau_1$ is weakly contained in $\tau_2$ if the linear action of $\tau_1$ can be arbitrarily well approximated or ``imitated'' by the linear action of $\tau_2$. Weak equivalence between $\tau_1$ and $\tau_2$ means that one is weakly contained in the other and vice versa, see Subsection \ref{weak equiv} in the Appendix for definitions. Under the assumptions of the main theorem, there are $\pi_1(S)$-equivariant harmonic maps $\psi_j:\tilde{S}\to \mathbb{S}^{2N_j-1}$ which achieve the renormalized energy $\Eren(S,\rho_j)$. 

\emph{Step 1 (Upper bound for the renormalized energy):}
The strong convergence of $\rho_j$ implies that the direct sum $\bigoplus_{j\geq1} \rho_j$ weakly contains the regular representation $\lambda_{\pi_1(S)}:\pi_1(S)\to \ell^2(\pi_1(S))$.
%an explicit unitary representation $\underline{\rho}_B : \pi_1(S)\to \End(H)$ which is itself weakly equivalent to the regular representation of $\pi_1(S)$. 
For any $\varepsilon>0$, there is an explicit map $\mathscr{P}_\epsilon$ from $\tilde{S}$ to the unit sphere of $\ell^2(\pi_1(S))$, equivariant with respect to  $\lambda_{\pi_1(S)}$, with energy $\frac{\pi}{4}(2\mathbf{g}+\mathbf{n}-2)+\varepsilon$ on a fundamental domain in $\tilde{S}$. 
By an approximation argument, as $j\to \infty$, we construct maps $\varphi_j:\tilde{S}\to \mathbb{S}^{2N_j-1}$ equivariant with respect to $\rho_j$, with renormalized energy at most $\frac{\pi}{4}(2\mathbf{g}+\mathbf{n}-2)+2\varepsilon$. In particular, $\limsup_{j\to \infty} \Eren(S,\rho_j)\leq\frac{\pi}{4}(2\mathbf{g}+\mathbf{n}-2)$.

\emph{Step 2 (Constructing a limit map):}
The harmonic maps $\psi_j$ can be shown to have their standard energy eventually upper bounded by $\limsup_{j\to \infty} \Eren(S,\rho_j)+\varepsilon$ for any $\varepsilon>0$, on each fixed compact subset of a fundamental domain of $\tilde{S}$. 
By the upper bound in Step 1 and standard harmonic map theory, the harmonic maps subsequentially converge on compact subsets to a limit harmonic map $\psi_\infty:\tilde{S}\to \mathbb{S}^\infty$ inside some Hilbert space. By lower semicontinuity of the energy, the map $\psi_\infty$ has energy at most $\frac{\pi}{4}(2\mathbf{g}+\mathbf{n}-2)$ on a fundamental domain.

\emph{Step 3 (Identifying the limit representation):} 
From the limit map $\psi_\infty$, we can construct a limit unitary representation $\rho'$ with respect to which $\psi_\infty$ is equivariant. 
By using the strong convergence of $\rho_j$ a second time (in a much more crucial way), we argue that $\rho'$ is in fact weakly equivalent to the regular  representation $\lambda_{\pi_1(S)}$.

\emph{Step 4 (Lower bound for the renormalized energy):} 
We show that any map $\tilde{S}\to \mathbb{S}_{\ell^2(\pi_1(S))}$ equivariant with respect to $\lambda_{\pi_1(S)}$ has energy  at least $\frac{\pi}{4}(2\mathbf{g}+\mathbf{n}-2)$ on a fundamental domain of $\tilde{S}$, and that an analogous property holds for $\rho'$ by approximation. We deduce from the previous step that $\psi_\infty$ has energy exactly $\frac{\pi}{4}(2\mathbf{g}+\mathbf{n}-2)$ on a fundamental domain, and that it is actually a branched minimal immersion. It also implies that $\lim_{j\to \infty} \Eren(S,\rho_j)=\frac{\pi}{4}(2\mathbf{g}+\mathbf{n}-2)$ as wanted.

\emph{Step 5 (Uniqueness of the limit pullback metric):} 
The last step is to show that any branched minimal immersion $\tilde{S}\to \mathbb{S}^\infty$ equivariant with respect to a representation weakly equivalent to $\lambda_{\pi_1(S)}$, and with energy $\frac{\pi}{4}(2\mathbf{g}+\mathbf{n}-2)$ on a fundamental domain, has to have its pullback metric equal to $\frac{1}{8}g_{\mathrm{hyp}}$. This is proved using a new interpolation argument which takes advantage of the infinite dimensionality of our setting, in addition to the  representation theory of $\mathrm{PSL}_2(\mathbb{R})$. 
It finishes the proof since $\psi_j$ subsequentially converges on compact subsets to $\psi_\infty$.

The key conceptual steps are the use of strong convergence in Step 3, and the uniqueness statement in Step 5.
The way we ended up writing our proof is a bit different from that outline; $\lambda_{\pi_1(S)}$ will in particular be replaced by another weakly equivalent, explicit, ``boundary'' unitary representation $\underline{\rho}_B$. Once the main theorem is proved, the applications (Theorems \ref{app1}, \ref{app2}, \ref{thm:almost hyperbolic}) follow by using well-chosen pairs $(S,\rho_j)$.

\subsection{Organization of the paper}
-- \textbf{Section \ref{section 1}:} Definition of the energy, renormalized energy and area; construction of equivariant harmonic maps; study of the special case of the regular representation.\\
--  \textbf{Section \ref{section 2}:} A general uniqueness result for equivariant minimal surfaces; application to the regular representation.\\
--  \textbf{Section \ref{section:bounds}:} Relation between weak containment and approximation of equivariant maps; constructions of extensions of maps with controlled renormalized energy; application to upper bounds for the renormalized energy.\\
--  \textbf{Section \ref{section 4}:} Proof of the main theorem: convergence of the renormalized energy and the pullback metric.\\
-- \textbf{Section \ref{section 5}:} Applications of the main theorem to different choices of surfaces and representations.\\
--  \textbf{Section \ref{section 6}:} Appendix: definitions of weak containment, weak equivalence and strong convergence; description of the three examples of strongly converging representations which we use.
%\end{itemize}

\section*{Acknowlegments}
I am grateful to  Michael Magee, Ramon van Handel, Beno\^{i}t Collins and Jorge Garza-Vargas for answering my numerous questions about strong convergence of unitary representations; Robert Bryant for a detailed explanation of his work on minimal surfaces in spheres; Misha Karpukhin and Daniel Stern for interesting exchanges about eigenvalue optimization problems.
I would also like to thank Nicos Kapouleas, Xin Zhou, Yuchin Sun and Christine Breiner for helpful discussions about minimal surfaces and harmonic maps; Peter Sarnak, Mario Micallef, Riccardo Caniato and Xingzhe Li for additional conversations.

%Nicos Kapouleas for explanations about gluing constructions of minimal surfaces,
%Xin Zhou, Yuchin Sun and Christine Breiner for conversations about harmonic maps. 
%Additionally, I thank Peter Sarnak and Mario Micallef for additional helpful discussions about this work. 

A.S. was partially supported by NSF grant DMS-2104254. This research was partially conducted during the period A.S. served as a Clay Research Fellow.
%UPDATE NSF SUPPORT

%\textcolor{red}{A revoir ! Voir aussi survey de Magee, refs dans DM Loh etc..}

\vspace{1em}

\section{Renormalized energy and harmonic maps} \label{section 1}
\subsection{Renormalized energy} \label{Eren}

Let $S$ be a punctured Riemann surface: equivalently, $S$ is a closed oriented surface minus finitely many points called punctures, endowed with a conformal structure. We will almost always assume that $S$ is of general type: it means that its Euler characteristic is negative, or equivalently that it admits a (unique) conformal, complete, finite area, hyperbolic metric called $g_{\mathrm{hyp}}$.
Let $\pi_1(S)$ be the fundamental group of $S$, let $\tilde{S}$ be its universal cover which is endowed with the lift of $g$, and on which $\pi_1(S)$ acts by deck transformations properly and free.
%.(this action is fixed once and for all). 
Let $\mathbf{D}_S\subset \tilde{S}$ be a Borel fundamental domain, which we can choose so that the boundary is piecewise smooth.
 
Let $V$ be a complex Hilbert space.
In this paper, the unit sphere of $V$ will always be called $\mathbb{S}_V$. Let 
$$\rho:\pi_1(S)\to \End(V)$$
be a unitary representation. It induces an isometric action of $\pi_1(S)$ on $\mathbb{S}_V$, which is endowed with its standard round Riemannian metric $g_V$.

Consider the space of maps
\begin{equation}\label{def H}
\mathscr{H}_{S,\rho}:= \{\text{$\pi_1(S)$-equivariant smooth maps from $\tilde{S}$ to $\mathbb{S}_V$}\}.
\end{equation}
It is an exercise in topology that this set is nonempty.
Recall that if $u$ is a smooth map from a Riemann surface $S'$ to a Riemannian manifold $(M,g_M)$, the energy of $u$ restricted to a subset $D\subset S'$ is defined as
$$\E(u\vert_D):= \frac{1}{2}\int_D |du(x)|^2 dv_{g_{S'}}(x)$$
where the $L^2$-norm of the norm of the differential $|du(x)|$ is computed with respect to any Riemannian metric $g_{S'}$ on $S'$ compatible with its conformal structure, and the metric $g_M$. 
This is well-defined by conformally invariance of the energy (it remains unchanged after replacing $g_{S'}$ by another conformal metric). 
We drop the subscript $D$ if $D={S'}$.

We introduce the spherical energy, or ``energy'' for short:
\begin{defn} \label{def E}
The spherical energy of $(S, \rho)$ is defined as
$$\E(S, \rho) := \inf\{\E(u\vert_{\mathbf{D}_S}); \quad u\in \mathscr{H}_{S,\rho}\} \in [0,\infty].$$
\end{defn}

When $S$ has at least one puncture, the energy defined above can be infinite in general. In order to get a finite number, we apply the renormalization procedure introduced by Bethuel-Brezis-H\'{e}lein \cite{BBH94} and generalized by Monteil-Rodiac-Van Schaftingen \cite{MRVS22}, Daskalopoulos-Mese \cite{DM23}. This procedure \emph{depends} on a choice of  conformal parameterization for a punctured neighborhood of each puncture by  a punctured flat unit disk. 
There are quite a few interesting choices for such parameterizations, but in this paper we will fix the following distinguished choice\footnote{Our main theorem does not in fact depend on the choice of parameterization.}, used in \cite{Wolpert07} for instance.
To define this parameterization, let us recall that in a complete, finite-area, hyperbolic surface, 
every cusp contains a unique embedded unit area cusp bounded by a curve with constant geodesic curvature, and these unit area cusps are all disjoint. 
Given $p\in  \mathrm{Punc}_S$, let $D(p,1)$ be the open region of $S$ surrounding $p$ corresponding to the unit area cusp in the cusp of $p$ when $S$ is endowed with its unique conformal complete, finite area hyperbolic metric $g_{\mathrm{hyp}}$. There is a conformal diffeomorphism unique up to rotations,
$$P_p: D(p,1)\to \mathbb{D}^*,$$
where $ \mathbb{D}^*$ is the Euclidean unit disk in $\mathbb{R}^2$ minus the origin $O$.
For $r\in (0,1]$, let 
\begin{equation}\label{dpr}
D(p,r):=\{q\in D(p,1); \quad \dist_{\mathrm{Eucl}}(P_p(q),O) < r\}.
\end{equation}
Denote by $\tilde{D}(p,r)$ the preimage of $D(p,r)$ under the natural projection map $\tilde{S}\to S$. 

Given a representation $\rho$ as above and a puncture $p\in \mathrm{Punc}_S$, oriented embedded small loops around the puncture determine a conjugacy class $\mathcal{J}_p \subset \pi_1(S)$.
% which in turn is sent by $\rho$ to a conjugacy class $\mathcal{I}_p$ of isometries of the sphere $\mathbb{S}_V$. 
Set
\begin{equation}\label{blambda}
\blambda_{\rho}(p) := \inf\{\dist_{\mathbb{S}_V}\big(x,\rho(J_p)x\big); \quad  x\in \mathbb{S}_V\}
\end{equation}
where $J_p$ is any element in $\mathcal{J}_p$ (the definition does not depend on this choice), $\rho(J_p)x$ denotes the image of $x$ by $\rho(J_p)$ and $\dist_{\mathbb{S}_V}$ is the standard Riemannian distance on $\mathbb{S}_V$.

For $u\in \mathscr{H}_{S,\rho}$, its renormalized energy is
$$\Eren(u\vert_{\mathbf{D}_S}) := \liminf_{r\to 0} \big[\E(u\vert_{\mathbf{D}_S\setminus \bigcup_{p\in \mathrm{Punc}_S}\tilde{D}(p,r)}) - \sum_{p\in \mathrm{Punc}_S} \frac{\blambda_{\rho}(p)^2}{4\pi} \log\frac{1}{r}\big].$$
Note that if $\blambda_{\rho}(p)=0$ for each $p\in \mathrm{Punc}_S$, then $\Eren(u\vert_{\mathbf{D}_S})=\E(u\vert_{\mathbf{D}_S})$. 
We can now define the renormalized spherical energy, or ``renormalized energy'' for short:
\begin{defn} \label{def Eren}
The renormalized spherical energy of $(S, \rho)$ is defined as
$$\Eren(S, \rho) := \inf\{\Eren(u\vert_{\mathbf{D}_S}) ; \quad u\in \mathscr{H}_{S,\rho}\}.$$
\end{defn}
The renormalized energy is always finite when $\dim_\mathbb{R}V<\infty$, see \cite{MRVS22,DM23}, or Corollary \ref{coro:basic} (2). 
If $S$ is closed (i.e. has no punctures), then $\E(S,\rho)=\Eren(S,\rho)$.
%The definition of $\Eren(S,\rho)$ can be generalized to other choices of  conformal parameterization for punctured neighborhoods of the punctures. 

\begin{rem} [Invariance]
Both $\E(S,\rho)$ and $\Eren(S,\rho)$ remain invariant when we conjugate $\rho$ by a unitary operator. The choice of distinguished conformal parameterization for punctured neighborhoods of the punctures is invariant by conformal diffeomorphism. Hence, $\E(S,\rho)$ and $\Eren(S,\rho)$  induce functions on the product of Teichm\"{u}ller spaces of surfaces and the moduli spaces of representations of their fundamental group.
\end{rem}

\begin{lem}\label{lem:one}
If $\E(S,\rho)$ is finite then $\E(S, \rho)= \Eren(S, \rho)$.
\end{lem}
\begin{proof}
The condition $\E(S,\rho)<\infty$ implies that there is a smooth map $u\in \mathscr{H}_{S,\rho}$ which has finite energy on the fundamental domain $\mathbf{D}_S\subset \tilde{S}$. 
For each puncture $p\in \Punc_S$,  we use the previous notation $\tilde{D}(p,1)$,
the domain $\mathbf{D}_S$ can be chosen so that the closure of  $\mathbf{D}_S\cap \tilde{D}(p,1)$  is conformal to the band $[0,2\pi]\times [0,\infty)$.
% after lifting the conformal parameterization $P_p:D(p,1)\to \mathbb{D}^*$. 
Let $\blambda_{\rho}(p)$ be defined as in (\ref{blambda}). Under the conformal  identification ${\mathbf{D}_S\cap \tilde{D}(p,1)} \approx [0,2\pi]\times [0,\infty)$, there is some deck-transformation $J_p$ of $\tilde{S}$ such that $J_p((0,h)) = (2\pi,h)$ for any $h\geq 0$.
Since by conformal invariance, $\E(u\vert_{[0,2\pi]\times [0,\infty)})  = \E(u\vert_{\mathbf{D}_S\cap \tilde{D}(p,1)}) <\infty$, by Fubini's theorem, there are some $h_j\to \infty$ such that 
$$\lim_{j\to \infty}\int_{[0,2\pi]\times\{h_j\}} |du|^2 =0.$$
By Cauchy-Schwarz, this implies that
$$\lim_{j\to \infty} \dist_{\mathbb{S}_V}\big(u((0,h_j)), u((2\pi,h_j))\big)=0.$$
%for some isometry $I_p$ of the sphere belonging to $\mathcal{I}_p$ as defined before (\ref{blambda}).
Since $u((2\pi,h_j)) = u(J_p((0,h_j))) = \rho(J_p) u((0,h_j))$ by equivariance of $u$, it means that $\blambda_{\rho}(p)=0$ for all $p\in \Punc_S$.
This in turn implies $\E(u\vert_{\mathbf{D}_S}) = \Eren(u\vert_{\mathbf{D}_S})$ by definition of the renormalized energy, so minimizing over $u$, we conclude $\E(S,\rho) = \Eren(S,\rho).$
\end{proof}

We will need the following known monotonicity property of the renormalized energy, see \cite[Lemma 2.11]{MRVS22} for instance. 
%First, for $p\in  \mathrm{Punc}_S$ and $r\in (0,1]$, let ${D}(p,r)$ and  $\tilde{D}(p,r)$ be defined as in (\ref{dpr}).
\begin{lem}\label{monotonicity}
For any $u\in \mathscr{H}_{S,\rho}$, and any $0<r_2<r_1\leq 1$,
\begin{align*}
& \E(u\vert_{\mathbf{D}_S\setminus \bigcup_{p\in \mathrm{Punc}_S}\tilde{D}(p,r_1)}) - \sum_{p\in \mathrm{Punc}_S} \frac{\blambda_{\rho}(p)^2}{4\pi} \log\frac{1}{r_1} \\
\leq & \E(u\vert_{\mathbf{D}_S\setminus \bigcup_{p\in \mathrm{Punc}_S}\tilde{D}(p,r_2)}) - \sum_{p\in \mathrm{Punc}_S} \frac{\blambda_{\rho}(p)^2}{4\pi} \log\frac{1}{r_2}\\
\leq & \Eren(u\vert_{\mathbf{D}_S}).
\end{align*}
\end{lem}
\begin{proof}
Given $p\in \Punc_S$, the domain $\mathbf{D}_S$ can be chosen so that the closure of  $\mathbf{D}_S\cap \tilde{D}(p,r_1)\setminus \tilde{D}(p,r_2) $ is conformal to the band $[0,2\pi]\times [\log\frac{1}{r_1},\log\frac{1}{r_2}]$. By definition of $\blambda_\rho(p)$, equivariance of $u$ and Cauchy-Schwarz,
\begin{align*}
\E(u\vert_{\mathbf{D}_S\cap \tilde{D}(p,r_1)\setminus \tilde{D}(p,r_2)}) 
& = \frac{1}{2}\int_{\log\frac{1}{r_1}}^{\log\frac{1}{r_2}} (\int_{[0,2\pi]} |du|^2 d\theta) ds  \\
& \geq \frac{\blambda_{\rho}(p)^2}{4\pi} \log\frac{1}{r_2}  - \frac{\blambda_{\rho}(p)^2}{4\pi} \log\frac{1}{r_1}.
\end{align*}
Summing over $p\in \mathrm{Punc}_S$, we get the first inequality in the statement. The second inequality follows by letting $r_2\to 0$.
\end{proof}

\vspace{1em}

\subsection{Harmonic maps}

Harmonic maps from a surface are smooth maps which are critical points for the energy functional on any compact domain of the surface \cite{SY97}. Geometrically, the main property of the renormalized energy is that, for finite dimensional representations, it is always realized by a harmonic map. 
Let $S$, $\mathbf{D}_S\subset \tilde{S}$, $\rho:\pi_1(S)\to \End(V)$,  $\mathscr{H}_{S,\rho}$ be as before.

\begin{defn}
Any map $\psi \in \mathscr{H}_{S,\rho}$ such that
$$\Eren(\psi\vert_{\mathbf{D}_S}) = \Eren(S, \rho).$$ 
is called a harmonic representative of $(S,\rho)$.
\end{defn}

Harmonic representatives are clearly harmonic maps.
The notion of harmonic representatives \emph{does not} depend on the specific choice of conformal parameterization of punctured neighborhoods of punctures used to define the renormalized energy.
We have the following existence theorem: 
\begin{thm} \label{realize}
If $\dim_\mathbb{R}V<\infty$, then $(S,\rho)$ admits a harmonic representative.
\end{thm}
\begin{proof}
We outline the proof, which follows from standard arguments, see \cite[Proposition 8.1]{MRVS22} \cite[Section 4]{DM23}. 
Consider a sequence of maps $u_j\in \mathscr{H}_{S,\rho}$ whose renormalized energies converge to the infimum $\Eren(S, \rho)$. We use Lemma \ref{monotonicity} to bound uniformly the energy of $u_j$ on any compact subset of $\mathbf{D}_S$ so that classical compactness arguments apply \cite{SY79,SU82} and a susequential limit map exists for each compact subset of $\mathbf{D}_S$. After a diagonal argument and taking a union of these limits, we get the desired map.
\end{proof}

\vspace{1em}

We will need a ``dimension-free'' $\varepsilon$-regularity theorem for harmonic maps into spheres:
\begin{thm}\label{epsilon reg}

For $r>0$, let $\mathbb{D}_r$ be the standard disk in $\mathbb{R}^2$ of radius $r$.
There is a universal constant $\varepsilon_0>0 $, and for each integer $m\geq0$ there is a constant $A_m>0$ such that the following holds.
Let $V$ be a finite  dimensional Hilbert space, with unit sphere $S_V$. Let $\phi: \mathbb{D}_r\to S_V$ be a harmonic map with energy at most $\varepsilon_0$. Then 
$$\|\phi\|_{C^m(\mathbb{D}_{r/2})} \leq \frac{A_m}{r^m}.$$
 
\end{thm}  
\begin{proof}
This result is classical \cite{SU81}, except maybe for the fact that  the constants $A_m$ do not depend on the dimension of $V$.
By rescaling and conformal invariance of the energy, it suffices to show the theorem for $r=1$.
A bound on $\|\phi\|_{C^1(\mathbb{D}_{r'})} $ for $r'\in (0,1)$, depending only on $r'$, can be deduced from the proof in \cite[pp 149--151]{CM11}, where only a  $C^2$ bound for the second fundamental form of the sphere $S_V$ as a submanifold of $V$ is used, and not its dimension. See \cite[Lemma 3.12]{KS20}, which made the same observation. 

Since we were not able to find a reference, we sketch below how the uniform bounds for the higher derivatives of $\phi$ follow from a bound on $\|\phi\|_{C^1(\mathbb{D}_{r'})} $ for $r'\in (0,1)$.
If $n=\dim_\mathbb{R} V$, the harmonic map equation is given by $$\Delta \phi^i = - |d\phi|^2 \phi^i \quad \text{for all $i=1,...,n$},$$
where $\phi^i$ are the coordinate functions of $\phi$. 
Thus, up to a factor depending on $m$, $|d^m \Delta \phi^i |$ can be bounded by a sum of terms of the form $|d^{m_0} \phi^i| |d^{m_1} \phi| |d^{m_2} \phi |$ where  $m_0,m_1,m_2$ are integers strictly smaller than $m+2$ such that $m_0+m_1+m_2=m+2$. Let $\mathcal{T}(m)$ be the set of such triples of integers $(m_0,m_1,m_2)$. 
In the remaining of the proof,  $C$ denotes a constant (which can change from line to line) depending on certain parameters, but independent of $n$ and $\phi$.
By the usual $L^p$ estimate \cite[Theorem 9.11]{GT77} with $p=2$, for any $m\geq 0$ and any $\delta>0$,
\begin{align*}
\int_{\mathbb{D}_{r'-\delta}}  |d^{m+2} \phi|^2 & = \int_{\mathbb{D}_{r'-\delta}}  \sum_{i=1}^n  |d^{m+2} \phi^i|^2  \leq C \sum_{i=1}^n \big( \int_{\mathbb{D}_{r'}}   |d^m \Delta\phi^i|^2 + \int_{\mathbb{D}_{r'}}   |d^m\phi^i|^2 \big) \\
 & \leq C\sum_{i=1}^n  \big(\int_{\mathbb{D}_{r'}} \sum_{(m_0,m_1,m_2)\in \mathcal{T}(m)} (|d^{m_0} \phi^i| |d^{m_1} \phi| |d^{m_2} \phi |)^2 
 + \int_{\mathbb{D}_{r'}}   |d^m\phi^i|^2 \big)
 \\
 & \leq C \sum_{(m_0,m_1,m_2)\in \mathcal{T}(m)} \int_{\mathbb{D}_{r'}} |d^{m_0} \phi|^2 |d^{m_1} \phi|^2 |d^{m_2} \phi |^2 + C\int_{\mathbb{D}_{r'}}   |d^m\phi|^2.
 \end{align*}
We already know that $|\phi|^2$ and $|d\phi|^2$ are uniformly bounded in $\mathbb{D}_{r'}$. Suppose that for an integer $m\geq 2$,
$ |d^{t} \phi|^2\leq C$ on $\mathbb{D}_{r'}$ for all $t\leq m-1$, and $\max\{\int_{\mathbb{D}_{r'}}  |d^{m} \phi|^2,\int_{\mathbb{D}_{r'}}  |d^{m+1} \phi|^2\}\leq C$, then the inequalities above imply that $\int_{\mathbb{D}_{r'-\delta}}  |d^{m+2} \phi|^2\leq C$ (since $m+(m+1) >m+2$, each factor in $|d^{m_0} \phi| |d^{m_1} \phi| |d^{m_2} \phi |$ is bounded pointwise by $C$ except maybe one). 
By the Sobolev inequalities, $|d^{m} \phi|^2\leq C$ on $\mathbb{D}_{r'-\delta}$. 
By induction, we deduce that for all $m\geq 0$, $\|\phi\|_{C^m(\mathbb{D}_{r''})} \leq C$, where $r''\in (0,1)$ (the constant $C$ depends on $m,r''$ but not on $n,\phi$).
This concludes the proof.
\end{proof}

\subsection{Spherical Plateau problem}
\label{spb}

The spherical Plateau problem refers to a collection of variational problems depending on an $n$-manifold and an orthogonal representation, and whose solutions (if they exist) are $n$-dimensional minimal surfaces in spheres, invariant under a group action \cite[Section 3]{Antoine23a} \cite{Antoine24b}.
In this subsection, we define an  invariant relevant to the spherical Plateau problem: the spherical area.
%its renormalized counterpart is not well-defined due to the choice of metric on punctures!
This invariant is an important motivation for studying the  energy.
It generalizes an invariant of Besson-Courtois-Gallot \cite[Subsections 3.I and 3.II]{BCG91} to arbitrary unitary representations. 

Let $S$ be a punctured Riemann surface of negative Euler characteristic,
 and let $\Sigma$ be the underlying topological surface. 
Let $\mathcal{T}_\Sigma$ be its Teichm\"{u}ller  space. 
The Riemann surface $S$ determines a class in  Teichm\"{u}ller space called $[S]$, and vice versa any Teichm\"{u}ller class is realized by a punctured Riemann surface $S'$ with underlying topological surface $\Sigma$. Consider a unitary representation.
$$\rho:\pi_1(\Sigma)\to \End(H).$$
Define $\mathbf{D}_\Sigma\subset \tilde{\Sigma}$ and $\mathscr{H}_{\Sigma,\rho}$ in a similar way as for $\mathbf{D}_S\subset \tilde{S}$ and $\mathscr{H}_{S,\rho}$.

%We consider the spherical energy of $(\Sigma,\mu_0, \rho)$, which is well-defined by conformal invariance:
\begin{defn} \label{def spherearea bis}
The spherical area of $(\Sigma,\rho)$ is defined as 
$$
\Area(\Sigma,\rho) := \inf\{\E(S', \rho);  \quad [S'] \in \mathscr{T}_\Sigma\}.
$$
\end{defn}
Like the energy, $\Area(\Sigma,\rho) $ could be infinite in general. 
\begin{rem}
If $\Sigma$ is a closed surface, then due to classical arguments (see for instance the proof of \cite[Theorem 3.1]{SY79}), the spherical area of $(\Sigma,\rho)$ satisfies
$$\Area(\Sigma,\rho) = \inf\{ \Area(\mathbf{D}_\Sigma,\phi^*g_V); \quad \phi\in \mathscr{H}_{\Sigma,\rho}\}<\infty.$$
A question for which we do not know the answer is the following: given any  closed surface $\Sigma$, is there a unitary representation $\rho:\pi_1(\Sigma)\to U(n)$ for which $\Area(\Sigma,\rho)>0$?
\end{rem}

Branched minimal surfaces are by definition images of harmonic maps which are weakly conformal.
A well-known general strategy to construct minimal surfaces is to divide an area minimization problem into, first, an energy minimization problem inside a conformal class (leading to the construction of a harmonic map), and then, an energy minimization problem on the Teichm\"{u}ller space of the surface (leading to the desired minimal surface). In our setting, the second step is subtle and will not be treated.  We only state the following known fact:

\begin{thm} \label{min surf thm}
Let $S$ be a punctured Riemann surface, $\Sigma$ the underlying topological surface and $\rho:\pi_1(S)\to \End(H))$ a unitary representation.
Suppose that 
$$0<\E(S,\rho) = \Area(\Sigma,\rho)<\infty.$$ 
Then, for any map $\psi\in \mathscr{H}_{S,\rho}$ such that 
$E(\psi\vert_{\mathbf{D}_S}) = \E(S,\rho)$,  $\psi$ is harmonic and weakly conformal. In other words, $\psi(\tilde{S})$ is a branched minimal surface.  
\end{thm}
\begin{proof}
This is a corollary of a standard result, which  states that a map which is a critical point of the energy with respect to both conformal changes and variations of the conformal class is a branched minimal immersion \cite[Theorem 1.8]{SU81}. 
%Compared to the original statement, a slight subtlety here is that our Riemann surface $S$ is  not necessarily compact.   Nevertheless, the statement of \cite[Theorem 1.8]{SU81} still holds true for any finite area hyperbolic surface  $(\Sigma',g_0)$.We can follow the original proof using that, for any metric $g$ coinciding with the hyperbolic $g_0$ outside of a compact subset of $\Sigma'$, 
%$(\Sigma',g)$ is conformally equivalent to a finite area hyperbolic metric $(\Sigma',g_1)$. Indeed, $g$ is conformal to a unique complete hyperbolic metric $g_1$ with finite area by \cite[Theorem C]{HT92} for instance, and thus corresponds to an element of the Teichm\"{u}ller space of $\Sigma'$. In fact, this metric $g_1$ and the corresponding conformal factor depend smoothly on $g$.

\end{proof}

%\vspace{1em}

\subsection{Energy and area for the regular representation}

Let us describe what happens for the ``standard'' spherical Plateau problem, where the unitary representation is given by the (left) regular representation.  
Let $S$ be a punctured Riemann surface of negative Euler characteristic $\chi(S)<0$, let $\Sigma$ be the underlying topological surface.

The regular representation $\lambda_{\pi_1(S)} : {\pi_1(S)}\to \End(\ell^2({\pi_1(S)},\mathbb{C}))$ of ${\pi_1(S)}$ is the following canonical representation: for all $\gamma,x\in {\pi_1(S)}$ and $f\in \ell^2({\pi_1(S)},\mathbb{C})$,
$$(\lambda_{\pi_1(S)}(\gamma).f) (x) := f(\gamma^{-1}x).$$
It induces a proper free isometric action on the unit sphere $\mathbb{S}_{\ell^2({\pi_1(S)},\mathbb{C})}$ of the Hilbert space $\ell^2({\pi_1(S)},\mathbb{C})$. 
Let $g_{\ell^2({\pi_1(S)},\mathbb{C})}$ be the standard round metric on $\mathbb{S}_{\ell^2({\pi_1(S)},\mathbb{C})}$.
When $\rho=\lambda_{\pi_1(S)}$, the spherical area of $(\Sigma,\rho)$ 
is exactly the ``spherical volume'' of $\Sigma$ considered in \cite{BCG91} \cite{Antoine23a} \cite{Antoine24b}. 
When $\rho=\lambda_{\pi_1(S)}$, the spherical area is also well-defined and easily shown to be equal to $0$ for the 2-sphere and the 2-torus. 
Slightly generalizing Besson-Courtois-Gallot's computation of the standard spherical area of closed surfaces \cite[Remarque 3.1.4 (i)]{BCG91}, we have the following:
\begin{thm}\label{spherearea}
Let $\mathbf{g}$ and $\mathbf{n}$ be respectively the genus and the number of punctures of $S$.
%Fix a complete finite area hyperbolic metric $g_{\mathrm{hyp}}$ on $\Sigma$. 
%Fix $[S]\in \mathscr{T}_\Sigma$.
Then
$$\E(S,\lambda_{\pi_1(S)}) =\Area(\Sigma,\lambda_{\pi_1(S)})= \frac{\pi}{4}(2\mathbf{g}+\mathbf{n}-2) = \frac{\pi}{4}|\chi(S)|.$$
\end{thm}
\begin{proof}
Let $g_{\mathrm{hyp}}$ be  the unique  complete finite area hyperbolic metric on $\Sigma$ compatible with the conformal structure  of $S$.
Note that the number $\frac{\pi}{4}(2\mathbf{g}+\mathbf{n}-2) = \frac{\pi}{4}|\chi(S)|$ in Theorem \ref{spherearea} is exactly  $\Area(\Sigma,\frac{1}{8}g_{\mathrm{hyp}})$.

Let $\phi\in \mathscr{H}_{S,\lambda_{\pi_1(S)}}  $ and view it as a $\mathbb{C}$-valued function of two variables $x\in \tilde{S}$ and $\gamma\in {\pi_1(S)}$. Adapting the proof of \cite[Th\'{e}or\`{e}me 3.8]{BCG91}, we get:
 \begin{align*}
 2E(\phi\vert_{\mathbf{D}_S}) & =  2E(\phi\vert_{(\mathbf{D}_\Sigma,g_{\mathrm{hyp}})}) \\
 & = \int_{\mathbf{D}_\Sigma}\sum_{\gamma\in {\pi_1(S)}} |d_1\phi(x,\gamma)|^2 dv_{g_{\mathrm{hyp}}}(x)  =  \int_{\mathbf{D}_\Sigma}\sum_{\gamma\in {\pi_1(S)}} |d_1\phi(\gamma^{-1}x,1)|^2 dv_{g_{\mathrm{hyp}}}(x)\\
 & =  \sum_{\gamma\in {\pi_1(S)}} \int_{\mathbf{D}_\Sigma}|d_1\phi(\gamma^{-1}x,1)|^2 dv_{g_{\mathrm{hyp}}}(x) =  \int_{\tilde{\Sigma}}|d_1\phi(x,1)|^2 dv_{g_{\mathrm{hyp}}}(x)\\
 & \geq \frac{1}{4} \int_{\tilde{\Sigma}}|\phi(x,1)|^2 dv_{g_{\mathrm{hyp}}}(x)= \frac{1}{4}  \sum_{\gamma\in {\pi_1(S)}} \int_{\mathbf{D}_\Sigma}|\phi(\gamma^{-1}x,1)|^2 dv_{g_{\mathrm{hyp}}}(x)\\
 & = \frac{1}{4}  \sum_{\gamma\in {\pi_1(S)}} \int_{\mathbf{D}_\Sigma}|\phi(x,\gamma)|^2 dv_{g_{\mathrm{hyp}}}(x) = \frac{1}{4} \Area(\Sigma,g_{\mathrm{hyp}}).
 \end{align*}
 The second line holds by ${\pi_1(S)}$-equivariance of $\phi$. The inequality comes from the well-known fact that the bottom of the Laplace spectrum for $W^{1,2}$ functions on the hyperbolic plane $(\tilde{\Sigma},g_{\mathrm{hyp}})$ is $\frac{1}{4}$, and the last equality comes from  the fact that $\sum_{\gamma\in {\pi_1(S)}} |\phi(x,\gamma)|^2 =1$ for all $x$.  Note that we implicitly used the fact that $(\Sigma,g_{\mathrm{hyp}})$ has finite area to ensure that $\phi(.,1)\in W^{1,2}(\tilde{\Sigma})$. 
Since this computation holds for any  $S$ and any $\phi\in \mathscr{H}_{S,\lambda_{\pi_1(S)}}  $,
 \begin{equation}\label{deux ineq}
 \E(S,\lambda_{\pi_1(S)})\geq\Area(\Sigma,\lambda_{\pi_1(S)})\geq \Area(S,\frac{1}{8}g_{\mathrm{hyp}}).
 \end{equation}

To show the reverse inequalities, for any $c>1$, consider the family of maps
\begin{align}\label{poissonc}
\begin{split}
\mathcal{P}_c &:  \tilde{S} \to  \mathbb{S}_{\ell^2({\pi_1(S)},\mathbb{C})}\subset \ell^2({\pi_1(S)},\mathbb{C})\\
 & x\mapsto \{\gamma \mapsto \frac{1}{\|e^{-\frac{c}{2}\dist_{g_{\mathrm{hyp}}}(x,.)}\|_{L^2(\tilde{S},g_{\mathrm{hyp}})}}   \big[\int_{\gamma. \mathbf{D}_S}  e^{-c\dist_{g_{\mathrm{hyp}}}(x,.)} dv_{g_{\mathrm{hyp}}}\big]^{1/2}\}.
 \end{split}
 \end{align}
It is simple to check that $\mathcal{P}_c\in  \mathscr{H}_{S,\lambda_{\pi_1(S)}} $. Furthermore, by \cite[Lemma 1.3]{Antoine23b}, for any $x\in \tilde{S}$, and any orthonormal basis $\{e_1,e_2\}$ of $T_x\tilde{S}$,
$$|d_x\mathcal{P}_c(e_1)|^2+|d_x\mathcal{P}_c(e_2)|^2 \leq \frac{c^2}{4}.$$
Thus, by letting $c$ go to $1$, we get
\begin{equation}\label{une eq e}
\E(S,\lambda_{\pi_1(S)}) \leq \Area(\Sigma,\frac{1}{8}g_{\mathrm{hyp}}).
\end{equation}
Together with (\ref{deux ineq}), this finishes the proof.

\end{proof}

 The notion of direct sum for Hilbert spaces and unitary representations is defined in \cite[Definition A.1.6]{BDLHV08}. 
%Let $S$, $\Sigma$ be as above. 
 Let $\bigoplus^\infty\lambda_{{\pi_1(S)}}$ be the infinite direct sum of $\lambda_{{\pi_1(S)}}$. The corresponding Hilbert space is $\bigoplus^\infty\ell^2({\pi_1(S)},\mathbb{C})$. 
 We will later need the following:
\begin{lem}\label{=+}
We have $$\E(S,\lambda_{\pi_1(S)}) = \Area(\Sigma,\lambda_{{\pi_1(S)}})=\E(S,\bigoplus^\infty\lambda_{{\pi_1(S)}}) = \Area(\Sigma,\bigoplus^\infty\lambda_{{\pi_1(S)}}).$$
\end{lem}
 \begin{proof}
Clearly, since $\lambda_{\pi_1(S)}$ is a subrepresentation of $\bigoplus^\infty\lambda_{{\pi_1(S)}}$, we have
$$\E(S,\lambda_{\pi_1(S)})  \geq \E(S,\bigoplus^\infty\lambda_{{\pi_1(S)}})\quad \text{and}\quad  \Area(\Sigma,\lambda_{{\pi_1(S)}}) \geq \Area(\Sigma,\bigoplus^\infty\lambda_{{\pi_1(S)}}).$$
Next,  we can rewrite $\bigoplus^\infty\ell^2({\pi_1(S)},\mathbb{C})$ as $\ell^2({\pi_1(S)},H)$ where $H$ is the separable infinite dimensional Hilbert space. Define the following map between the unit spheres
$$\mathcal{A} :\mathbb{S}_{\ell^2({\pi_1(S)},H)}\to \mathbb{S}_{\ell^2({\pi_1(S)},\mathbb{C})}$$
such that for any $f\in \mathbb{S}_{\ell^2({\pi_1(S)},H)}$ and $\gamma\in {\pi_1(S)}$, $\mathcal{A}(f)(\gamma) := |f(\gamma)|_H$ where $|.|_H$ denotes the norm in $H$. Clearly, $\mathcal{A}$ is distance non-increasing and equivariant.  
%If $g_{\mathrm{hyp}}$ is a hyperbolic metric on $\Sigma$, 
%Recall that $S$ is $\Sigma$ endowed with a Riemannian metric.
If $\phi\in \mathscr{H}_{\Sigma,\bigoplus^\infty\lambda_{{\pi_1(S)}}}$, then $\mathcal{A}\circ \phi$ is still equivariant. For any $\epsilon>0$, after smoothing $\mathcal{A}\circ \phi$, we get a map $\phi' \in \mathscr{H}_{S,\lambda_{{\pi_1(S)}}}$ such that $E(\phi'\vert_{\mathbf{D}_S}) \leq E(\phi\vert_{\mathbf{D}_S}) +\epsilon$ and so
$$\E(S,\lambda_{\pi_1(S)})  \leq \E(S,\bigoplus^\infty\lambda_{{\pi_1(S)}}).$$
Taking the infimum over conformal classes $[S]\in \mathcal{T}_\Sigma$, this implies
$$\Area(\Sigma,\lambda_{{\pi_1(S)}}) \leq \Area(\Sigma,\bigoplus^\infty\lambda_{{\pi_1(S)}}).$$
 \end{proof}

Next, we consider a special unitary representation and a special embedding $\mathscr{P}$ of the hyperbolic plane into a Hilbert sphere, closely related to the maps $\mathcal{P}_c$ defined earlier in (\ref{poissonc}), and which played a crucial role in Besson-Courtois-Gallot's paper on the entropy inequality \cite{BCG95} (see \cite[Subsection 4.2]{Antoine23a}).
First, if $g_{\mathrm{hyp}}$ is as usual the unique conformal hyperbolic metric on $S$, we view $\pi_1(S)$ as a subgroup of the oriented  isometry group $\mathrm{PSL_2(\mathbb{R})}$ of the hyperbolic plane $(\tilde{S},g_{\mathrm{hyp}})$. Fix a basepoint $\mathbf{o}\in \tilde{S}$, let $\partial \tilde{S}$ be the boundary at infinity of $\tilde{S}$ with the standard uniform probability measure determined by $\mathbf{o}$, and let $\mathbb{S}_2(\partial \tilde{S})$ be the unit sphere in the complex $L^2$-space $L^2(\partial \tilde{S})$ endowed with the standard Riemannian metric $g_{L^2(\partial \tilde{S})}$.  For any $\theta\in \partial \tilde{S}$, the corresponding Busemann function is defined for any $x\in \tilde{S}$ as
$$B_\theta(x) := \lim_{t\to \infty} (\dist_{g_{\mathrm{hyp}}}(y,c(t)) -t)$$  
where $c:[0,\infty)$ is the half-geodesic starting at $\mathbf{o}$, and converging to $\theta$. 
The group ${\pi_1(S)}$ acts naturally on $\tilde{S}$ and $\partial \tilde{S}$ as a subgroup of $\mathrm{PSL_2(\mathbb{R})}$. This induces a unitary representation, called ``boundary representation'':
$$\underline{\rho}_B :{\pi_1(S)}\to \End(L^2(\partial \tilde{S}))$$
which is defined, for all $\gamma\in {\pi_1(S)}$, $f\in \mathbb{S}_2(\partial \tilde{S})$, by
\begin{equation}\label{rhoB}
\underline{\rho}_B(\gamma). f (\theta) = f(\gamma^{-1}(\theta)) e^{-\frac{1}{2}B_\theta(\gamma(\mathbf{o}))},
\end{equation}
see \cite[Lemme 2.2]{BCG95}.
Next, consider the following embedding
\begin{equation}\label{mathscrP}
\begin{split}
\mathscr{P} : (\tilde{S},g_{\mathrm{hyp}}) & \to (\mathbb{S}_2(\partial \tilde{S}), g_{L^2(\partial \tilde{S})})\\
\mathscr{P}(x) &:= \{\theta \mapsto e^{-\frac{1}{2}B_\theta(x)}\}.
\end{split}
\end{equation}
As explained in \cite[Section 2]{BCG95}, this map is ${\pi_1(S)}$-equivariant, namely it belongs to $\mathscr{H}_{S,\underline{\rho}_B}$.
The notions of weak equivalence and weak containment for representations are recalled in Definition \ref{weak equiv} in the Appendix.
We need the following facts about $\underline{\rho}_B$ and $\mathscr{P}$:
\begin{lem}\label{recallli}
\begin{enumerate}
\item The boundary representation $\underline{\rho}_B$ is irreducible, and weakly equivalent to the regular representation $\lambda_{\pi_1(S)}$.
\item
The map $\mathscr{P}$ is an isometric embedding of the rescaled hyperbolic plane $(\tilde{S},\frac{1}{8}g_{\mathrm{hyp}})$ into $\mathbb{S}_2(\partial \tilde{S})$. In particular, 
$$\E(\mathscr{P}\vert_{\mathbf{D}_S})=\Area(\Sigma,\lambda_{\pi_1(S)})= \frac{\pi}{4}(2\mathbf{g}+\mathbf{n}-2).$$
\end{enumerate}
\end{lem}
\begin{proof}
(1) The group ${\pi_1(S)}$ is a lattice of the simple Lie group $\mathrm{PSL}_2(\mathbb{R})$ of isometries of the hyperbolic plane. The representation $\underline{\rho}_B$ is the restriction of a unitary representation $\hat{\rho}_B: \mathrm{PSL}_2(\mathbb{R})\to \End(L^2(\partial \tilde{S}))$ defined by the generalization of formula (\ref{rhoB}) to the whole Lie group $\mathrm{PSL}_2(\mathbb{R})$. 
Actually,  $\hat{\rho}_B$ can be identified with the irreducible representation from the principal series called $\rho_0$ in the notation of \cite[Theorem 5.2.3]{Lubotzky94}, where the author reviews the classification of irreducible representations of $\mathrm{PSL}_2(\mathbb{R})$ containing $SO(2)/\{\pm 1\}$-invariant vectors. 
Indeed, the representation $\hat{\rho}_B$ is cyclic (\cite[Definition C.4.8]{BDLHV08}): for instance it can be easily checked that for any nonzero $f\in L^2(\partial \tilde{S})$, there is some $g\in\mathrm{PSL}_2(\mathbb{R}) $ such that $\langle \hat{\rho}_B(g) 1,f\rangle_{L^2(\partial \tilde{S})} \neq 0$. 
Moreover, the constant function $1$ in $L^2(\partial \tilde{S})$ is fixed by  $SO(2)/\{\pm 1\}$, and 
$$\varphi :g\mapsto \langle \hat{\rho}_B(g) 1,1\rangle_{L^2(\partial \tilde{S})}$$ is a spherical function in the sense of \cite[Definition 5.1.5]{Lubotzky94}.
As a function on the hyperbolic plane, by a direct computation using (\ref{rhoB}), it satisfies the equation $\Delta \varphi =-\frac{1}{4}\varphi$. 
Since this spherical function uniquely determines the cyclic representation $\hat{\rho}_B$ by the GNS construction \cite[Theorem C.4.10]{BDLHV08}, we can identify  $\hat{\rho}_B$ with $\rho_0$ from the principal series by  \cite[Definition 5.1.5]{Lubotzky94}, as claimed. In particular, $\hat{\rho}_B$ is not a discrete series, since those do not contain $SO(2)/\{\pm 1\}$-invariant vectors.
The irreducibility of $\underline{\rho}_B$ is then a consequence of \cite[Proposition 2.5 (ii)]{CS91}. 

By $C^*$-simplicity of surface groups such as $\pi_1(S)$
%\cite{Powers75} 
\cite[Definitions and examples on pages 2 and 3]{DLH07}, in order to show that $\underline{\rho}_B$ is weakly equivalent to  $\lambda_{\pi_1(S)}$, it suffices to show that the first is weakly contained in the second.
The fact that $\underline{\rho}_B$ is weakly contained in $\lambda_{\pi_1(S)}$ follows from the classical property that $\hat{\rho}_B$ is weakly contained in $\lambda_{\mathrm{PSL}_2(\mathbb{R})}$, or can be checked directly by ``approximating by hand'' functions in $L^2(\partial \tilde{S})$ by functions in $\ell^2({\pi_1(S)},\mathbb{C})$. 

(2) Checking that $\mathscr{P}$ is an isometry is a computation done in \cite[2.6 Retour aux exemples. a)]{BCG95} (translating their notations to our case, $n=2$, $h_0=1$, $g_{\sqrt{p_0}} = \mathscr{P}^* g_{L^2(\partial \tilde{S})}$).

\end{proof}

\begin{rem}
As we will see later, the relevance of the special embedding $\mathscr{P}$ is that it models the limit of equivariant harmonic maps into finite-dimensional spheres, which are equivariant with respect to strongly converging representations. What are the other possible limits of equivariant harmonic maps into finite-dimensional spheres?
\end{rem}

\vspace{1em}

\section{Intrinsic uniqueness of equivariantly minimizing surfaces} \label{section 2}

The goal of this subsection is to prove a general intrinsic uniqueness result for equivariant minimal surfaces in spheres, which seems to be the first rigidity result in this spirit.
This will play a key role: we will use it to identify the unique limit of the pullback metric by the harmonic maps of the main theorem.
%Average of minimal immersions in spheres}

%Consider a complete oriented finite area  hyperbolic surface $(\Sigma ,g_{\mathrm{hyp}})$ with fundamental group $\Gamma $, universal cover $\tilde{\Sigma} $ with fundamental domain $D_{\Sigma }$, 
Let $S$ be a punctured Riemann surface, with universal cover $\tilde{S}$, and fundamental domain $\mathbf{D}_S$. 
Let $g$ be any Riemannian metric on $S$ compatible with its conformal structure.
Let $\Sigma$ be the topological surface underlying $S$. 
Consider a unitary representation $\rho:{\pi_1(S)} \to\End(H')$ and its infinite direct sum (see \cite[Definition A.1.6]{BDLHV08}):
$$\bigoplus^\infty \rho :{\pi_1(S)} \to \End(\bigoplus^\infty H').$$ The main result of this subsection is that in the conformal class of $S $, there is at most one Riemannian metric $g_1$ for which there exists an equivariant,  branched, minimal, isometric immersion from $(\tilde{S} ,g_1)$ to the unit sphere $\mathbb{S}_{\bigoplus^\infty H'}$, which achieves $\Area(\Sigma,\bigoplus^\infty\rho)$.
We take the infinite direct sum of $\rho$, because the ``self-similarity'' of $\bigoplus^\infty\rho$ turns out to be an extremely helpful feature of the infinite dimensional setting.\footnote{The result probably fails for $\rho$ in place of $\bigoplus^\infty\rho$.}

\begin{thm}\label{conformal uniqueness}
Consider two functions $f_1,f_2: \tilde{S}\to [0,\infty)$ which are allowed to vanish at a discrete set of points. Suppose that for each $i=1,2$,
there is a branched, minimal, isometric immersion 
$$\varphi_i : (\tilde{S} ,f_i^2g) \to \mathbb{S}_{\bigoplus^\infty H'}$$
which is ${\pi_1(S)} $-equivariant with respect to $\bigoplus^\infty \rho$ and such that
$$\E(\varphi_i\vert_{\mathbf{D}_S}) = \E(S, \bigoplus^\infty \rho).$$
Then $f_1=f_2$.

\end{thm}

\begin{rem}
A consequence of Theorem \ref{conformal uniqueness} is that surfaces in spheres which are equivariant with respect to a ``self-similar'' representation and equivariantly area-minimizing  are intrinsically unique in each conformal class (if they exist). Are they also extrinsically unique?
\end{rem}

Theorem \ref{conformal uniqueness}
will directly follow from the next lemma.
Let $H$, $K$ be two Hilbert spaces, whose direct sum is denoted by $H\oplus K$ and is endowed with the direct sum of the scalar products of $H$, $K$. Denote their respective unit spheres by $\mathbb{S}_{H}$, $\mathbb{S}_K$, $\mathbb{S}_{H\oplus K}$.
Consider two conformal factors $f_1,f_2:\tilde{S}\to [0,\infty)$ which are allowed to vanish at a discrete set of points, and the two conformal metrics
$f_1^2g$, $f_2^2g.$ Suppose that there are two minimal isometric branched immersions 
$$\varphi_1: (\tilde{S},f_1^2g) \to \mathbb{S}_H,\quad  \varphi_2: (\tilde{S},f_2^2g) \to \mathbb{S}_K.$$
By the classical result of Takahashi   
(adapted to minimal surfaces in Hilbert spheres) \cite{Takahashi66}, and by conformal invariance of the Laplacian in dimension $2$, we have around points where $f_1\neq 0, f_2\neq 0$:
\begin{equation}\label{minimal equation}
\Delta_{g} \varphi_1 = -2 f_1^2 \varphi_1, \quad \Delta_{g} \varphi_2 = -2 f_2^2 \varphi_2.
\end{equation}
Define the ``average'' branched immersion of $\varphi_1,\varphi_2$ to be 
$$\varphi_3:\tilde{S}  \to \mathbb{S}_{H\oplus K},$$
$$\varphi_3(x):= \frac{1}{\sqrt{2}} (\varphi_1(x)\oplus \varphi_2(x)).$$
Note that $\varphi_3$ is automatically a conformal map, and that 
\begin{equation}\label{average e}
\E(\varphi_3\vert_{\mathbf{D}_S}) = \frac{1}{2}(\E(\varphi_1\vert_{\mathbf{D}_S}) +\E(\varphi_2\vert_{\mathbf{D}_S}) ).
\end{equation}

\begin{lem} \label{l}
With the notations above, if the average branched immersion $\varphi_3$ is harmonic, then $f_1=f_2$.
\end{lem}

\begin{proof}

If $\varphi_3$ is harmonic then
%, since it is always conformal, it is a branched minimal immersion. Thus,
$$\Delta_{g} \varphi_3 = a_3 \varphi_3$$
for some function $a_3:\tilde{S}\to \mathbb{R}$.
On the other hand, by definition of $\varphi_3$, and by (\ref{minimal equation}):
$$\Delta_{g} \varphi_3 = \frac{1}{\sqrt{2}} (\Delta_{g}\varphi_1\oplus \Delta_{g} \varphi_2)= -\sqrt{2}( f_1^2 \varphi_1\oplus f_2 ^2 \varphi_2).$$
In order for these two equations to be compatible, necessarily
${f_1}={f_2}$.

\end{proof}

\begin{proof}[Proof of Theorem \ref{conformal uniqueness}]
Set  $H=K=\bigoplus^\infty H'$. A crucial remark is that $\bigoplus^\infty \rho$ is self-similar, in the sense that the representation 
$(\bigoplus^\infty \rho \oplus \bigoplus^\infty\rho, H\oplus K)$
 is equivalent to $(\bigoplus^\infty\rho, H)$. Let us write this fact as 
 \begin{equation}\label{cruciial}
 \bigoplus^\infty \rho \oplus \bigoplus^\infty\rho = \bigoplus^\infty\rho.
\end{equation}
 
Under our assumptions, the average $\varphi_3$ of $\varphi_1,\varphi_2$ is a branched conformal immersion 
$$\varphi_3: \tilde{S} \to \mathbb{S}_{H\oplus K} = \mathbb{S}_{\bigoplus^\infty H'\oplus \bigoplus^\infty H'}.$$
Moreover, $\varphi_3$ is ${\pi_1(S)}$-equivariant with respect to $\bigoplus^\infty \rho \oplus \bigoplus^\infty\rho$, and by (\ref{average e}), its energy necessarily satisfies
$$\E(\varphi_3\vert_{\mathbf{D}_S}) = \E(S , \bigoplus^\infty \rho).$$
By (\ref{cruciial}), $\varphi_3$ can be viewed as a map taking values in $\mathbb{S}_{\bigoplus^\infty H'}$ and equivariant with respect to $\bigoplus^\infty \rho$. 
By definition of the spherical  energy, this implies that $\varphi_3$ is also harmonic.
Applying Lemma \ref{l}, we conclude the proof.
\end{proof}

We will need the following consequence of Theorem \ref{conformal uniqueness} for the regular representation, which has a higher-dimensional  analogue \cite[Corollary 4.3]{Antoine23a} (proved using completely different methods based on the barycenter map \cite{BCG95}). Let $\mathbf{g}$ and $\mathbf{n}$ be the genus and number of punctures of $S$, and let  $g_{\mathrm{hyp}}$ denote the unique complete, finite area, conformal hyperbolic metric on the Riemannian surface $S$ or its lift to $\tilde{S}$. The notions of weak equivalence and weak containment for representations, denoted with the symbols $\sim$ and $\prec$ respectively,  are recalled in Definition \ref{weak equiv} in the Appendix.

\begin{cor}\label{slack}
 Let $\rho_1:{\pi_1(S)} \to\End(H_1)$ be  a unitary representation with $\rho_1\sim \lambda_{{\pi_1(S)} }$ and let $(\mathbb{S}_{H_1},g_{H_1})$ be the unit sphere in $H_1$ with its standard Riemannian metric.
%Consider a conformal factor $f_1:\tilde{\Sigma}\to (0,\infty)$ which is allowed to vanish at a discrete set of points. Suppose that  there is a branched, minimal, isometric immersion  
%$$\varphi_1 : (\tilde{\Sigma} ,f_1^2g_{\mathrm{hyp}}) \to \mathbb{S}_{H_1}$$ which is ${\pi_1(S)} $-equivariant with respect to $\rho_1$ and such that $$E(\varphi_1\vert_{(D_{\Sigma },g_{\mathrm{hyp}})}) \leq \Area(\Sigma , \lambda_{{\pi_1(S)} }).$$ Then in fact  $$E(\varphi_1\vert_{(D_{\Sigma },g_{\mathrm{hyp}})}) = \Area(\Sigma , \lambda_{{\pi_1(S)} })$$ and  $$f_1=\frac{1}{\sqrt{8}}.$$
Consider a smooth ${\pi_1(S)} $-equivariant map
$$\varphi_1 : \tilde{S} \to (\mathbb{S}_{H_1},g_{H_1}).$$
Then 
$$\E(\varphi_1\vert_{\mathbf{D}_S}) \geq \Area(\Sigma, \lambda_{{\pi_1(S)} })=\frac{\pi}{4}(2\mathbf{g}+\mathbf{n}-2)=\frac{\pi}{4}|\chi(S)|$$
and equality holds if and only if
$$\varphi_1^* g_{H_1} = \frac{1}{8}g_{\mathrm{hyp}}.$$

\end{cor}

\begin{proof}
The value of $\Area(\Sigma, \lambda_{{\pi_1(S)} })$ is computed in Theorem \ref{spherearea}.
Consider the unitary representation 
$$\rho:=\rho_1\oplus \underline{\rho}_B: {\pi_1(S)}  \to \End(H_1\oplus L^2(\partial \tilde{S} ))$$
where $\underline{\rho}_B$ is the boundary representation defined in (\ref{rhoB}), and the corresponding Hilbert space
$$H':=H_1\oplus L^2(\partial \tilde{S} ).$$
Note that since $\rho_1 \prec \lambda_{\pi_1(S)}$ and $\underline{\rho}_B\prec \lambda_{\pi_1(S)}$,
$$\bigoplus^\infty \rho \prec \lambda_{{\pi_1(S)} }.$$
%Let $\mu_0$ be the element of Teichm\"{u}ller space $\mathscr{T}_\Sigma$ which is represented by $g_{\mathrm{hyp}}$. 
%Let $S$ be the Riemannian surface $(\Sigma,g_{\mathrm{hyp}})$.
Thanks to this and Corollary \ref{coree} (1) which we will prove in Section \ref{section:bounds}, we obtain
\begin{equation}\label{bydefe2}
 \E(S,\bigoplus^\infty \lambda_{{\pi_1(S)} }) \leq \E(S,\bigoplus^\infty \rho).
 \end{equation}
 Since $(\rho_1,H_1)$ is a subrepresentation of $(\rho,H')$ and thus  of $(\bigoplus^\infty\rho,\bigoplus^\infty H')$, the map $\varphi_1$ can be viewed as taking values into $\mathbb{S}_{\bigoplus^\infty H'}$ and  equivariant with respect to $\bigoplus^\infty \rho$.
 We then have by definition of energy:
\begin{equation}\label{bydefe}
 \E(S,\bigoplus^\infty \rho)\leq \E(\varphi_1\vert_{\mathbf{D}_S}).
\end{equation}
On the other hand,  from  Lemma \ref{=+} and Theorem \ref{spherearea},
%Theorem \ref{spherearea}, and the definitions of spherical energy and area, 
\begin{align}\label{chain =}
\E(S,\bigoplus^\infty \lambda_{{\pi_1(S)} }) & =\Area(\Sigma ,\bigoplus^\infty \lambda_{{\pi_1(S)} })  = \Area(\Sigma ,\lambda_{{\pi_1(S)} })  = \E(S, \lambda_{{\pi_1(S)} }) = \frac{\pi}{4}(2\mathbf{g}+\mathbf{n}-2).
\end{align}
Hence, by (\ref{bydefe}), (\ref{bydefe2}) and the line above, we already obtain the inequality in the statement of the corollary (modulo Corollary \ref{coree} (1) in Section \ref{section:bounds}).

Suppose that equality holds: $\E(\varphi_1\vert_{\mathbf{D}_S}) = \Area(\Sigma, \lambda_{{\pi_1(S)} })$.
Together with (\ref{bydefe}), (\ref{bydefe2}) and (\ref{chain =}), we deduce
 $$\E(\varphi_1\vert_{\mathbf{D}_S}) = \E(S ,\bigoplus^\infty \rho).$$
% and so $\varphi_1$ is a weakly conformal harmonic map.
So by Theorem \ref{min surf thm}, $\varphi_1$ is both harmonic and weakly conformal, namely a branched minimal immersion.
%Fix any Riemannian metric $g$ on $S$ compatible with its conformal structure, then we have a conformal factor $f_1:\tilde{S}\to [0,\infty)$, which vanishes only at a discrete set of points, such that  $$\varphi_1^*g_{H_1} = f_1^2 g$$ where $g$ is the lift of the metric on $S$. 
Let us compare the following two maps:
\begin{itemize}
\item the branched minimal immersion $\varphi_1$, equivariant with respect to $\rho_1$,
\item  the minimal embedding $\mathscr{P}$ defined in (\ref{mathscrP}), equivariant with respect to $\underline{\rho}_B$. 
\end{itemize}
Since $(\rho_1,H_1)$ and $(\underline{\rho}_B, L^2(\partial \tilde{S}))$ are subrepresentations of $(\bigoplus^\infty\rho,\bigoplus^\infty H')$, these two maps can both be viewed as maps into the sphere $\mathbb{S}_{\bigoplus^\infty H'}$ which are equivariant with respect to $\bigoplus^\infty \rho$.
Besides, by our assumption and by Lemma \ref{recallli} (2), both have energy $\Area(\Sigma, \lambda_{{\pi_1(S)} })=\frac{\pi}{4}(2\mathbf{g}+\mathbf{n}-2)$ on the fundamental domain $\mathbf{D}_S$. By Lemma \ref{recallli} (2), $\mathscr{P}$ is in fact an isometric embedding of the rescaled hyperbolic plane $(\tilde{S},\frac{1}{8}g_{\mathrm{hyp}})$ where $g_{\mathrm{hyp}}$ is the lift of the unique conformal hyperbolic metric on $S$.  
Applying the intrinsic uniqueness result, Theorem \ref{conformal uniqueness}, we deduce  that 
$$\varphi_1^*g_{H_1} =\frac{1}{8}g_{\mathrm{hyp}}$$ and the corollary is proved. 
\end{proof}

\vspace{1em}

\section{Bounds on the renormalized energy}\label{section:bounds}
This section extracts the geometric consequences of notions from representation theory like weak containment.
%Some of the  statements are quite technical but the main proof ideas are simple.  
%but they will be important in the proof of the main theorem.

\subsection{Weak containment and approximations of equivariant maps} \label{approx subsection}

In this subsection, we discuss how the notion of weak containment of representations  relates to geometric approximations of maps from compact regions of the universal cover of a surface. 

Let $S$ be a punctured Riemann surface. Let $\tilde{S}$ be its universal cover, 
$\mathbf{D}_S$ a Borel fundamental domain in $\tilde{S}$ which can be chosen to be a domain with piecewise smooth boundary. Let $\mathrm{Punc}_S=\{p_1,...,p_n\}$ be the possibly empty set of punctures of $S$.

If $\mathbf{n}>0$, recall from Subsection \ref{Eren} that there are disjoint punctured disks 
$$D(p_1,1),...,D(p_\mathbf{n},1) \subset S$$ 
such that for each $q\in \{1,...,\mathbf{n}\}$, there is a conformal diffeomorphism
$$P_{p_q} :D(p_q,1)\to \mathbb{D}^*,$$
and that $D(p_q,r)$ is defined as the set of points in $D(p_q,1)$ sent by $P_{p_q}$ to points $r$-close to the origin in $\mathbb{D}^*$. For any $\delta\in  (0,1]$, set
\begin{equation}\label{adelta}
A_\delta:= S\setminus \bigcup_{q=1}^\mathbf{n} D(p_q,\delta).
\end{equation}
If $\mathbf{n}=0$, define by convention $A_\delta=S$.
Let $\tilde{A}_\delta$ be the lift of $A_\delta$ to $\tilde{S}$.

If $H$ is a Hilbert space, we will always denote by $g_H$ its Hilbert metric and $\mathbb{S}_H$ its unit sphere.
Given a unitary representation $\nu:{\pi_1(S)} \to\End(K)$, as before  we denote by $\bigoplus^\infty \nu :{\pi_1(S)}\to \End(\bigoplus^\infty K)$  the infinite direct sum
% of a countable sequence of representations all equivalent to $\nu$, see
(\cite[Definition A.1.6]{BDLHV08}). For the definitions of weak containment and weak equivalence of representations, and the corresponding symbols $\prec$ and $\sim$, see Subsection \ref{weak equiv} in the Appendix.

\begin{prop} \label{proposition:approx}
Let $\delta\in (0,1]$. Consider two unitary representations $\rho:{\pi_1(S)} \to\End(H)$ and $\nu:{\pi_1(S)} \to\End(K)$. 
%Consider the infinite direct sum $\bigoplus^\infty \nu :{\pi_1(S)}\to \End(\bigoplus^\infty K)$. 
Suppose that $\rho \prec \nu$.  Let 
$$\varphi: \tilde{A}_\delta\to \mathbb{S}_H$$
be a smooth ${\pi_1(S)}$-equivariant map with respect to $\rho$.
Then for any $\epsilon'>0$, 
\begin{enumerate}
\item
there is a smooth ${\pi_1(S)}$-equivariant map with respect to $\bigoplus^\infty \nu$,
$$\varphi' : \tilde{A}_\delta \to \mathbb{S}_{\bigoplus^\infty K}$$
 such that
 %$$\|\varphi' - \varphi\|_{C^3(D_{\Sigma} \cap \tilde{A}_\delta)} < \epsilon'$$
$$\|(\varphi')^*g_{\bigoplus^\infty K} - \varphi^*g_H\|_{C^3(\mathbf{D}_S \cap \tilde{A}_\delta)} < \epsilon'$$
and 
$$|\E(\varphi'\vert_{\mathbf{D}_S\cap \tilde{A}_\delta}) - \E(\varphi \vert_{\mathbf{D}_S\cap \tilde{A}_\delta}) |<\epsilon',$$
\item if moreover $\rho$ is irreducible, then for any direct sum  decomposition 
$$K=\bigoplus_{m=1}^\infty K_m   ,\quad \nu=\bigoplus_{m=1}^\infty \nu_m,\quad \nu_m:{\pi_1(S)}\to \End(K_m),$$
 the map  $\varphi'$ can be chosen to take values in $\mathbb{S}_{K_{m_0}}$ for some $m_0$ depending on $\rho, \epsilon'$.
\end{enumerate}
\end{prop}

\begin{proof}

 It is possible to find a finite or infinite sequence of unit vectors $\{e_1,e_2,e_3,...\}$ such that the family of vectors
$ \{\rho(g)e_l\}_{l\geq 1, g\in {\pi_1(S)} }$ generates $H$, and for any $l_1\neq l_2$, the subspaces generated by 
$\{\rho(g)e_{l_1}\}_{g\in {\pi_1(S)} }$ and $\{\rho(g)e_{l_2}\}_{g\in {\pi_1(S)} }$ are orthogonal. 

 We fix an ordering on ${\pi_1(S)}$. 
 For each $l\geq 1$, let $\{u_{l,j}\}_{j\geq 1}$ be the unique orthonormal Hilbert basis of the Hilbert subspace generated by $\{\rho(g)e_{l}\}_{g\in {\pi_1(S)} }$ given by the Gram-Schmidt algorithm applied to $\{\rho(g)e_{l}\}_{g\in {\pi_1(S)} }$ using  the ordering on ${\pi_1(S)}$.

We can choose the fundamental domain $\mathbf{D}_S$ so that the compact domain $\overline{\mathbf{D}}_S \cap \tilde{A}_\delta$ has a piecewise smooth boundary, where $\overline{\mathbf{D}}_S$ is the closure of $\mathbf{D}_S$. 
The restriction of $\varphi$ to $\mathbf{D}_S\cap \tilde{A}_\delta$ can be written under the form:
$$\varphi:\mathbf{D}_S\cap \tilde{A}_\delta\to \mathbb{S}_H$$
$$x\in \mathbf{D}_S\cap \tilde{A}_\delta \mapsto  \sum_{l\geq 1}\sum_{j\geq 1} h_{l,j}(x) u_{l,j}$$
for some uniquely determined  functions $h_{l,j}: \mathbf{D}_\Sigma \cap \tilde{A}_\delta\mapsto \mathbb{R}$ which vary smoothly with respect to $x$.

Let $\epsilon>0$. By a standard approximation argument, we can find functions 
$$h^\epsilon_g: \mathbf{D}_S \cap \tilde{A}_\delta\mapsto\mathbb{R}$$  such that the new map 
$$\varphi^\epsilon:\mathbf{D}_S\cap \tilde{A}_\delta \to \mathbb{S}_H$$
$$x\in \mathbf{D}_S\cap \tilde{A}_\delta   \mapsto \sum_{l\geq 1}\sum_{j\geq 1} h^\epsilon_{l,j}(x) u_{l,j}$$
satisfies two properties:
\begin{enumerate}
\item $h^\epsilon_{l,j}(x)=0$ for any $x\in \mathbf{D}_S\cap \tilde{A}_\delta$, except for finitely many  $l\geq 1$ and $j\geq 1 $,
\item  $\|(\varphi^\epsilon)^*g_H -\varphi^*g_H\|_{C^3(\mathbf{D}_S\cap \tilde{A}_\delta)} \leq \epsilon.$
\end{enumerate}
In particular, for any $\epsilon'>0$, if $\epsilon$ is small enough, then the maps above satisfy:
$$\|(\varphi^\epsilon)^*g_H-\varphi^*g_H\|_{C^3(\mathbf{D}_S\cap \tilde{A}_\delta)} \leq \epsilon'/3,$$
$$|\E(\varphi^\epsilon\vert_{\mathbf{D}_S\cap \tilde{A}_\delta}) - \E(\varphi\vert_{\mathbf{D}_S\cap \tilde{A}_\delta})| <\epsilon'/3.$$

Since $\{u_{l,j}\}_{j\geq 1}$ was obtained by applying the Gram-Schmidt procedure to $\{\rho(g)e_{l}\}_{g\in {\pi_1(S)} }$, the functions $h^\epsilon_{l,j\geq1}$ determine uniquely a family of functions 
$$\{k^\epsilon_{l,g} :  \mathbf{D}_S \cap \tilde{A}_\delta \mapsto\mathbb{R} \}_{l\geq 1, g\in{\pi_1(S)}}$$ such that for any $x\in \mathbf{D}_S\cap \tilde{A}_\delta$,
$$\varphi^\epsilon(x)= \sum_{l\geq 1}\sum_{g\in {\pi_1(S)}} k^\epsilon_{l,g}(x) \rho(g)e_l$$
and for all but finitely many $l\geq 1$ and $g\in {\pi_1(S)}$, we have $k^\epsilon_{l,g}(x)=0$ for any $x$.

By our assumption that $\rho\prec \nu$, for any $l\geq 0$,  for all $\epsilon_1$ and every finite subset $Q\subset {\pi_1(S)} $, there exist $\eta_{l,1},...,\eta_{l,m}$ in $K$ such that for all $g\in Q$,
 $$\big| \langle \rho(g)e_l,e_l \rangle - \sum_{j=1}^m \langle \nu(g)\eta_{l,j},\eta_{l,j}\rangle \big| <\epsilon_1.$$
 We can reformulate that by saying the following: 
for any $l\geq 0$,  for all $\epsilon_1$ and every finite subset $Q\subset {\pi_1(S)} $, there exist $\eta_{l}$ in $\bigoplus^\infty K$ such that for all $g\in Q$,
 \begin{equation}\label{approx lambda}
 \big| \langle \rho(g)e_l,e_l \rangle - \langle \bigoplus^\infty \nu(g)\eta_l,\eta_l \rangle \big| <\epsilon_1.
 \end{equation}
Moreover, if $l_1\neq l_2$, then the subspace generated by $\eta_{l_1}$ is orthogonal to the one generated by  $\eta_{l_2}$ inside $\bigoplus^\infty K$. Inspecting the inequality for $g=e$ the unit element, we can assume that each $\eta_l$ has norm one.

We choose the finite subset $Q\subset {\pi_1(S)}$ so that it contains all elements of the form $g_1^{-1}g_2$ where $g_1,g_2\in {\pi_1(S)}$ are such that  $k^{\epsilon}_{l,g_1}(x_1) \neq 0, k^{\epsilon}_{l,g_2}(x_2) \neq 0$ for some $l\geq 1$ and some $x_1,x_2\in \mathbf{D}_S \cap \tilde{A}_\delta $.
Given $\epsilon_1>0$, let $\eta_l\in \bigoplus^\infty K$ be as above. 
Define the approximation map
$$\mathscr{A}^\epsilon:\mathbf{D}_S \cap \tilde{A}_\delta \to \mathbb{S}_{\bigoplus^\infty K}$$
$$x\in \mathbf{D}_S\cap \tilde{A}_\delta \mapsto \frac{1}{\|\sum_{l\geq 1} \sum_{g\in {\pi_1(S)} }k^\epsilon_{l,g}(x) \bigoplus^\infty\nu(g)\eta_l\|}\sum_{l\geq 1} \sum_{g\in {\pi_1(S)} }k^\epsilon_{l,g}(x) \bigoplus^\infty\nu(g)\eta_l.$$
%Moreover, $\mathscr{A}^\epsilon$ can be extended to a map from $\mathbf{D}_\Sigma$ to 
By (\ref{approx lambda}) and using that the unitary representations $\rho$ and $\bigoplus^\infty \nu$ preserve scalar products, this map is well-defined for $\epsilon_1$ small enough since then, the vector on the right side above is nowhere vanishing.
By (\ref{approx lambda}) again, for any $\epsilon'$,  if $\epsilon_1$ is chosen small enough, then $\mathscr{A}^\epsilon$ and $\varphi^\epsilon$ satisfy:
$$\|(\mathscr{A}^\epsilon)^*g_{\bigoplus^\infty K} - (\varphi^\epsilon)^*g_H\|_{C^3(\mathbf{D}_S\cap \tilde{A}_\delta)} \leq \epsilon'/3,$$
$$|\E(\mathscr{A}^\epsilon\vert_{\mathbf{D}_S\cap \tilde{A}_\delta}) - \E(\varphi^\epsilon\vert_{\mathbf{D}_S\cap \tilde{A}_\delta})| <\epsilon'/3.$$

It remains to construct from $\mathscr{A}^\epsilon$ a nearby map that is the restriction to $\mathbf{D}_S\cap \tilde{A}_\delta$ of a ${\pi_1(S)} $-equivariant map from $\tilde{A}_\delta $ to $\mathbb{S}_{\bigoplus^\infty K}$.
Let us call $s_1,...,s_j,...$ the finitely many sides of the polygon $\overline{\mathbf{D}}_S \cap \tilde{A}_\delta$.
For any side $s_j$ of $\overline{\mathbf{D}}_S\cap \tilde{A}_\delta$, 
if $g.s_j$ is another side $s_k$ of $\overline{\mathbf{D}}_S\cap \tilde{A}_\delta$ for some $g\in {\pi_1(S)} $, then by ${\pi_1(S)}$-equivariance of $\varphi$, the two maps 
$\varphi \circ g: s_j \to \mathbb{S}_H$ and $\rho(g) \circ  \varphi : s_j \to \mathbb{S}_H$ are equal.
It is then clear that  if $\epsilon$, $\epsilon_1$ are small enough, 
$\mathscr{A}^\epsilon \circ g: s_j \to \mathbb{S}_{\bigoplus^\infty K}$ and $\bigoplus^\infty\nu(g) \circ  \mathscr{A}^\epsilon: s_j \to \mathbb{S}_{\bigoplus^\infty K}$ are arbitrarily close in the $C^\infty$-topology. One can make them equal by a small perturbation of $\mathscr{A}^\epsilon$.
Repeating this for every side $s_j$, it is then quite elementary  to perturb the map $\mathscr{A}^\epsilon$ to get a new map 
$$\mathscr{A}^\epsilon_0: \overline{\mathbf{D}}_S \cap \tilde{A}_\delta \to \mathbb{S}_{\bigoplus^\infty K}$$
such that $\mathscr{A}^\epsilon_0$ is the restriction of a smooth ${\pi_1(S)} $-equivariant  map 
$$\varphi' :\tilde{A}_\delta \to \mathbb{S}_{\bigoplus^\infty K},$$ and 
$$\|(\varphi')^*g_{\bigoplus^\infty K} - (\mathscr{A}^\epsilon)^*g_{\bigoplus^\infty K}\|_{C^3(\mathbf{D}_S\cap \tilde{A}_\delta)} \leq \epsilon'/3,$$
$$|\E(\varphi' \vert_{\mathbf{D}_S\cap \tilde{A}_\delta}) - \E(\mathscr{A}^\epsilon\vert_{\mathbf{D}_S\cap \tilde{A}_\delta})| <\epsilon'/3.$$
%Next, we can arbitrarily extend $\varphi'$ to a ${\pi_1(S)}$-equivariant map from $\tilde{\Sigma}$ to $\mathbb{S}_{\bigoplus^\infty K}$ still denoted by $\varphi'$. The  particular way in which it is extended does not matter as long as it is equivariant, which is always possible since the sphere $\mathbb{S}_{\bigoplus^\infty K}$ is simply connected.
Combining the previous inequalities, we get
$$\|\varphi'- \varphi\|_{C^3(\mathbf{D}_S\cap \tilde{A}_\delta)} \leq \epsilon',$$
$$|\E(\varphi'\vert_{\mathbf{D}_S\cap \tilde{A}_\delta}) - \E(\varphi\vert_{\mathbf{D}_S\cap \tilde{A}_\delta})| <\epsilon'.$$ 
Property (1) of the proposition is proved.

To show (2), note that if $\rho$ is irreducible, then $H$ is generated by $\{\rho(g) e_1\}_{g\in {\pi_1(S)}}$ (i.e. there is only one nonzero vector $e_l$).
Moreover, since $\rho\prec \nu$, if $K=\bigoplus_{m=1}^\infty K_m  $, $\nu=\bigoplus_{m=1}^\infty \nu_m$, then (\ref{approx lambda}) holds with  $\eta_{1}$ belonging to  some $K_m$ instead of $\bigoplus^\infty K$.
Indeed, this follows from the proof of \cite[Proposition F.1.4]{BDLHV08} (in its proof, replace $\mathcal{F}$ by the set of normalized functions of positive type on ${\pi_1(S)}$ of the form $\langle \nu(.)\xi, \xi\rangle$ where $\xi \in \bigcup_{m=1}^\infty K_m$). 
Thus, the  approximations $\mathscr{A}^\epsilon$ and $\mathscr{A}^\epsilon_0$  can be chosen to take values in $\mathbb{S}_{K_m}$ instead of $\mathbb{S}_{\bigoplus^\infty K}$, and so $\varphi'$ can be chosen to take values in $\mathbb{S}_{K_m}$ instead of $\mathbb{S}_{\bigoplus^\infty K}$.
 %be in $\mathscr{H}_{\Sigma ,\nu_{m}}$ instead of $\mathscr{H}_{\Sigma ,\bigoplus^\infty \nu}$.  

\end{proof}

\subsection{Extending equivariant maps}
Below, we explain how to equivariantly extend a map from a compact subset of a universal cover (like those produced in the previous subsection) to the whole universal cover, in such a way that the renormalized  energy on a fundamental domain remains controlled\footnote{This subsection is technical and is only needed when working with punctured surfaces. In particular, if the reader is only interested in closed surfaces,  Subsection \ref{approx subsection} is enough to obtain the main consequences  of this subsection, Corollaries \ref{coro:basic} and \ref{coree}.}. A necessary and sufficient condition for a pair $(S,\rho)$ to have finite (non-renormalized) energy is also given in this subsection. 

Let $S$ be a punctured Riemann, with a fundamental domain $\mathbf{D}_S \subset \tilde{S}$ as before,
and let $\mathrm{Punc}_S$ be the set of punctures $p_1,...,p_\mathbf{n}$ of $S$. 
Let $\rho:{\pi_1(S)}\to \End(V)$ be a unitary representation (possibly of infinite dimension), $\mathbb{S}_V$ the unit sphere of the Hilbert space $V$. 
Let $\mathscr{H}_{S,\rho}$ be the space of smooth equivariant maps defined in (\ref{def H}). Given $p\in  \mathrm{Punc}_S$, the quantity $\blambda_{\rho}(p)$ was defined in (\ref{blambda}).

\textbf{Preparation and some notations:}
Let $\delta\in (0,1]$.
If $\mathbf{n}>0$, let  $A_\delta\subset S$ and $\tilde{A}_\delta\subset \tilde{S}$ be as in (\ref{adelta}). By convention, if $\mathbf{n}=0$,  $A_\delta=S$.
Let ${\gamma}_1,...,{\gamma}_\mathbf{n}$ be elements of the fundamental group which can be represented by oriented embedded loops around the respective punctures $p_1,...,p_\mathbf{n}$. Those elements are uniquely determined up to conjugacy. Those loops are freely homotopic to the boundary components of $A_\delta$. 

Let 
$$\Pi:{\tilde{S}}\to S$$
be the natural projection.  
Consider the complement $\tilde{A}_\delta^c$ of $\tilde{A}_\delta$ inside $\tilde{S}$. 
The region $\tilde{A}_\delta^c$ is a disjoint union of connected components, each of which is mapped to an embedded punctured disk inside $S$ under the projection $\Pi$.
For each $q\in \{1,...,\mathbf{n}\}$, recall that we fixed a conformal parameterization
$$P_{p_q}:D(p_q,1)\to \mathbb{D}^*$$ 
and defined $D(p_q,r)$ in (\ref{dpr}). 
Let us also fix a standard conformal diffeomorphism from $\mathbb{D}^*$ to the half-infinite cylinder
\begin{align*}
\mathcal{Q}:\quad \mathbb{D}^* &\to \{z\in\mathbb{C};\quad \im(z)\geq 0\}/z\mapsto z+2\pi\\
(\theta,r) & \mapsto (\theta,\log \frac{1}{r}).
\end{align*}

For each $q\in \{1,...,\mathbf{n}\}$, consider a connected  component $C_{\delta,q}$ of $\tilde{A}_\delta^c$ left invariant by the cyclic subgroup $\langle \gamma_q \rangle$ (this component is uniquely defined). 
By lifting the conformal parameterization $\mathcal{Q}\circ P_{p_q}$ to $C_{\delta,q}$, the closure of the connected component $C_{\delta,q}$ is conformal to 
$$\{z\in\mathbb{C};\quad \im(z)\geq \log\frac{1}{\delta}\}.$$
 In this conformal parametrization, $\gamma_q$ acts by translation $z\mapsto z+2\pi.$ Note that $$\tilde{A}_\delta^c = \bigcup_{q=1}^\mathbf{n} \Pi^{-1}(\Pi(C_{\delta,q})).$$

\begin{comment}
\begin{lem} \label{lem:extension}
Let $\delta\in (0,1]$. Consider a  smooth ${\pi_1(S)}$-equivariant map $\varphi_0:\tilde{A}_\delta\to \mathbb{S}_V$ 
%such that for each $q\in \{1,...,\mathbf{n}\}$,  $$\int_{\mathbf{D}_S\cap \partial C_{\delta,q}} |d\varphi_0|^2 \leq \pi.$$ 
Then there is a universal constant $C>0$ such that the following holds:
\begin{enumerate}
\item if $\blambda_{\rho}(p)=0$ for all $p\in  \mathrm{Punc}_S$, then there is a map $\hat{\varphi}\in \mathscr{H}_{S,\rho}$ such that
$$\E(\hat{\varphi}\vert_{\mathbf{D}_S}) \leq \E(\varphi_0\vert_{\mathbf{D}_S\cap \tilde{A}_\delta}) + C \big(\int_{\mathbf{D}_S\cap \partial  \tilde{A}_\delta} |d\varphi_0|^2+\sqrt{\mathbf{n}}\sqrt{\int_{\mathbf{D}_S\cap \partial \tilde{A}_\delta} |d\varphi_0|^2 }\big),$$
\item if $\dim_\mathbb{R} V<\infty$, then there is a map $\hat{\varphi}\in \mathscr{H}_{S,\rho}$ such that
$$\Eren(\hat{\varphi}\vert_{\mathbf{D}_S}) \leq \E(\varphi_0\vert_{\mathbf{D}_S\cap \tilde{A}_\delta}) + C \big(\int_{\mathbf{D}_S\cap \partial  \tilde{A}_\delta} |d\varphi_0|^2+\sqrt{\mathbf{n}}\sqrt{\int_{\mathbf{D}_S\cap \partial \tilde{A}_\delta} |d\varphi_0|^2 }\big).$$
\end{enumerate}
\end{lem}
\end{comment}

\begin{lem} \label{lem:extension}
Let $\delta\in (0,1]$. Consider a  smooth ${\pi_1(S)}$-equivariant map $\varphi_0:\tilde{A}_\delta\to \mathbb{S}_V$ with respect to $\rho$, such that for some $\bar{L}>0$, for each $q\in \{1,...,\mathbf{n}\}$,  $$\int_{\mathbf{D}_S\cap \partial C_{\delta,q}} |d\varphi_0|^2 \leq \bar{L}.$$ 
Then there is a constant $C=C(\bar{L})>0$ such that the following holds:
\begin{enumerate}
\item if $\blambda_{\rho}(p)=0$ for all $p\in  \mathrm{Punc}_S$, then there is a map $\hat{\varphi}\in \mathscr{H}_{S,\rho}$ such that
$$\E(\hat{\varphi}\vert_{\mathbf{D}_S}) \leq \E(\varphi_0\vert_{\mathbf{D}_S\cap \tilde{A}_\delta}) + C \sum_{q=1}^\mathbf{n}\sqrt{\int_{\mathbf{D}_S\cap \partial C_{\delta,q}} |d\varphi_0|^2 },$$
\item if $\dim_\mathbb{R} V<\infty$, then there is a map $\hat{\varphi}\in \mathscr{H}_{S,\rho}$ such that
$$\Eren(\hat{\varphi}\vert_{\mathbf{D}_S}) \leq \E(\varphi_0\vert_{\mathbf{D}_S\cap \tilde{A}_\delta}) + C \sum_{q=1}^\mathbf{n}\sqrt{\int_{\mathbf{D}_S\cap \partial C_{\delta,q}} |d\varphi_0|^2 }.$$
\end{enumerate}
\end{lem}

\begin{proof}

The lemma is only nontrivial when there are some punctures, so we will assume in this proof that  $\mathbf{n}>0$.
Let $\delta\in (0,1]$. Consider the given  smooth ${\pi_1(S)}$-equivariant map $\varphi_0:\tilde{A}_\delta\to \mathbb{S}_V.$
 Since $\mathbf{D}_S\cap \tilde{A}_\delta$ is compact, 
$\E((\varphi_0)\vert_{\mathbf{D}_S\cap \tilde{A}_\delta})<\infty.$
We want to extend this map to a map from the whole $\tilde{S}$ with (renormalized) energy controlled in terms of $\E(\varphi_0\vert_{\mathbf{D}_S\cap \tilde{A}_\delta})$ and the boundary integral $\int_{\mathbf{D}_S\cap \partial \tilde{A}_\delta} |d\varphi_0|^2$.

 \textbf{Case $(1)$:}
The assumption that $\blambda_{\rho}(p)=0$ for all $p\in  \mathrm{Punc}_S$ means the following:
 for each $j\geq 1$, there is a vector $v_{q,j}\in \mathbb{S}_V$ and some length-minimizing geodesic segment $\beta_j$ from $v_{q,j}$ to $\rho(\gamma_q)v_{q,j}$ with length converging to $0$ as $j\to \infty$. Consider  a small constant 
 $$\alpha\in (0,10^{-10})$$ which will be fixed later. After taking a subsequence in $j$, we can assume without loss of generality that
\begin{equation}\label{length bound}
\Length(\beta_j) \leq \alpha e^{-j}.
\end{equation}
Furthermore, by changing $v_{q,j}$ to its opposite $-v_{q,j}$, we can assume without loss of generality that for any $j$
\begin{equation}\label{jj}
\dist_{\mathbb{S}_V}(v_{q,j},v_{q,j+1}) \leq \frac{\pi}{2}.
\end{equation}
We can also assume (after taking the opposite of the $v_{q,j}$'s if necessary) that 
\begin{equation}\label{11}
\dist_{\mathbb{S}_V}(\varphi_0((0,\log\frac{1}{\delta})),v_{q,1}) \leq \frac{\pi}{2}.
\end{equation}

%We will first define a new map $\varphi_1$  on $\tilde{A}_\delta^c$, which we require to be equivariant, to coincide with $\varphi_0$ on $\partial \tilde{A}_\delta$, and such that the area of the image of $\mathbf{D}_S \cap \tilde{A}_\delta^c$ has controlled area. Then we will use a reparameterization trick  to obtain the desired extension $\varphi_2$ with finite energy on $\mathbf{D}_S$.

Fix momentarily $q\in  \{1,...,\mathbf{n}\}$ and consider a corresponding component $C_{\delta,q}$ of $\tilde{A}_\delta^c$ as before the lemma.
By a slight abuse of notations, we identify any map $\phi:C_{\delta,q}\to \mathbb{S}_V$ (resp. $\phi:\partial C_{\delta,q}\to \mathbb{S}_V$) with the corresponding map
$$\{z\in\mathbb{C};\quad \im(z)\geq \log\frac{1}{\delta}\} \to \mathbb{S}_V \quad (\text{resp.}\quad \partial \{z\in\mathbb{C};\quad \im(z)\geq \log\frac{1}{\delta}\} \to \mathbb{S}_V)$$
induced by the conformal parametrization fixed above.
We will first define a new map $\varphi_1$  on $\{z\in\mathbb{C};\quad \im(z)\geq \log\frac{1}{\delta}\} $, which we require to be equivariant, to coincide with $\varphi_0$ on the boundary $\partial  \{z\in\mathbb{C};\quad \im(z)\geq \log\frac{1}{\delta}\} $, and such that the area of the image of $[0,2\pi]\times[\log\frac{1}{\delta},\infty)$ has controlled area. Then, we will use a reparameterization trick  to obtain a map $\varphi_2$ with controlled energy on $[0,2\pi]\times[\log\frac{1}{\delta},\infty)$. 
%For future use, we will keep track of the bounds in terms of $\int_{\mathbf{D}_S\cap \partial \tilde{A}_\delta} |d\varphi_0|^2$ (note that this integral is on  the boundary of $\tilde{A}_\delta$).

We start by setting $\varphi_1=\varphi_0$ on the boundary $\partial \{z\in\mathbb{C};\quad \im(z)\geq \log\frac{1}{\delta}\} = \mathbb{R}\times \{\log\frac{1}{\delta}\}$. By an arbitrarily small perturbation of $\varphi_0$, we can assume (without loss of generality for the proof) that the restriction of $\varphi_1$ to $\partial \{z\in\mathbb{C};\quad \im(z)\geq 1\} = \mathbb{R}\times \{1\}$ is an immersion.
On the square $[0,2\pi]\times[\log\frac{1}{\delta},\log\frac{1}{\delta}+1] \subset \{z\in\mathbb{C};\quad \im(z)\geq \log\frac{1}{\delta}\}$, we extend $\varphi_1$  in such a way that 
\begin{itemize}
\item $\varphi_1(.,\log\frac{1}{\delta}+1)$ is a parameterization proportional to arclength of the curve $\varphi_1(.,\log\frac{1}{\delta})$, in other words
$$\forall t\in [0,2\pi],\quad |\frac{\partial}{\partial t}\varphi_1(t,\log\frac{1}{\delta}+1)| = \frac{1}{2\pi} \int_{0}^{2\pi}|\frac{\partial}{\partial t'}\varphi_1(t',\log\frac{1}{\delta})| dt',$$ 
\item when $s\in [\log\frac{1}{\delta},\log\frac{1}{\delta}+1]$, 
$\varphi_1(.,s)$
 is the natural linear interpolation between the two parameterizations $\varphi_1(.,\log\frac{1}{\delta})$ and $\varphi_1(.,\log\frac{1}{\delta}+1)$, namely if we set $\gamma(.):=\varphi_1(.,\log\frac{1}{\delta}+1)$ and view it as a bijection onto its image, 
 $$\forall (t,s)\in[0,2\pi]\times [\log\frac{1}{\delta},\log\frac{1}{\delta}+1], \quad \varphi_1(t,s) = \gamma\big[(s-\log\frac{1}{\delta})t+ (1-s+\log\frac{1}{\delta}) \gamma^{-1}(\varphi_1(t,\log\frac{1}{\delta}))\big].$$
 \end{itemize}
By those two properties, plus the assumption in the lemma, we estimate easily that  
\begin{equation}\label{9o0}
\int_{[0,2\pi]\times[\log\frac{1}{\delta},\log\frac{1}{\delta}+1]} |d\varphi_1|^2 \leq C_{(0)} \int_{\mathbf{D}_S\cap \partial C_{\delta,q}} |d\varphi_0|^2 \leq C_{(0)}\sqrt{\bar{L}}\sqrt{\int_{\mathbf{D}_S\cap \partial C_{\delta,q}} |d\varphi_0|^2}
\end{equation}
for some constant $C_{(0)}$ depending only on $\bar{L}$. Here the energy is computed using the standard Euclidean metric $g_{\mathrm{Eucl}}$ on $[0,2\pi]\times[\log\frac{1}{\delta},\log\frac{1}{\delta}+1]$.

Next set 
$$\varphi_1:[0,2\pi] \times \{\log\frac{1}{\delta}+2\} \to \mathbb{S}_V $$
to be the geodesic $\beta_2$ from $\varphi_1(0,\log\frac{1}{\delta}+2):=v_{q,2}$ to $\varphi_1(2\pi,\log\frac{1}{\delta}+2):=\rho(\gamma_q)v_{q,2}$, parametrized proportionally to arclength. 
From the previous paragraph, $\gamma(.):=\varphi_1(.,\log\frac{1}{\delta}+1):[0,2\pi]\to \mathbb{S}_V$ is a parameterization proportional to arclength of the curve  $\varphi_1(.,\log\frac{1}{\delta})$, and 
\begin{equation}\label{addo}
\Length(\gamma) = \int_{0}^{2\pi}|\frac{\partial}{\partial t'}\varphi_1(t',\log\frac{1}{\delta})| dt' \leq \sqrt{2\pi} \sqrt{\int_{\mathbf{D}_S\cap \partial C_{\delta,q}} |d\varphi_0|^2} \leq  \sqrt{2\pi}\sqrt{\bar{L}}.
\end{equation}
 We claim that for some constant $C_{(1)}$ depending only on $\bar{L}$, we can extend $\varphi_1$ to the region $[0,2\pi]\times [\log\frac{1}{\delta}+1,\log\frac{1}{\delta}+2]$ with a map whose differential satisfies at all $(x,y)\in [0,2\pi]\times [\log\frac{1}{\delta}+1,\log\frac{1}{\delta}+2]$:
\begin{equation}\label{1q2}
\begin{split}
|\frac{\partial \varphi_1}{\partial x}| & \leq C_{(1)}\big[(1-(y-\log\frac{1}{\delta}-1))\sqrt{\int_{\mathbf{D}_S\cap \partial C_{\delta,q}} |d\varphi_0|^2}+(y-\log\frac{1}{\delta}-1)\Length(\beta_2)\big] \\
|\frac{\partial \varphi_1}{\partial y}| & \leq C_{(1)}.
\end{split}
\end{equation}
If the length of $\gamma$ is smaller than say ${10^{-10}}$, such an extension can be cosntructed for instance by interpolating linearly between $\gamma$ and $\beta_2$ inside the Hilbert space $V$, using (\ref{11}) and radially projecting the interpolation to the sphere $\mathbb{S}_V$. If the length of $\gamma$ is between ${10^{-10}}$ and $\sqrt{2\pi}\sqrt{\bar{L}}$ (which is always an upper bound by (\ref{addo})), such an extension exists by an argument by contradiction and a compactness argument. This justifies the claim.

%Let $$\varphi_1:\{0\} \times [\log\frac{1}{\delta}+1,\log\frac{1}{\delta}+2] \to \mathbb{S}_V$$ 
%be a geodesic segment of length at most $\pi$ from $\varphi_1((0,\log\frac{1}{\delta}+1))$ to $\varphi_1((0,\log\frac{1}{\delta}+2)):=v_{q,2}$ parametrized %proportionally to arclength. This is possible by (\ref{11}). 
%Imposing $\langle \gamma_q\rangle$-equivariance, this determines $\varphi_1$ on $\{2\pi\}\times [\log\frac{1}{\delta}+1,\log\frac{1}{\delta}+2]$, so that now $%\varphi_1$ is well-defined on the boundary of $[0,2\pi]\times [\log\frac{1}{\delta}+1,\log\frac{1}{\delta}+2]$.
%We can extend $\varphi_1$ to the interior of $[0,2\pi]\times [\log\frac{1}{\delta}+1,\log\frac{1}{\delta}+2]$ with a map whose differential satisfies at all $(x,y)\in %[0,2\pi]\times [\log\frac{1}{\delta}+1,\log\frac{1}{\delta}+2]$:
%\begin{equation}\label{1q2}
%\begin{split}
%|\frac{\partial \varphi_1}{\partial x}| & \leq 100\big[(1-(y-\log\frac{1}{\delta}-1))\sqrt{\int_{\mathbf{D}_S\cap \partial C_{\delta,q}} |d\varphi_0|^2}+(y-\log\frac{1}%{\delta}-1)\Length(\beta_2)\big] \\
%|\frac{\partial \varphi_1}{\partial y}| & \leq 100.
%\end{split}
%\end{equation}

For each $j\geq 3$, set 
$$\varphi_1:[0,2\pi] \times \{\log\frac{1}{\delta}+j\} \to \mathbb{S}_V $$
to be the geodesic $\beta_j$ from $\varphi_1(0,\log\frac{1}{\delta}+j):=v_{q,j}$ to $\varphi_1(2\pi,\log\frac{1}{\delta}+j):=\rho(\gamma_q)v_{q,j}$, parametrized proportionally to arclength.
For any $j\geq 2$ and $t\in [0,2\pi]$, define
$$\varphi_1:\{t\}\times[\log\frac{1}{\delta}+{j},\log\frac{1}{\delta}+j+1] \to \mathbb{S}_V$$
to be a geodesic segment of length less than $\pi$ from $\varphi_1((t,\log\frac{1}{\delta}+j))$ to $\varphi_1((t,\log\frac{1}{\delta}+j+1))$ parametrized proportionally to arclength. By (\ref{length bound}) and (\ref{jj}),  $\varphi_1$ is well-defined and smooth on $[0,2\pi]\times [\log\frac{1}{\delta}+j,\log\frac{1}{\delta}+j+1]$. Moreover, it is elementary to check that for some constant $C_{(2)}$ depending only on $\bar{L}$,
at any point $(x,y)\in [0,2\pi]\times [\log\frac{1}{\delta}+j,\log\frac{1}{\delta}+j+1]$:
\begin{equation}\label{1q3}
\begin{split}
|\frac{\partial \varphi_1}{\partial x}| & \leq C_{(2)}\big[(1-(y-\log\frac{1}{\delta}-j))\Length(\beta_j)+(y-\log\frac{1}{\delta}-j)\Length(\beta_j+1)\big] \\
|\frac{\partial \varphi_1}{\partial y}| & \leq C_{(2)}.
\end{split}
\end{equation}

Now, the map $\varphi_1$ is extended to the whole band $[0,2\pi]\times [\log \frac{1}{\delta},\infty)$,  but it may not have controlled energy. On the other hand, it has bounded differential with respect to a metric which we construct now.
%Consider the natural coordinates $(x,y)$ for $[0,2\pi]\times[\log\frac{1}{\delta},\infty)$. 
We denote by $g_{\mathrm{Eucl}}$ the standard Euclidean metric on $[0,2\pi]\times[\log\frac{1}{\delta},\infty)$.
Let $g_1$ be the piecewise smooth Riemannian metric on $[0,2\pi]\times[\log\frac{1}{\delta},\infty)$ defined as follows: 
\begin{itemize}
\item on $[0,2\pi]\times[\log\frac{1}{\delta},\log\frac{1}{\delta}+1]$,
$$g_1=\big(\int_{\mathbf{D}_S\cap \partial C_{\delta,q}} |d\varphi_0|^2\big)g_{\mathrm{Eucl}},$$
\item on $[0,2\pi]\times[\log\frac{1}{\delta}+1,\log\frac{1}{\delta}+2]$,
$$g_1= \big[(1-(y-\log\frac{1}{\delta}-1))\sqrt{\int_{\mathbf{D}_S\cap \partial C_{\delta,q}} |d\varphi_0|^2}+(y-\log\frac{1}{\delta}-1)\Length(\beta_2)\big]^2dx^2 + dy^2,$$ 
\item for each $j\geq 2$, on $[0,2\pi]\times[\log\frac{1}{\delta}+j,\log\frac{1}{\delta}+j+1]$,
$$g_1= \big[(1-(y-\log\frac{1}{\delta}-j))\Length(\beta_j)+(y-\log\frac{1}{\delta}-j)\Length(\beta_{j+1})\big]^2dx^2 + dy^2.$$
\end{itemize}
Note that this metric  induces  a metric on the cylinder $\{z\in\mathbb{C};\quad \im(z)\geq \log\frac{1}{\delta}\}/{z\mapsto z+2\pi}$ which is rotationally symmetric. 
Here is the key property of this metric:  thanks to  (\ref{1q2}) and (\ref{1q3}),  we have with respect to $g_1$:
$$|d \varphi_1|_{g_1}\leq C_{(3)}$$ on $[0,2\pi]\times [\log\frac{1}{\delta}+1,\infty)$ for a  constant $C_{(3)}$ depending only on $\bar{L}$. 
In particular, 
\begin{align*}
\int_{[0,2\pi]\times[\log\frac{1}{\delta},\infty)} |d\varphi_1|_{g_1}^2 dv_{g_1} \leq &  C_{(0)} \sqrt{\bar{L}}\sqrt{\int_{\mathbf{D}_S\cap \partial C_{\delta,q}} |d\varphi_0|^2} + C_{(3)}^2 \Area([0,2\pi]\times[\log\frac{1}{\delta}+1,\infty),g_1)\\
\end{align*}
where the first term on the right is due to (\ref{9o0}) and conformal invariance of the energy. 
Because of the lengths bound for $\beta_j$ (\ref{length bound}) and using the definition of $g_1$, we have  for some constant $C_{(4)}$ depending only on $\bar{L}$,
$$\Area([0,2\pi]\times[\log\frac{1}{\delta}+1,\infty),g_1)\leq C_{(4)} (\sqrt{\int_{\mathbf{D}_S\cap \partial  C_{\delta,q}} |d\varphi_0|^2} +\alpha)$$
which implies with the previous inequality that
\begin{equation}\label{5t6}
\int_{[0,2\pi]\times[\log\frac{1}{\delta},\infty)} |d\varphi_1|_{g_1}^2 dv_{g_1} \leq   \big(C_{(0)} \sqrt{\bar{L}}+C_{(3)}^2 C_{(4)}\big)(\sqrt{\int_{\mathbf{D}_S\cap \partial  C_{\delta,q}} |d\varphi_0|^2} +\alpha).
\end{equation}

In order to get a map with controlled energy with respect to $g_{\mathrm{Eucl}}$ instead of $g_1$, we can use a reparameterization trick. Since $g_1$ is rotationally symmetric on the cylinder $\{z\in\mathbb{C};\quad \im(z)\geq \log\frac{1}{\delta}\}/{z\mapsto z+2\pi}$, this Riemannian cylinder is conformal to a flat cylinder $\mathcal{C}$ which is either bounded or half-infinite,
% $g_{\mathrm{Eucl}}$ 
via a conformal map which is itself equivariant by rotations of the cylinder. Since the metric $g_1$ is of the form $f(y)^2dx^2+dy^2$ on $[0,2\pi]\times[\log\frac{1}{\delta},\infty)$ for all $y\in [\log\frac{1}{\delta},\infty)$ large enough, where $|f(y)|\leq 1$ by the length bound (\ref{length bound}) for the $\beta_j$'s, the flat cylinder $\mathcal{C}$ is necessarily half-infinite and we can take it to be $\{z\in\mathbb{C};\quad \im(z)\geq \log\frac{1}{\delta}\}/{z\mapsto z+2\pi}$ with the standard metric $g_{\mathrm{Eucl}}$.
After lifting to the universal covers, we obtain an equivariant, conformal diffeomorphism
$$\mathcal{R} : (\mathbb{R}\times[\log\frac{1}{\delta},\infty),g_{\mathrm{Eucl}})\to (\mathbb{R}\times[\log\frac{1}{\delta},\infty),g_1)$$
which restricts to the identity map  on the boundary $\mathbb{R} \times \{\log\frac{1}{\delta}\}$.

Consider the new map 
$$\varphi_2:=\varphi_1\circ \mathcal{R}:[0,2\pi]\times[\log\frac{1}{\delta},\infty) \to \mathbb{S}_V.$$
By (\ref{5t6}) and conformal invariance of the energy, this map satisfies for some constant $C_{(5)} $ depending only on $\bar{L}$:
\begin{equation}\label{energy estim 1}
\int_{[0,2\pi]\times[\log\frac{1}{\delta},\infty)} |d\varphi_2|^2 \leq C_{(5)} (\sqrt{\int_{\mathbf{D}_S\cap \partial C_{\delta,q}} |d\varphi_0|^2} +\alpha)
\end{equation}
where the left-hand side is computed with respect to the Euclidean metric $g_{\mathrm{Eucl}}$.

The map $\varphi_2$ defines, after imposing equivariance, a map from $\Pi^{-1}(\Pi(C_{\delta,q}))$ to $\mathbb{S}_V$. 
Repeating the whole construction for each $q\in \{1,...,\mathbf{n}\}$, we obtain an  equivariant map 
$$\varphi_2: \tilde{A}^c_\delta \to \mathbb{S}_V$$
such that
\begin{equation}\label{energy estim 2}
\begin{split}
\int_{\mathbf{D}_S\cap \tilde{A}^c_\delta} |d\varphi_2|^2 & 
\leq C_{(5)} \sum_{q= 1}^{\mathbf{n}}\big(\sqrt{\int_{\mathbf{D}_S\cap \partial C_{\delta,q}} |d\varphi_0|^2 }+\alpha\big)\\
& \leq C_{(5)} \sum_{q= 1}^{\mathbf{n}}\sqrt{\int_{\mathbf{D}_S\cap \partial C_{\delta,q}} |d\varphi_0|^2 } +\mathbf{n}\alpha
\end{split}
\end{equation}
where the energy is computed with respect to the Euclidean metric $g_{\mathrm{Eucl}}$.
% and  $C_{(5)}$ is a universal constant.
We can glue the maps $\varphi_0$ and $\varphi_2$ along the boundary of $\tilde{A}_\delta$ to obtain our final equivariant map 
$\hat{\varphi}:\tilde{S} \to \mathbb{S}_V.$
It belongs to $\mathscr{H}_{S,\rho}$ after a smoothing, and its energy satisfies by (\ref{energy estim 2}):
%\begin{equation}\label{energy estim real}
$$
\E(\hat{\varphi}\vert_{\mathbf{D}_S}) \leq \E(\varphi_0\vert_{\mathbf{D}_S\cap \tilde{A}_\delta}) +  C_{(5)} \sum_{q= 1}^{\mathbf{n}}\sqrt{\int_{\mathbf{D}_S\cap \partial C_{\delta,q}} |d\varphi_0|^2 } +\mathbf{n}\alpha$$
%\end{equation}
for a constant $C_{(5)}$ depending only on $\bar{L}$.  The positive constant $\alpha>0$ could have been taken arbitrarily small. Either $ \sum_{q= 1}^{\mathbf{n}}\sqrt{\int_{\mathbf{D}_S\cap \partial C_{\delta,q}} |d\varphi_0|^2 }>0$ and we can chose $\mathbf{n}\alpha$ to be smaller. Or $ \sum_{q= 1}^{\mathbf{n}}\sqrt{\int_{\mathbf{D}_S\cap \partial C_{\delta,q}} |d\varphi_0|^2 }=0$, which means that $\varphi_0$ is constant on connected components of $\partial \tilde{A}_\delta$, so we could have taken the trivial extension instead of $\hat{\varphi}$ above. The trivial extension has same energy as $\varphi_0$.
In both cases, we conclude Case (1).
\vspace{1em}

 \textbf{Case $(2)$:}
 Let $\delta\in (0,1]$. We use the same notations as above. We also introduce the following notation. 
Given  $q\in\{1,...,\mathbf{n}\}$, we have identified  $C_{\delta,q}$ with $\{z\in\mathbb{C};\quad \im(z)\geq \log\frac{1}{\delta}\}.$ 
Set $$\mathbf{r}:\{z\in\mathbb{C};\quad \im(z)\geq \log\frac{1}{\delta}\}\to [0,\infty)$$
$$\mathbf{r}: z\mapsto \exp(-\im(z)).$$
%\dist_{\mathrm{Eucl}}(\mathcal{Q} \circ \Pi(x),O)$. 
In particular, 
$\mathbf{r}(z) =\delta$ for $z \in \partial \{z\in\mathbb{C};\quad \im(z)\geq \log\frac{1}{\delta}\}$.
% and more generally, $$\mathbf{r}(z) = \exp(-\im(z)).$$
%$\mathbf{r}(z)\to 0$ as $\im(z)\to \infty$, 
%the level sets $\{\mathbf{r}=r\}$ are horizontal lines  of the form $\{\im(z)=t\}$.

Fix $q\in\{1,...,\mathbf{n}\}$.
By the assumption that $\dim_\mathbb{R}V<\infty$ and by compactness, there is a point $v_{q}\in \mathbb{S}_V$ and some length-minimizing geodesic segment parameterized by arclength
 $$\beta:[0,2\pi]\to \mathbb{S}_V$$
 such that  $\beta(0)=v_{q}$ and $\beta(1) = \rho(\gamma_q)v_{q}$, and 
 $$\Length(\beta)=\blambda_{\rho}(q).$$
By changing $v_{q}$ to its opposite $-v_{q}$, we can assume without loss of generality that 
\begin{equation}\label{11bis}
\dist_{\mathbb{S}_V}(\varphi_0((0,\log\frac{1}{\delta})),v_{q}) \leq \frac{\pi}{2}.
\end{equation}

Next, we want to define an extension of $\varphi_0$ with controlled energy. We first extend $\varphi_0$ to a map $\varphi_0'$ on $[0,2\pi]\times[\log\frac{1}{\delta},\log\frac{1}{\delta}+2]$ such that 
$$\varphi_0'(.,\log\frac{1}{\delta}+2)=\beta.$$ 
%and 
%$$\int_{[0,2\pi]\times[\log\frac{1}{\delta},\log\frac{1}{\delta}+2]} |d\varphi_0'|^2\leq C_{(6)} \int_{\mathbf{D}_S\cap \partial C_{\delta,q}} |d\varphi_0|^2$$
%where $C_{(6)}$ is a universal constant. 
Similar to the proof of Case (1), we choose $\varphi_0'$ to have controlled energy and area on $[0,2\pi]\times[\log\frac{1}{\delta},\log\frac{1}{\delta}+2]$, and the reparameterization trick enables us to construct another extension $\varphi_1:[0,2\pi]\times[\log\frac{1}{\delta},h]\to \mathbb{S}_V$ for some $h>\log\frac{1}{\delta}$, 
such that $\varphi_1(.,h) = \varphi_0'(.,\log\frac{1}{\delta}+2)$ and for a constant $C_{(6)}$ depending only on $\bar{L}$: 
\begin{equation}\label{edc}
\begin{split}
\int_{[0,2\pi]\times[\log\frac{1}{\delta},h]} |d\varphi_1|^2 
&\leq C_{(6)} \big(\sqrt{\int_{\mathbf{D}_S\cap \partial C_{\delta,q}} |d\varphi_0|^2} +\blambda_{\rho}(q)\big)\\
&\leq C_{(6)} \big(1+\sqrt{2\pi})\sqrt{\int_{\mathbf{D}_S\cap \partial C_{\delta,q}} |d\varphi_0|^2}
\end{split}
\end{equation}
where we used that $\blambda_{\rho}(q) = \Length(\beta)\leq \sqrt{2\pi} \sqrt{\int_{\mathbf{D}_S\cap \partial C_{\delta,q}} |d\varphi_0|^2}$ by Cauchy-Schwarz. 
Then we extend $\varphi_1$  to $[0,2\pi]\times [\log\frac{1}{\delta},\infty)$ by setting for all $(t,s)\in [0,2\pi]\times [h,\infty)$,
$$\varphi_1(t,s) = \beta(t).$$
%It is well-known that with respect to the fixed Riemannian metric $g$ on $S$,
%our choice of conformal parametrization $$\{z\in\mathbb{C}; \Im(z)\geq h_q\} \subset \{z\in\mathbb{C};\quad \im(z)\geq 1\} \to \mathbb{D}^*,$$
The ``renormalized'' energy of $\varphi_1$ on $[0,1]\times [h,\infty)$ satisfies 
%viewed as a map from a $g$-neighborhood of $p_q\in \bar{S}$, satisfies
$$\lim_{r\to 0} \big[\big( \frac{1}{2}\int_{[0,2\pi]\times [h,\infty) \setminus \mathbf{r}^{-1}((0,r))} |d\varphi_1|^2 \big)-\frac{\blambda_{\rho}(p_q)^2}{4\pi}\log\frac{1}{r} \big]= -\frac{\blambda_{\rho}(p_q)^2}{4\pi} h\leq 0.$$
%\log \mathbf{r}((0,h)) \leq 0.$$
The ``renormalized'' energy of $\varphi_1$ on $[0,2\pi]\times [\log\frac{1}{\delta},\infty)$  is thus at most $\int_{[0,2\pi]\times[\log\frac{1}{\delta},h]} |d\varphi_1|^2 $, which is controlled by (\ref{edc}). 
Repeating this construction for every cusp, defining $\varphi_1$ on the whole universal cover $\tilde{S}$ by equivariance,  and adding the previous estimates together, we conclude Case (2).

\end{proof}

The following corollary contains a practical characterization of pairs $(S,\rho)$ with finite energy, and gives an a priori bound on $\Eren(S,\rho)$. 
\begin{cor}\label{coro:basic}
\begin{enumerate}
\item $\E(S,\rho)<\infty$
if and only if $\blambda_{\rho}(p)=0$ for all $p\in  \mathrm{Punc}_S$. 
%\item If  $\dim_\mathbb{R} V<\infty$ and $\E(S, \rho)<\infty$ then $\Eren(S, \rho)=\E(S, \rho) $.  
\item If  $\dim_\mathbb{R} V<\infty$, the renormalized energy $\Eren(S,\rho)$ is  finite and bounded from above by a constant $C_S$ depending only on $S$. 
\end{enumerate}
\end{cor}
\begin{proof}
Fix  a triangulation of $S$ and any $\delta\in (0,1]$. By basic topological arguments, there is always a smooth ${\pi_1(S)}$-equivariant map $\varphi_0:\tilde{A}_\delta\to \mathbb{S}_V$. Moreover, since the diameter of $\mathbb{S}_V$ is $\pi$, by straightening the map on the 1-skeleton of the triangulation then extending the map by hand to the 2-skeleton, $\varphi_0$ can be chosen so that $\E(\varphi_0\vert_{\mathbf{D}_S}) \leq C'_S$ and $\int_{\mathbf{D}_S\cap \partial {C_{\delta,q}}_\delta} |d\varphi_0|^2 < C'_S$ for all $q\in \{1,...,\mathbf{n}\}$, for some $C'_S>0$ independent of $\rho$  or the dimension of $V$.
Here $C_{\delta,q}$ is defined before Lemma \ref{lem:extension}.

(1):  The implication $\forall p\in  \mathrm{Punc}_S, \blambda_{\rho}(p)=0 \Rightarrow \E(S,\rho)<\infty$ is an easy consequence of Lemma \ref{lem:extension} (1) and the existence of $\varphi_0$ explained above.
Next, the direction $\E(S,\rho)<\infty \Rightarrow  \forall p\in  \mathrm{Punc}_S, \blambda_{\rho}(p)=0$  follows from Fubini's theorem and was already showed in the proof of Lemma \ref{lem:one}. 

(2): This follows from Lemma \ref{lem:extension} (2) and the existence of $\varphi_0$ explained above.

\end{proof}

%\begin{lem}\label{}
%For any $\epsilon'_1$, there exists $\epsilon'_2$  independent of $\rho$ and $\dim_\mathbb{R} V$, so that the %following holds. 
%Consider a harmonic representative $\psi$  of $(S,\rho)$. 
%Take $\delta<\epsilon'_2$.
%Then for some $\delta'\in [\delta/2,\delta]$,
%$$\int_{ \mathbf{D}_S\cap \partial \tilde{A}_{\delta'}} |d\psi|^2 \leq \epsilon'_1.$$
%\end{lem}

\begin{cor} \label{coree}
\begin{enumerate}
\item If  $\rho:{\pi_1(S)} \to \End(V)$, $\nu :{\pi_1(S)}\to \End(K)$ are unitary representations such that $\rho \prec \nu$, then
%and $\E(S, \rho)<0$, then
$$ \E(S,\bigoplus^\infty \nu) \leq \E(S, \rho).$$
\item If moreover $\rho$ is irreducible, then for any decomposition $\nu=\bigoplus_{m=1}^\infty \nu_m$,
%and each $(S,\nu_m)$ satisfies $\mathcal{F}_{\epsilon_m}$ with $\epsilon_m\to 0$ as $m\to \infty$, NO NEED?
and for any $\epsilon''>0$, there is $m_0\geq 1$ such that
$$\Eren(S,\nu_{m_0}) \leq \E(S, \rho)+\epsilon''.$$
\end{enumerate}
\end{cor}

\begin{proof}
(1): Suppose $\E(S, \rho)<\infty$, otherwise there is nothing to do. Consider a map $\varphi_0\in \mathscr{H}_{S,\rho}$ with finite energy on the fundamental domain $\mathbf{D}_S$.
 By Fubini's theorem as in the proof of Lemma \ref{lem:one}, for any $\epsilon>0$, there is an arbitrarily small $\delta>0$ such that for any $q\in \{1,...,\mathbf{n}\}$,
 $$\int_{\mathbf{D}_S\cap \partial C_{\delta,q}} |d\varphi_0|^2 \leq \epsilon$$
 where $C_{\delta,q}$ is defined before Lemma \ref{lem:extension}.
 By Proposition \ref{proposition:approx} (1), we get a smooth map $\varphi_0': \tilde{A}_\delta \to \mathbb{S}_{\bigoplus^\infty K}$, which is $\pi_1(S)$-equivariant with respect to $\bigoplus^\infty \nu$, such that 
 $$|\E(\varphi_0'\vert_{\mathbf{D}_S})-\E(\varphi_0\vert_{\mathbf{D}_S})|<\epsilon \quad \text{and}\quad \int_{\mathbf{D}_S\cap \partial C_{\delta,q}} |d\varphi_0'|^2 \leq 2\epsilon \quad \text{for any $q\in \{1,...,\mathbf{n}\}$}.$$
 Applying Lemma \ref{lem:extension} (1) and letting $\epsilon\to 0$, we obtain $ \E(S,\bigoplus^\infty \nu) \leq \E(S, \rho)$.

(2): Suppose moreover that $\rho$ is irreducible, and $\nu=\bigoplus_{m=1}^\infty \nu_m$ with $\nu_m:\pi_1(S)\to \End(K_m)$. 
Using  Proposition \ref{proposition:approx} (2), for any $\epsilon>0$, 
we get an $m_0=m_0(\epsilon)\geq 1$ and a smooth map $\varphi_0'': \tilde{A}_\delta \to \mathbb{S}_{K_{m_0}}$, which is $\pi_1(S)$-equivariant with respect to $\nu_{m_0}$, such that 
 $$|\E(\varphi_0''\vert_{\mathbf{D}_S})-\E(\varphi_0\vert_{\mathbf{D}_S})|<\epsilon \quad \text{and}\quad \int_{\mathbf{D}_S\cap \partial C_{\delta,q}} |d\varphi_0''|^2 \leq 2\epsilon \quad \text{for any $q\in \{1,...,\mathbf{n}\}$}.$$
 Applying Lemma \ref{lem:extension} (2) and choosing $\epsilon$ small enough, we have $\Eren(S,\nu_{m_0}) \leq \E(S, \rho)+\epsilon''$ for any $\epsilon''>0$ fixed. 
\end{proof}

%\textcolor{orange}{Do i really need the next lemma?}
%\begin{lem}\label{coro:ext prop}
%$(S,\bigoplus^\infty \lambda_{\pi_1(S)})$ satisfies $\mathcal{F}_{0}$.
%\end{lem}
%\begin{proof}
%For any element ${\gamma}\in {\pi_1(S)}$, consider the function $v$ in $\mathbb{S}_{\ell^2({\pi_1(S)},\mathbb{C})}$ equal to $\frac{1}{\sqrt{N}}$ on $\{\gamma,...,\gamma^N\}$ and $0$ elsewhere for some $N>0$. Then the element $V:=v\oplus 0\oplus 0...\in \mathbb{S}_{\bigoplus^\infty \ell^2({\pi_1(S)},\mathbb{C})}$ satisfies
%$$\|V-\bigoplus^\infty \lambda_{{\pi_1(S)} }(\gamma) V\| = \sqrt{\frac{2}{N}} \xrightarrow[N\to \infty]{}0.$$
%\end{proof}

\vspace{1em}

\section{From strong convergence to convergence of harmonic maps} \label{section 4}

We prove the main result, Theorem \ref{thm:main}, in this section. Let us recall our assumptions and notations. Let $S$ be a punctured Riemann surface, with genus $\mathbf{g}$ and $\mathbf{n}\geq0$ punctures. Suppose that $S$ has negative Euler characteristic $\chi(S)<0$.
Let 
$$\rho_j: \pi_1(S)  \to U(N_j)$$
be a sequence of unitary representations.

\subsection{Convergence of the renormalized energy}

In this subsection, we show that the first part of Theorem \ref{thm:main}: strong convergence of a sequence of representations (see Subsection \ref{strong cv} in the Appendix) implies convergence of the renormalized energy.

\begin{thm} \label{energy continuity}

If $\rho_j$ strongly converges, then
$$\lim_{j\to \infty}\Eren(S,\rho_j) = \frac{\pi}{4}(2\mathbf{g}+\mathbf{n}-2) = \frac{\pi}{4}|\chi(S)|.$$
\end{thm}

Note that $\frac{\pi}{4}(2\mathbf{g}+\mathbf{n}-2)$ is exactly $\frac{1}{8}\Area(S,g_{\mathrm{hyp}})$, where  $g_{\mathrm{hyp}}$ is the unique complete, finite area, hyperbolic metric on $S$ compatible with its conformal structure.
The theorem above will follow from Proposition \ref{e<} and Proposition \ref{>e} below.

\begin{prop}\label{e<}
Under the previous assumptions, 
$$\limsup_{j\to \infty}\Eren(S,\rho_j) \leq \frac{\pi}{4}(2\mathbf{g}+\mathbf{n}-2).$$
\end{prop}

\begin{proof}

Arguing towards a contradiction, let us assume that for some $\epsilon''>0$, there is a sequence  $\{j_m\}_{m\geq 1}$ converging to $\infty$ such that 
\begin{equation}\label{contrid}
\inf_{m\geq 1} \Eren(S,\rho_{j_m}) >\frac{\pi}{4}(2\mathbf{g}+\mathbf{n}-2)+\epsilon''.
\end{equation}
Set $$\nu:=\bigoplus_{m\geq 1} \rho_{j_m}.$$
Let $\mathscr{P} :\tilde{S}\to \mathbb{S}_2(\partial \tilde{S})$ and $\underline{\rho}_B$ be, respectively,  the map in $\mathscr{H}_{S,\underline{\rho}_B}$ and boundary representation defined in (\ref{mathscrP}) and (\ref{rhoB}). 
Note that since from Lemma \ref{recallli} (1),  $\underline{\rho}_B \sim \lambda_{\pi_1(S)}$,  we get
$$\|\underline{\rho}_B(f)\|= \|\lambda_{\pi_1(S)}(f)\|\quad \text{for all $f\in\mathbb{C}[{\pi_1(S)}]$}$$
by Theorem \ref{equivalent}.
Since $\rho_j$ strongly converges by assumption, 
$$\lim_{j\to \infty} \|\rho_j(f)\|=\|\lambda_{\pi_1(S)}(f)\|    \quad \text{for all $f\in\mathbb{C}[{\pi_1(S)}]$}.$$
Hence, we clearly have
$$\|\underline{\rho}_B(f)\|\leq \|\nu(f)\|    \quad \text{for all $f\in\mathbb{C}[{\pi_1(S)}]$}$$
which means according to Theorem \ref{equivalent} that $\underline{\rho}_B\prec \nu$. 
Next, we know from Lemma \ref{recallli} that $\underline{\rho}_B$ is irreducible and $\E(\mathscr{P}\vert_{\mathbf{D}_S})  = \frac{\pi}{4}(2\mathbf{g}+\mathbf{n}-2)$.
%, so that 
%$$\E(S,\underline{\rho}_B)\leq\frac{\pi}{4}(2\mathbf{g}+\mathbf{n}-2).$$
We can then apply Corollary \ref{coree} (2) to $\rho:=\underline{\rho}_B$, $\nu:=\bigoplus_{m\geq 1} \rho_{j_m}$, and conclude that for some $m_0\geq 1$,
$$\Eren(S,\rho_{j_{m_0}}) \leq \E(S,\underline{\rho}_B) +\epsilon''\leq \E(\mathscr{P}\vert_{\mathbf{D}_S})  +\epsilon'' = \frac{\pi}{4}(2\mathbf{g}+\mathbf{n}-2) +\epsilon''.$$
That contradicts (\ref{contrid}), as wanted.
\end{proof}

We observe an easy consequence of the strong convergence of $\rho_j$ (see (\ref{blambda}) for the definition of $\blambda_{\rho_j}(p)$):
\begin{lem} \label{blam0}
If $\rho_j$ strongly converges, for any puncture $p\in \mathrm{Punc}_S$, we have 
\begin{equation*} 
\lim_{j\to \infty}\blambda_{\rho_j}(p)=0.
\end{equation*}
\end{lem}
\begin{proof}
Indeed, consider any element $g\in \pi_1(S)$ in the conjugacy class determined by small embedded oriented loops around $p$ in $S$.  For any $\varepsilon>0$, there is a function $F_\varepsilon\in \ell^2(\pi_1(S))$ of unit $L^2$-norm  such that $\|\lambda_{\pi_1(S)}(g)F_\varepsilon-F_\varepsilon\|\leq \varepsilon$ where $\lambda_{\pi_1(S)}$ is the regular representation (e.g. take $F_\varepsilon$ to be constant on its support and  supported on $\{g,g^2,....,g^{N_\varepsilon}\}$ for some large integer $N_\varepsilon$). This implies in terms of operator norm: $\|\lambda_{\pi_1(S)}(g)+\lambda_{\pi_1(S)}(e)\|=2$ where $e\in \pi_1(S)$ is the identity element. By strong convergence, we have
$\lim_{j\to \infty} \|\rho_j(g)+\rho_j(e)\|=2$. But this means by Cauchy-Schwarz that for any $\varepsilon>0$, for all large $j$, there are unit norm vectors $v_j$ such that $\|\rho_j(g)v_j-v_j\|\leq \varepsilon$, which is exactly the lemma.

\end{proof}

The next proposition treats the opposite inequality. Its proof  is the key place where the concept of strong convergence for representations is combined with harmonic maps. There, we construct a limit of harmonic maps, equivariant with respect to a limit representation, which we show is weakly equivalent to the regular representation.

\begin{prop}\label{>e}
Under the above assumptions, 
$$\liminf_{j\to \infty}\Eren(S,\rho_j) \geq  \frac{\pi}{4}(2\mathbf{g}+\mathbf{n}-2).$$
\end{prop}

\begin{proof}
For each $j\geq1$, let $\psi_j:\tilde{S}\to \mathbb{S}^{2N_j-1}$ be a harmonic representative of $(S,\rho_j)$: it is a harmonic map in $\mathscr{H}_{S,\rho_j}$ given by Theorem \ref{realize}, so that
 $$\Eren(\psi_j\vert_{\mathbf{D}_S}) = \Eren(S,\rho_j).$$
Due to Lemma \ref{monotonicity}, on every compact subset of $\mathbf{D}_S$, the energy of $\psi_j$ is uniformly bounded since $\Eren(S,\rho_j)$ is uniformly bounded thanks to Corollary \ref{coro:basic} (2).

By standard harmonic map theory \cite{SU82}, after picking a subsequence and using the ${\pi_1(S)}$-equivariance of the map $\psi_j$, a priori there are finitely many points $y_1,...,y_k\in \mathbf{D}_S$ where the energy might concentrate as $j\to \infty$. Using the $\epsilon$-regularity theorem, Theorem \ref{epsilon reg}, we can replace $\psi_j$ in a small neighborhood of the ${\pi_1(S)}$-orbit of $\bigcup_{j=1}^k y_j$ equivariantly by a map with renormalized energy strictly smaller than $\psi_j$ for $j$ large. But since $\psi_j$ is energy minimizing among all smooth equivariant maps, we have a contradiction (this is the usual replacement argument appearing in \cite[proof of Theorem 5.1]{SU82} for instance). 
 Thus, the energy of $\psi_j$ does not concentrate around any point as $j\to \infty$. 
Next, by the $\epsilon$-regularity theorem again, Theorem \ref{epsilon reg}, 
 on each compact subset of $\tilde{S}$, the maps $\psi_j$ are uniformly bounded in the $C^m$-topology for any $m\geq0$. 
 
 In order to construct a limit, let us view all the spaces $\mathbb{C}^{N_j}$ as embedded as linear subspaces of one common Hilbert space $H$ with unit sphere $\mathbb{S}_H$. 
 Choose a dense countable sequence of points $\{p_t\}_{t\geq 0}$ in $\tilde{S}$.
 Then, since the unitary group of $H$ acts transitively on (complex) orthonormal bases of $H$, there is a sequence of unitary transformations $F_j$ of $H$ such that (after taking a subsequence in $j$ if necessary) for any given $t\geq0$,
 the sequence $F_j\circ \psi_j(p_t)$ converges to a point in $\mathbb{S}_H$ as $j\to\infty$. 
 Since the maps $\psi_j$ are uniformly Lipschitz on compact subsets, by an Arzel\`{a}-Ascoli argument, $F_j\circ \psi_j$ converges in the $C^0$-topology on compact subsets to a limit locally Lipschitz map 
  $$\psi_\infty: \tilde{S}\to \mathbb{S}_{H}.$$ 
 Because of the uniform $C^k$-bound on compact subsets, the convergence can be upgraded to $C^k$-convergence on compact subsets for any $k\geq0$. For instance, in order to check local $C^1$-convergence, note that the second derivative $D^{(2)} (F_j\circ \psi_j)$ is uniformly bounded on compact subsets, so a Taylor expansion at each $p_t$ ensures that the first derivative at $p_t$, $D^{(1)} (F_j\circ \psi_j)(p_t)$, must converge, which is enough to show local $C^1$-convergence.
 For later use, after picking a subsequence if necessary$j_l$, we can assume that $\lim_{l\to \infty} \Eren(S,\rho_{j_l})$ exists and is the liminf of the original sequence.
 %All of this can be checked by approximating Hilbert spaces by finite dimensional spaces.

We now bound the energy of the limit map $\psi_\infty$. 
By the lower semicontinuity of the energy under convergence of harmonic maps, Lemma \ref{monotonicity} and Lemma \ref{blam0}:
 \begin{equation*}
 \begin{split}
 \E(\psi_\infty\vert_{\mathbf{D}_S}) & =\lim_{\delta\to 0}\E(\psi_\infty\vert_{\mathbf{D}_S\cap \tilde{A}_\delta}) \\
 & \leq \lim_{\delta\to 0}\liminf_{j\to \infty}\E(\psi_j\vert_{\mathbf{D}_S\cap \tilde{A}_\delta})\\
 & = \lim_{\delta\to 0} \liminf_{j\to \infty}[\Eren(S,\rho_j)+\sum_{p\in \mathrm{Punc}_S} \frac{\blambda_{\rho_j}(p)^2}{4\pi} \log \frac{1}{\delta}]\\
 & = \liminf_{j\to \infty}\Eren(S,\rho_j). 
 \end{split}
  \end{equation*}
Hence, 
%by Proposition \ref{e<}:
 \begin{equation}\label{lininf e}
 \E(S,\rho') \leq  \E(\psi_\infty\vert_{\mathbf{D}_S})\leq \liminf_{j\to \infty}\Eren(S,\rho_j).
 %\leq \frac{\pi}{4}(2\mathbf{g}+\mathbf{n}-2).
 \end{equation}

 %This smooth limit harmonic map is a map $$\psi_\infty: \tilde{\Sigma}\to \mathbb{S}_H$$
% where $\mathbb{S}_H$ denotes the unit sphere in a Hilbert space $H$ of infinite or finite dimension. We can assume without loss of generality that no totally geodesic subsphere of $\mathbb{S}_H$ contains $\psi_\infty(\tilde{\Sigma})$.

Next, let us explain how the limit map $\psi_\infty$ determines a unitary representation $\rho':{\pi_1(S)}\to \End(H).$
For the sake of simplicity, we write ``converge'' and ``$\lim_{j\to \infty}$'' instead of ``subsequentially converge''.

By reducing the Hilbert space $H$, we will assume without loss of generality that no closed linear (strict) subspace of $H$ contains $\psi_\infty(\tilde{S})$. In other words, there is a sequence of points $z_1,z_2,...,z_i,...\in \tilde{S}$, such that the vectors $\psi_\infty(z_1),\psi_\infty(z_2),...\in \psi_\infty(\tilde{S})$ are linearly independent, and after applying the Gram-Schmidt procedure, we get a Hilbert orthonormal basis $e_1=\psi_\infty(z_1),e_2,...,e_i,...$ of $H$.
Note that $F_j\circ \psi_j(z_i)$ converges to $\psi_\infty(z_i)$.
Thus, for any finite set of complex numbers $a_j$, 
$\sum_{i} a_i F_j\circ \psi_j(z_i)$ converges to $\sum_{i} a_i \psi_\infty(z_i)$. 
Next, for any $g\in {\pi_1(S)}$, and any $z_i$, set
 \begin{equation}\label{ro'}
\rho'(g) \psi_\infty(z_i) := \psi_\infty(g.z_i) = \lim_{j\to \infty} (F_j\circ \rho_j(g)\circ F_j^{-1}) \circ F_j\circ\psi_j(z_i).
\end{equation}
Since $e_i$ is a finite linear combination of $\psi_\infty(z_1),\psi_\infty(z_2),...,\psi_\infty(z_i)$, the above formula defines by linearity a linear transformation $\rho'(g):H\to H$ which is easily checked to be unitary, since $\rho_N(g)$ is unitary. Moreover, by the formula above again, we clearly have $\rho'(gg')=\rho'(g)\rho'(g')$ since the analogue property holds for $F_j\circ \rho_j(.)\circ F_j^{-1}$. This finishes the definition of the unitary representation $\rho'$ and the map $\psi_\infty$, which by construction is ${\pi_1(S)}$-equivariant with respect to $\rho'$, namely $\psi_\infty$ is an element of $\mathscr{H}_{S,\rho'}$:
\begin{equation}\label{rep map}
\rho':{\pi_1(S)}\to \End(H),\quad \psi_\infty:\tilde{S}\to \mathbb{S}_{H}.
\end{equation}

We claim that for all $f\in\ell^1({\pi_1(S)},\mathbb{C})$,
 \begin{equation}\label{claim ro'}
 \|\rho'(f)\|\leq \limsup_j\|\rho_j(f)\|.
  \end{equation}
Indeed, by (\ref{ro'}) and using the previous notations, for any finite linear combination $f=\sum_{k} \alpha_k g_k \in \mathbb{C}[{\pi_1(S)}]$, and any finite linear combination $\sum_i a_i \psi_\infty(z_i)$, 
\begin{align*}
\|\rho'(f) \big(\sum_i a_i \psi_\infty(z_i)\big)\| & =\|\lim_{j\to \infty} (F_j\circ \rho_j(f)\circ F_j^{-1}) \big(\sum_i a_i F_j\circ \psi_j(z_i)\big)\| \\
& = \lim_{j\to \infty}\|(F_j\circ \rho_j(f)\circ F_j^{-1}) \big(\sum_i a_i F_j\circ \psi_j(z_i)\big)\|\\
& \leq \big(\limsup_{j\to\infty}\|\rho_j(f)\|\big) \|\sum_i a_i \psi_\infty(z_i)\|.
\end{align*}
That is enough to show the claim (\ref{claim ro'}).
Since by assumption $\rho_j$ converges strongly, we get from (\ref{claim ro'}) that for all $f\in\ell^1({\pi_1(S)},\mathbb{C})$:
$$\|\rho'(f)\|\leq\|\lambda_{\pi_1(S)}(f)\|.$$
Thus by Theorem \ref{equivalent}, $\rho'\prec \lambda_{\pi_1(S)}$ and so by $C^*$-simplicity of ${\pi_1(S)}$ \cite{Powers75} \cite{DLH07},
\begin{equation}\label{rho'}
\rho'\sim \lambda_{\pi_1(S)}.
\end{equation}
By (\ref{rho'}), Lemma \ref{=+}, Theorem \ref{spherearea} and Corollary \ref{coree} (1) applied to $\rho:=\rho'$ and $\nu:= \lambda_{\pi_1(S)}$,
 we find that 
$$\frac{\pi}{4}(2\mathbf{g}+\mathbf{n}-2) = \E(S,\bigoplus^\infty \lambda_{\pi_1(S)}) \leq    \E(S,\rho').$$
Combined with (\ref{lininf e}), this finishes the proof.
For future use, we record the following consequence:
\begin{equation}\label{limit energy}
 \E(S,\rho') = \E(\psi_\infty\vert_{\mathbf{D}_S}) = \liminf_{j\to \infty}\Eren(S,\rho_j) = \frac{\pi}{4}(2\mathbf{g}+\mathbf{n}-2).
\end{equation}

\end{proof}

\begin{rem}\label{magee}
M. Magee pointed out that, when $\E(S,\rho_j)$ is finite, Proposition \ref{>e}  alternatively follows from a  generalization of the resolvent method of \cite{HM23} applied to flat $\mathbb{C}^N$-bundles over hyperbolic surfaces \cite{Zargar22} \cite{Hide23}. Our proof of Proposition \ref{>e} is different and has the advantage of producing a limit harmonic map, which will be crucial in the proof of the main part of Theorem \ref{thm:main} in the next subsection.
\end{rem}

\subsection{Convergence of the pullback metric} \label{subsection:cv metric}

To finish the proof of the main theorem, it remains to identify the limit of the pullback metric.
For each $N$, let $\psi_j:\tilde{S}\to \mathbb{S}^{2N_j-1}$ be a harmonic representative of $(S,\rho_j)$.
Let $g_{\mathrm{Eucl}}$ be the standard Euclidean metric on $\mathbb{S}^{2N_j-1}$ and let $g_{\mathrm{hyp}}$ be the unique complete, conformal, hyperbolic metric on the Riemannian surface $S$. By equivariance of $\psi_j$, the pullback metric $(\psi_j)^*g_{\mathrm{Eucl}}$ descends to a metric on $S$ still denoted by $(\psi_j)^*g_{\mathrm{Eucl}}$.

\begin{thm} \label{pullback metric}
If $\rho_j$ strongly converges, then the harmonic representatives $\psi_j$ satisfy
$$\lim_{j\to \infty} (\psi_j)^*g_{\mathrm{Eucl}} = \frac{1}{8}g_{\mathrm{hyp}}$$
 in the $C^\infty$ topology on compact subsets of $S$.
\end{thm}

\begin{proof}

Let $\rho':{\pi_1(S)}\to \End(H)$, $\psi_\infty:\tilde{S}\to \mathbb{S}_{H}$ be the limit unitary representation and limit map constructed in (\ref{rep map}) in the proof of Proposition \ref{>e}. 
The map $\psi_\infty$ is the limit of a subsequence of $\psi_j:\tilde{S}\to \mathbb{S}^{2N_j-1}$, after composing with  unitary transformations $F_j$. The convergence is smooth on any compact subset of $\tilde{S}$.
By (\ref{rho'}), we have $\rho'\sim \lambda_{\pi_1(S)}$.

By (\ref{limit energy}), $\E(\psi_\infty\vert_{\mathbf{D}_S}) = \E(S,\rho') = \frac{\pi}{4}(2\mathbf{g}+\mathbf{n}-2)$.
By the equality case of Corollary \ref{slack}, 
\begin{equation}\label{slack consequence}
(\psi_\infty)^*g_{H}  = \frac{1}{8}g_{\mathrm{hyp}}
\end{equation}
where $g_{\mathrm{hyp}}$ is the unique  hyperbolic metric on $S$ compatible with its conformal structure.

By the way $\psi_j$ converges to $\psi_\infty$,  for any fixed compact subset $G$ of $\tilde{S}$,
we have for any $m\geq 0$:
\begin{equation}\label{convergence psi}
\lim_{j\to \infty}\|\psi _{j}^*g_{\mathrm{Eucl}} -\frac{1}{8}g_{\mathrm{hyp}}\|_{C^m(G)}=0.
\end{equation}
This implies the desired conclusion that $\psi _{j}^*g_{\mathrm{Eucl}}$, viewed as a Riemannian metric on $S$, converges to $\frac{1}{8}g_{\mathrm{hyp}}$ on compact subsets of $S$ in the $C^\infty$ topology.

\end{proof}

\vspace{1em}

\section{Applications}\label{section 5}

\subsection{Limit of the $N$th average renormalized energy} \label{average}
Let $S$ be a punctured Riemann surface and suppose that $\pi_1(S)$ is a free group $F_k$ of rank $k\geq 2$, namely $S$ has  at least one puncture. 
Its Euler characteristic is
$$\chi(S) = 1-k.$$
Let $g_{\mathrm{hyp}}$ be its hyperbolic metric.
Unitary representations of $F_k$ into $U(N)$ are in one-to-one correspondence with $k$-tuple $(u_1,...,u_k)\in U(N)^k$. 
For $N\geq 1$, set 
$$\mathbb{E}_N(S) := \text{ average of $\Eren(S,\rho)$ over all unitary representations $\rho:\pi_1(S)\to U(N)$}$$
where the average is taken with respect to the Haar measure on $U(N)^k$. 
This measure is preserved by the action of the outer automorphism group $\text{Out}(F_k)$ (see \cite[Lemma 1.5]{CMP19}), which means that $\mathbb{E}_N(S)$ \emph{does not} depend on the identification $\pi_1(S)\approx F_k$, and in fact induces a function on moduli space of Riemann surfaces. When $\pi_1(S)$ is not free (a case we will not consider here), the definition of $\mathbb{E}_N(S)$ can naturally be generalized by taking the average of $\Eren(S,\rho)$ on the moduli space of unitary representations with respect to the Atiyah-Bott-Goldman symplectic volume form \cite{AB83} \cite{Goldman84}.

In the next statement, by ``random sequence of unitary representations $\tau_N:\pi_1(S)\to U(N)$'', we  mean a sequence sampled on $\Pi_{N\geq 1} U(N)^k$ from the standard product Haar probability measure. Theorem \ref{app1} is a corollary of the following theorem:
\begin{thm} \label{concentration}
Let $S$ be a punctured Riemannian surface with $\pi_1(S)$ isomorphic to a free group of rank $k\geq2$. 
Consider a random sequence of unitary representations  $\tau_N:\pi_1(S)\to U(N)$. For each $N\geq1$, let $\psi_N:\tilde{S}\to (\mathbb{S}^{2N-1},g_{\mathrm{Eucl}})$ be a harmonic representative of $(S,\tau_N)$.
Then  
 $$ \lim_{N\to \infty} \Eren(S,\tau_N) = \frac{\pi}{4}(k-1) \text{ almost surely}$$
and $(\psi_N)^*g_{\mathrm{Eucl}}$ converges to $\frac{1}{8}g_{\mathrm{hyp}}$  almost surely  in the $C^\infty$-topology on compact sets of $S$.
In particular,
 $$ \lim_{N\to \infty}\mathbb{E}_N(S)= \frac{\pi}{4}(k-1).$$
\end{thm}
\begin{proof}
If clearly suffices to prove the first part of the theorem because by Lemma \ref{coro:basic}, $\Eren(S,\rho)\leq C_S$ for a constant depending only on $S$.
By \cite{CM14},  a random sequence $\tau_N:\pi_1(S)\to U(N)$ strongly converges almost surely (see Theorem \ref{CM} in the Appendix). The statement then follows from Theorem \ref{thm:main}.
\end{proof}

\subsection{Special immersions of surfaces into $\mathbb{R}^n$}

\begin{thm} \label{thm:special}
Let $\epsilon>0$. For any closed hyperbolic surface $(\Sigma,g_{\mathrm{hyp}})$, there exists a finite degree covering $(\Sigma',g_{\mathrm{hyp}})$ and a harmonic immersion into a Euclidean unit sphere
$$\psi:(\Sigma',g_{\mathrm{hyp}}) \to (\mathbb{S}^{n_\epsilon},g_{\mathrm{Eucl}})$$
 such that $\psi^* g_{\mathrm{Eucl}} $ is $\epsilon$-close to $\frac{1}{8} g_{\mathrm{hyp}}$ in the $C^2$-topology. 
\end{thm}
\begin{proof}
Set $S$ be the Riemann surface corresponding to $(\Sigma,g_{\mathrm{hyp}})$.
By \cite{LM25} (see Theorem \ref{LM} in the Appendix), there is a sequence of unitary representations 
$$\rho_j:\pi_1(S) \to U(N_j)$$
which strongly converges, and 
such that $\rho_j(\pi_1(S))$ is a finite subgroup of $U(N_j)$. The kernel $\ker(\rho_j)$ is thus a finite index subgroup of $\pi_1(S)$.
This implies that, if $\psi_j:\tilde{S}\to \mathbb{S}^{2N_j-1}$ is a harmonic representative of $(S,\rho_j)$, 
the map $\psi_j$ factors as follows: there is a finite covering $S'$ of $S$ with fundamental group $\ker(\rho_j)$,  and a map 
$$\psi'_j: S' \to \mathbb{S}^{2N_j-1}$$   
such that
$$\psi_j =\psi'_j \circ \Pi$$
where $\Pi:\tilde{S}\to S'$ is the natural projection. This  $\psi'_j $ is a harmonic map and by Theorem \ref{thm:main}, the pullback metric $(\psi'_j)^*g_{\mathrm{Eucl}}$ converges to  $\frac{1}{8}g_{\mathrm{hyp}}$ and in particular $\psi'_j$ is an immersion for large $j$. The theorem is proved.

\end{proof}

\begin{rem}[Meaning of $\frac{1}{8}$]\label{meaning}
For $\epsilon>0$ small, the harmonic map $\psi$ in Theorem \ref{thm:special} is given by $n_\epsilon+1$ coordinates functions $(\psi)^i:(\Sigma',g_{\mathrm{hyp}})\to \mathbb{R}$ which are almost Laplace eigenfunctions with eigenvalue $2$. Indeed, the image $\psi(\Sigma')$ is almost a minimal surface in $\mathbb{S}^{n_\epsilon}$, and coordinate functions on 2-dimensional minimal surfaces in Euclidean unit spheres are Laplace eigenfunctions with eigenvalue $2$ by \cite{Takahashi66}.  But $2$ is exactly the bottom of the Laplace operator on the rescaled hyperbolic plane $(\mathbb{H}^2,\frac{1}{8}g_{\mathrm{hyp}})$.  Optimistically, it could be that Theorem \ref{thm:special}  holds (resp. cannot hold) if the factor $\frac{1}{8}$ is replaced by any larger (resp. smaller) constant.
\end{rem}

\subsection{Almost hyperbolic minimal surfaces in spheres} \label{min surff}

%Given a branched minimal surface  $\mathfrak{S}$ in some Riemannian manifold $(M,g_M)$, we call \emph{parameterization} of  $\mathfrak{S}$ any branched minimal immersion $$\varphi : T\to (M,g_M)$$ such that $\mathfrak{S} = \varphi(T)$, where $T$ is a Riemann surface.  The parametrization is  called \emph{simple} if there is no other parameterization  $\mathfrak{S} = \hat{\varphi}(\hat{T})$ and a branched covering map $f:T\to \hat{T}$ so that $\varphi=\hat{\varphi} \circ f$. Simple parameterizations of branched minimal surfaces always exist.  For a simple parameterization $\mathfrak{S}=\varphi(T)$, the pullback metric $\varphi^*g_M$ is a nondegenerate Riemannian metric everywhere on $T$ except at finitely many points. Its Gaussian curvature is well-defined almost everywhere, and is denoted by $K_{\varphi^*g_M}$. The area $\Area(\mathfrak{S})$ is naturally defined as $\Area(T,\varphi^*g_M)$, 

Consider a Riemann surface $T$ and 
$$\varphi:T\to (M,g_M)$$
a branched minimal immersion of $T$ into some Riemannian manifold $(M,g_M)$. Then the pullback metric $\varphi^*g_M$ is a nondegenerate Riemannian metric everywhere on $T$ except at a discrete set of points. It induces a well-defined metric on $T$, and its Gaussian curvature, denoted by $K_{\varphi^*g_M}$, is well-defined almost everywhere. 
The map $\varphi$ is said to be  \emph{simple} if there is no nontrivial branched covering map $f:T\to \hat{T}$ and there is no map $\hat{\varphi}:\hat{T}\to M$ so that $\varphi=\hat{\varphi} \circ f$. 
%For any branched minimal immersion, there is a simple branched immersion with same image.  

\begin{thm} \label{min surf 1}
For any $j\geq 1$, there are a closed Riemann surface $T_j$ and a simple, branched, minimal immersion
$$\varphi_j: T_j \to (\mathbb{S}^{n_j},g_{\mathrm{Eucl}})$$
such that
\begin{enumerate}
\item $\lim_{j\to \infty} \frac{1}{\Area(T_j, \varphi_j^*g_{\mathrm{Eucl}})} \int_{T_j} |K_{\varphi_j^*g_{\mathrm{Eucl}}}+8|=0$,
%where $K_{\mathfrak{S}_j}$ is the Gaussian curvature of $\mathfrak{S}_j$,
\item $(T_j,\varphi_j^*g_{{\mathrm{Eucl}}})$ Benjamini-Schramm  converges to $(\mathbb{H}^2,\frac{1}{8}g_{\mathrm{hyp}})$ as $j\to \infty$.
\end{enumerate}
\end{thm}

Benjamini-Schramm convergence was originally a notion of convergence for random graphs. It has been recently extended to manifolds, see for instance Ab\'{e}rt-Biringer \cite[Section 1]{AB22}.
In our statement, it means the following: 
there are regions $\Omega_j  \subset T_j$ such that 
$$ \lim_{j \to \infty}  \frac{\mathrm{Area}\bigl(\Omega_j, \varphi_j ^*g_{\mathrm{Eucl}} \bigr)}{\mathrm{Area}\bigl(T_j, \varphi_j ^*g_{\mathrm{Eucl}} \big)} = 1 $$
and for any $R>0$ and  $x_j \in \Omega_j$, the metric $R$-ball in $\bigl(T_j, \varphi_j^* g_{\mathrm{Eucl}}\bigr)$ centered at $x_j$ converges in the Gromov-Hausdorff topology to an $R$-ball in $\bigl(\mathbb{H}^2, \frac{1}{8}g_{\mathrm{hyp}}\bigr)$.

%means that for \emph{any} choice of parameterization $\mathfrak{S}_j = \varphi_j(T_j)$ where $T_j$ is a closed topological surface and $\varphi_j:T_j\to \mathbb{S}^{n_j}$ is a branched immersion, the following holds: there exist regions 
%$\Omega_j  \subset T_j $ such that 
%$$ \lim_{j \to \infty} 
%\frac{\mathrm{Area}\bigl(\Omega_j , (\varphi_j )^*g_{\mathrm{Eucl}}\bigr)}{\mathrm{Area}\bigl(T_j , (\varphi_j )^* g_{\mathrm{Eucl}}\bigr)} = 1 $$
%and for any $R>0$ and  $x_j \in \Omega_j$, the metric $R$-ball in $\bigl(T_j, (\varphi_j)^* g_{\mathrm{Eucl}}\bigr)$ centered at $x_j$ converges in the Gromov-Hausdorff topology to an $R$-ball in $\bigl(\mathbb{H}^2, \frac{1}{8}g_{\mathrm{hyp}}\bigr)$.

%the $(\varphi_j )^*g_{\mathrm{Eucl}}$-injectivity radius at $x_j$ tends to $\infty$ as $j\to \infty$, and $(\varphi_j )^*g_{\mathrm{Eucl}}$ converges smoothly to a rescaled hyperbolic metric $\frac{1}{8}g_{\mathrm{hyp}}$ on the embedded $R$-ball around $x_j$.

%and for any $R > 0$ and $x_j \in \Omega_j$, 
%the metric $R$-ball in $\bigl(T_j, (\varphi_j)^* g_{\mathrm{Eucl}}\bigr)$ centered at $x_j$ converges in the $C^\infty$-topology, to an $R$-ball in $\bigl(\mathbb{H}^2, \frac{1}{8}g_{\mathrm{hyp}}\bigr)$.

\begin{proof}[Proof of Theorem \ref{min surf 1}]

\textbf{For (1):}
Take $S$ to be the round unit 2-sphere with three points removed. Let $\Sigma$ be the underlying topological surface, which is a thrice-punctured sphere. 
It is well-known that the Teichm\"{u}ller space of $\Sigma$ is reduced to a single point $\mathscr{T}_\Sigma = \{\mu_0\}$, and that the mapping class group of $\Sigma$ (i.e. the set of isotopy classes of the group of orientation-preserving homeomorphisms of $\Sigma$) is  isomorphic to symmetric group $\mathbf{S}_3$ (the permutation group of $\{1,2,3\}$) \cite[remark after Proposition 10.5, and Proposition 2.3]{FaMa11}.
Let $g_{\mathrm{hyp}}$ be the complete finite area hyperbolic metric on $\Sigma$, unique up to isometries. The elements of the mapping class group are represented by the ``obvious'' symmetries of $(\Sigma,g_{\mathrm{hyp}})$.

%Let $a,b$ denote a choice of two generators of the fundamental group ${\pi_1(S)}\approx F_2$ of $\Sigma$, such that $a,b, b^{-1}a$ correspond to embedded loops around each of the three punctures. By \cite{}, for any $\epsilon>0$, for any $z\in \mathbb{C}[\pi_1(\Sigma)]$, if we choose the images of the generators $a,b$ uniformly at random in the symmetric group $\mathbf{S}_n$, then we have with probability tending to $1$ as $n\to \infty$ $$\|(\std\circ \theta_N)(z)\|\leq \|\lambda_{F_2}(z)\|+\varepsilon.$$

We have $\pi_1(S)\approx F_2$.
By \cite{BC19}, namely Theorem \ref{BC} in the Appendix, and Lemma \ref{proba1}, there exists a sequence of unitary representations 
$\rho_j:\pi_1(S)\to U(N_j)$ 
which strongly converges,  such that $\rho_j(\pi_1(S))$ is finite, and such that $\rho_j$ satisfies $\blambda_{\rho_j}(p)=0$ for each puncture $p$ of $S$ (see (\ref{blambda}) for the definition of $\blambda_{\rho_j}(p)$). In particular 
\begin{equation}\label{e=eren}
\E(S,\rho_j)=\Eren(S,\rho_j)<\infty
\end{equation} 
by Corollary \ref{coro:basic} (1) and Lemma \ref{lem:one}. Moreover by Theorem \ref{energy continuity},  for large $j$,
$$\E(S,\rho_j)>0.$$ 
%These $\rho_j$ are obtained with high probability by sending each generator of $F_2$ to independent random permutations, see \ref{.}.

By Theorem \ref{realize}, 
there is a harmonic representative $\psi_j:\tilde{S}\to \mathbb{S}^{2N_j-1}$ of $(S,\rho_j)$, so thanks to (\ref{e=eren}), 
%$$\Area(\mathbf{D}_S, (\psi_j)^*g_{\mathrm{Eucl}})\leq
$ \E(\psi_j\vert_{\mathbf{D}_S}) = \E(S,\rho_j).$
Because the Teichm\"{u}ller space of $\Sigma$ is trivial, actually $
\E(S,\rho_j)=\Area(\Sigma,\rho_j)
$
(see Definition \ref{def spherearea bis}). By Theorem \ref{min surf thm}, $\psi_j$ is both harmonic and weakly conformal, namely it is a branched minimal immersion. Hence,
\begin{equation}\label{e=a}
\Area(\mathbf{D}_S, (\psi_j)^*g_{\mathrm{Eucl}}) = \E(\psi_j\vert_{\mathbf{D}_S}) = \E(S,\rho_j)=\Area(\Sigma,\rho_j).
\end{equation}

Since $\rho_j(\pi_1(S))$ is finite, $\ker(\rho_j)$  has finite index in $\pi_1(S)$. Thus, we can consider the connected finite covering $S_j$ of $S$, with fundamental group $\ker(\rho_j)$. The map $\psi_j$ factors through a  branched minimal immersion $\psi'_j:S_j\to \mathbb{S}^{2N_j-1}$ with finite total energy. 
Let $\bar{S}_j$ be the closed Riemann surface obtained from $S_j$ by closing the punctures.
By the removable singularity theorem of Sacks-Uhlenbeck \cite{SU81} \cite[Theorem 4.26]{CM11},  $\psi'_j$ extends across the punctures to  a  branched minimal immersion still denoted by 
$$\psi'_j: \bar{S}_j \to\mathbb{S}^{2N_j-1}.$$ 
As a consequence,
$$\psi_j(\tilde{S}) = \psi'_j(\bar{S}_j)$$
is a closed, branched, minimal surface in $\mathbb{S}^{2N_j-1}$.

By Theorem \ref{thm:main}, the pullback metrics $(\psi_j)^*g_{\mathrm{Eucl}}$, which descends to a metric on $S$ by equivariance of $\psi_j$, 
% and $(\psi'_j)^*g_{\mathrm{Eucl}}$ 
converges smoothly to $\frac{1}{8}g_{\mathrm{hyp}}$ on any compact subsets of $S$.
Hence, the Gaussian curvature of $(\psi_j)^*g_{\mathrm{Eucl}}$ converges smoothly to $-8$ away from the punctures.
It remains to control the Gaussian curvature around the punctures, which is achieved with general properties of minimal surfaces. 
Let $\alpha_j$ be the degree of the covering $S_j\to S$.
Let $\delta\in(0,1]$ be a small constant.
Let ${A}_\delta\subset{S}$, $\tilde{A}_\delta\subset\tilde{S}$ be defined as in (\ref{adelta}), and let $\tilde{A}_{j}$ be the image of $\tilde{A}_\delta$ in $S_j$ by the projection $\tilde{S}\to S_j$.
Write $$\tilde{Z}_j:=\bar{S}_{j}\setminus \tilde{A}_{j}.$$
Both $\tilde{A}_j$ and $\tilde{Z}_j$ implicitly depend on $\delta$.
By (\ref{e=a}), and Theorem \ref{thm:main}, 
%$$\lim_{\delta\to \infty} \liminf_{N\to \infty} \frac{E(\psi_{\rho_N}\vert_{\tilde{A}_N})}{E(\psi_{\rho_N}\vert_{\Sigma_{\rho_N}})}=1.$$
%Hence, by Theorem \ref{},
\begin{equation}\label{limit quotient}
\begin{split}
&\lim_{j\to \infty}\frac{\Area(\bar{S}_{j},(\psi'_j)^*g_{\mathrm{Eucl}})}{\alpha_{j}} = \frac{\pi}{4},\\
&\lim_{\delta\to 0} \liminf_{j\to \infty}\frac{\Area(\tilde{A}_{j},(\psi'_j)^*g_{\mathrm{Eucl}})}{\Area(\bar{S}_{j},(\psi'_j)^*g_{\mathrm{Eucl}})} =1,\\
& \lim_{\delta\to 0} \limsup_{j\to \infty}\frac{\Area(\tilde{Z}_j,(\psi'_j)^*g_{\mathrm{Eucl}})}{\Area(\bar{S}_{j},(\psi'_j)^*g_{\mathrm{Eucl}})} =0,
\end{split}
\end{equation}
\begin{equation}\label{limit ||}
\lim_{j\to \infty}\frac{1}{\Area(\bar{S}_{j},(\psi'_j)^*g_{\mathrm{Eucl}})}\int_{\tilde{A}_{j}}|K_{(\psi'_j)^*g_{\mathrm{Eucl}}}+8|=0.
\end{equation}

Note that $\partial A_\delta$ is a union of embedded curves around punctures with geodesic curvature constant equal to $8$ with respect to the rescaled  hyperbolic metric $\frac{1}{8}g_{\mathrm{hyp}}$ on $S$.
Thus by (\ref{convergence psi}) again, as $j\to \infty$, the geodesic curvature $\kappa$ of $\partial \tilde{Z}_j$ inside of $(\tilde{Z}_N,(\psi'_j)^*g_{\mathrm{Eucl}})$ becomes arbitrarily close to the constant $8$. Besides, 
$$\lim_{j\to \infty}\frac{1}{\alpha_{j}}\Length(\partial \tilde{Z}_j,(\psi'_j)^*g_{\mathrm{Eucl}}) =\lim_{j\to \infty}\frac{1}{\alpha_{j}}\Length(\partial \tilde{A}_{j},(\psi'_j)^*g_{\mathrm{Eucl}}) = \Length(\partial {A}_{\delta} ,\frac{1}{8}g_{\mathrm{hyp}})$$
Thus, by (\ref{limit quotient}) and since $\lim_{\delta\to 0} \Length(\partial {A}_{\delta},\frac{1}{8}g_{\mathrm{hyp}}) =0$,  
\begin{equation}\label{moyenne kappa}
\liminf_{\delta\to 0}\limsup_{j\to \infty}\big| \frac{1}{\Area(\bar{S}_{j},(\psi'_j)^*g_{\mathrm{Eucl}} )}\int_{\partial\tilde{Z}_j}\kappa\big| = 0.
\end{equation}

The restriction of the map $\psi'_{j}$  from $ \tilde{Z}_j$ to $\mathbb{S}^{2N_j-1}$ is a branched minimal immersion of a disjoint union of disks.  By the Gauss-Bonnet formula applied to $\tilde{Z}_j$,
$$\int_{\tilde{Z}_j} K_{(\psi'_j)^*g_{\mathrm{Eucl}}} \geq  - \int_{\partial \tilde{Z}_j} \kappa ,$$
where the inequality comes from the branched points and the disk topology (an easy way to see this inequality is to observe that each branched point can be locally smoothed in a neighborhood to a smooth immersion such that the integral of the Gaussian curvature in this neighborhood is negative). From (\ref{moyenne kappa}), we get
\begin{equation}\label{moyenne geq}
\liminf_{\delta\to 0}\liminf_{j\to \infty} \frac{1}{\Area(\bar{S}_j,(\psi'_j)^*g_{\mathrm{Eucl}})}\int_{\tilde{Z}_j} K_{(\psi'_j)^*g_{\mathrm{Eucl}}} \geq 0.
\end{equation}

Next, by the Gauss equation \cite[Chapter 6, Theorem 2.5]{DoCarmo92}, the Gaussian curvature (which is well-defined almost everywhere) satisfies
%\begin{equation}\label{gauss}
$$K_{(\psi'_j)^*g_{\mathrm{Eucl}}}(x) \leq 1\quad \text{for almost all $x\in \bar{S}_j$}.$$
%\end{equation}
%Combining that with (\ref{limit quotient}) and (\ref{moyenne geq}), 
%$$\liminf_{\delta\to 0}\liminf_{N\to \infty} \frac{1}{\Area(T _N)}\int_{\tilde{Z}_N} K_{T_N} = 0.$$
This implies that
\begin{align*}
& \frac{1}{\Area(\bar{S}_j,(\psi'_j)^*g_{\mathrm{Eucl}})}  \int_{\tilde{Z}_j} |K_{(\psi'_j)^*g_{\mathrm{Eucl}}}|\\
& \leq \frac{\Area(\tilde{Z}_j,(\psi'_j)^*g_{\mathrm{Eucl}})}{\Area(\bar{S}_j,(\psi'_j)^*g_{\mathrm{Eucl}})} + \frac{1}{\Area(\bar{S}_j,(\psi'_j)^*g_{\mathrm{Eucl}})}\int_{\tilde{Z}_j} (1-K_{(\psi'_j)^*g_{\mathrm{Eucl}}}) \\
& = 2\frac{\Area(\tilde{Z}_j,(\psi'_j)^*g_{\mathrm{Eucl}})}{\Area(\bar{S}_j,(\psi'_j)^*g_{\mathrm{Eucl}})}  - \frac{1}{\Area(\bar{S}_j,(\psi'_j)^*g_{\mathrm{Eucl}})}\int_{\tilde{Z}_j} K_{(\psi'_j)^*g_{\mathrm{Eucl}}} .
\end{align*}
So combining (\ref{limit quotient}) and (\ref{moyenne geq}, we find
$$\limsup_{\delta\to 0} \limsup_{j\to \infty} \frac{1}{\Area(\bar{S}_j,(\psi'_j)^*g_{\mathrm{Eucl}})}\int_{\tilde{Z}_j} |K_{(\psi'_j)^*g_{\mathrm{Eucl}}}| =0.$$
Together with (\ref{limit quotient}) and (\ref{limit ||}), we get
\begin{align*}
0\leq & \limsup_{j\to  \infty} \frac{1}{\Area(\bar{S}_j,(\psi'_j)^*g_{\mathrm{Eucl}})}\int_{\bar{S}_j}|K_{(\psi'_j)^*g_{\mathrm{Eucl}}}+8| \\
 \leq &\limsup_{\delta\to 0} \limsup_{j\to  \infty}\big( \frac{1}{\Area(\bar{S}_j,(\psi'_j)^*g_{\mathrm{Eucl}})}\int_{\tilde{A}_j}|K_{(\psi'_j)^*g_{\mathrm{Eucl}}}+8| \\
 & +\frac{1}{\Area(\bar{S}_j,(\psi'_j)^*g_{\mathrm{Eucl}})} \int_{\tilde{Z}_j}(|K_{(\psi'_j)^*g_{\mathrm{Eucl}}}|+8)\big)\\
  = &0.
 \end{align*}
 This almost ends the proof of (1): the branched minimal immersion $\psi'_j: \bar{S}_j \to\mathbb{S}^{2N_j-1}$ satisfies all the requirements except that  it is maybe not simple. But we can always find another branched minimal immersion $\varphi_j: T_j \to\mathbb{S}^{2N_j-1}$ where $T_j$ is a closed Riemann surface covered by $\bar{S}_j$, which is simple and which has same image as $\psi'_j$. Since clearly 
 $$ \frac{1}{\Area(\bar{S}_j,(\psi'_j)^*g_{\mathrm{Eucl}})}\int_{\bar{S}_j}|K_{(\psi'_j)^*g_{\mathrm{Eucl}}}+8| =  \frac{1}{\Area(T_j,\varphi_j^*g_{\mathrm{Eucl}})}\int_{T_j}|K_{\varphi_j^*g_{\mathrm{Eucl}}}+8|,$$ we conclude (1).

\vspace{1em}

\textbf{For (2):}
%Continuing with the above notations, $\mathfrak{S}_j = \psi_j(\tilde{S})=\psi_j'(\bar{S}_j)$.
% and $S_j$ is the connected finite covering of $S$ with fundamental group $\ker(\rho_j)$. 
We continue with the notations used above.
Recall that  
$$\psi_j(\tilde{S})=\psi_j'(\bar{S}_j) = \varphi_j(T_j)$$
where $\psi_j$, $\psi_j'$ and $\varphi_j$ are branched minimal immersions and $\varphi_j$ is simple.
From the proof of Proposition \ref{>e} and Theorem \ref{pullback metric}, we can assume that $\psi_j$ smoothly converges (after composing with unitary matrices) on compact subsets to a smooth, minimal immersion $\psi_\infty:\tilde{S}\to \mathbb{S}_H$ where $\mathbb{S}_H$ is the unit sphere of the smallest Hilbert space containing $\psi_\infty(\tilde{S})$, and $\psi_\infty$ is equivariant with respect to a representation $\rho':\pi_1(S)\to \End(H)$ such that $\rho'\sim \lambda_{\pi_1(S)}$,
%weakly equivalent to the regular representation $\lambda_{\pi_1(S)}$
and $\psi_\infty$ is equivariantly area-minimizing. 
By (\ref{slack consequence}), $(\psi_\infty)^*g_H = \frac{1}{8}g_{\mathrm{hyp}}$ where $g_{\mathrm{hyp}}$ is the hyperbolic metric on $\tilde{S}$ given by the lift of the hyperbolic metric on the thrice-punctured sphere $S$. In particular for each $\delta>0$, $\psi_j$ is an immersion on $\tilde{A}_\delta$ for large $j$, and  $(\psi_j)^*g_{\mathrm{Eucl}}$ converges smoothly to $\frac{1}{8}g_{\mathrm{hyp}}$ on $\tilde{A}_\delta$. By (\ref{limit quotient}), the ratio of the $(\psi_j)^*g_{\mathrm{Eucl}}$-area of $\tilde{A}_{j}$ and the $(\psi_j)^*g_{\mathrm{Eucl}}$-area of $\bar{S}_j$ tends to $1$ as $j\to \infty$ and $\delta\to 0$.

%viewed as a metric on $S$, converges smoothly to $\frac{1}{8}g_{\mathrm{hyp}}$ on compact subsets of $S$. By (\ref{.}), the region of $S$ where $(\psi_j)^*g_{\mathrm{Eucl}}$ is far from $\frac{1}{8}g_{\mathrm{hyp}}$ has $(\psi_j)^*g_{\mathrm{Eucl}}$-area tending to $0$ as $j\to \infty$.

%Consider any parameterization $\varphi_j(T_j) = \mathfrak{S}_j$ where $T_j$ is a closed topological  surface and $\varphi_j$ is a branched immersion. 
We say that a branched immersion $\varphi$ from a domain $D$ is ``\emph{almost injective}'' when it sends any two disjoint open sets $U,V\subset D$ to different images $\varphi(U)\neq \varphi(V)$. 
Due to the convergence properties of $(\psi_j)^*g_{\mathrm{Eucl}}$ recalled above,
if Benjamini-Schramm convergence of $(T_j,\varphi_j^*g_{\mathrm{Eucl}})$ to $(\mathbb{H}^2, \frac{1}{8}g_{\mathrm{hyp}})$ fails, it must be because for some $R>0$, some $\delta>0$ and some $x_j\in \tilde{A}_{\delta}$, the restriction of $\psi_j$ to the $R$-ball in $(\tilde{S},(\psi_j)^*g_{\mathrm{Eucl}})$ centered at $x_j$ is not almost injective for all $j$ large (i.e. $\psi_j$ maps the $R$-ball to a surface which intersects itself on some open set).   
Thus, since $\psi_j$ converges to $\psi_\infty$ on compact sets, in order to show Benjamini-Schramm convergence of $(T_j,\varphi_j^*g_{\mathrm{Eucl}})$ to $(\mathbb{H}^2,\frac{1}{8}g_{\mathrm{hyp}})$, we just need to show that the map $\psi_\infty$ is almost injective.

%Consider a finite-degree  covering of $S$ with minimal degree, called $S_{j,\mathrm{min}}$,  such that $\psi_j$ factors through some map $\eta_j:S_{j,\mathrm{min}}\to \mathbb{S}^{2N_j-1}$, and consider the closed surface $\bar{S}_{j,\mathrm{min}}$  obtained by adding the punctures of $S_{j,\mathrm{min}}$.  As we explained earlier, using the standard removable singularity theorem \cite{}, the map $\eta_j$ extends to a branched minimal immersion  $\eta_j:\bar{S}_{j,\mathrm{min}}\to \mathbb{S}^{2N_j-1}$ which  is simple by minimality of the covering degree of $S_{j,\mathrm{min}}$, and around every point $p_0\in (\bar{S}_{j,\mathrm{min}},(\eta_j)^*g_{\mathrm{Eucl}})$ which is not a puncture, there are larger and larger embedded disks as $j\to \infty$ since $\psi_\infty$ is simple. This implies the desired Benjamini-Schramm convergence.

We now show that $\psi_\infty$ is almost injective. If that is not the case, then by the classical unique continuation property of minimal surfaces \cite[Theorem 6.1]{CM11} and the fact that the pullback metric $(\psi_\infty)^*g_H$ is a rescaled hyperbolic metric, there is a nontrivial group of isometries $\mathcal{J}$ 
defined as the set of isometries $J$ of $(\tilde{S},\frac{1}{8}g_{\mathrm{hyp}})$ such that for any $x\subset \tilde{S}$, $\psi_\infty(J(x)) = \psi_\infty(x)$. Since $\psi_\infty$ is an immersion,   $\mathcal{J}$  acts freely on the rescaled hyperbolic plane $(\tilde{S}, \frac{1}{8}g_{\mathrm{hyp}})$.
Recall that $\pi_1(S)$ also acts by isometries on $(\tilde{S}, \frac{1}{8}g_{\mathrm{hyp}})$.
Let $\mathcal{J}'$ be the subgroup of isometries generated by these two subgroups. 
Let us check that this group of isometries can neither be non-discrete, nor be discrete.
Consider the surface $\tilde{S}/\mathcal{J}$ endowed with its rescaled hyperbolic metric $\frac{1}{8}g_{\mathrm{hyp}}$.  This surface is not the hyperbolic plane because $\mathcal{J}$ is nontrivial. The group $\pi_1(S)$ is contained in the normalizer of $\mathcal{J}$ because if $J\in \mathcal{J}$ and $a\in \pi_1(S)$, then for any $x\subset \tilde{S}$, 
$$\psi_\infty(aJa^{-1}(x)) = \rho'(a)\big(\psi_\infty(Ja^{-1}(x))\big) = \rho'(a)\big(\psi_\infty(a^{-1}(x))\big)  = \psi_\infty(x).$$
So $\pi_1(S)$ also acts by isometries on the rescaled hyperbolic surface $\tilde{S}/\mathcal{J}$. It is known that the isometry group of a complete (rescaled) hyperbolic surface, such as $\tilde{S}/\mathcal{J}$, is either discrete, or isomorphic to the circle $\mathbb{S}^1$ (in that case $\tilde{S}/\mathcal{J}$ is a topological annulus and $\mathcal{J}$ is abelian), see \cite{Agol11}. In the latter case, the group $\pi_1(S)$ would have to act as a subgroup of the circle on the topological annulus $\tilde{S}/\mathcal{J}$, but this is impossible since $\pi_1(S)$ would have to be abelian. Hence $\pi_1(S)$ acts as a discrete subgroup of isometries on $\tilde{S}/\mathcal{J}$, and so $\mathcal{J}'$ has to be discrete.
But this is impossible too.
Indeed if it  is discrete, since $\pi_1(S)$ is a strict subgroup of $\mathcal{J}'$,  the quotient $\tilde{S}/\mathcal{J}'$ is a nontrivial quotient of $(S,\frac{1}{8}g_{\mathrm{hyp}})$. But $S$ is the thrice-punctured sphere, which has no smooth nontrivial quotient, so $\tilde{S}/\mathcal{J}'$ has at least one orbifold point with angle strictly less than $2\pi$. 
This means that for some nontrivial element $g\in \pi_1(S)$, the corresponding isometry of  $\tilde{S}/\mathcal{J}$ is torsion. 
Since $H$ is by assumption the smallest Hilbert space containing $\psi_\infty(\tilde{S})$, the unitary transformation $\rho'(g)$ of $H$ is uniquely determined by its action on $\psi_\infty(\tilde{S})$. We deduce that
$\rho'(g)$ is also torsion.
%: for some integer $m>0$, $\rho'(g)^m = \rho'(g^m)=\Id$. 
But then, if $e\in \pi_1(S)$ is the identity element, for some positive integer $m$, $0=\|\rho'(e) -\rho'(g^m) \| =  \|\lambda_{\pi_1(S)}(e) -\lambda_{\pi_1(S)}(g^m)\|\neq 0$ were the middle equality comes from Theorem \ref{equivalent} and the fact that $\rho'\sim \lambda_{\pi_1(S)}$. This is a contradiction, thus $\mathcal{J}$ cannot be nontrivial, and  $\psi_\infty$ is almost injective.

\end{proof}

\vspace{1em}

\section{Appendix: About unitary representations} \label{section 6}

\subsection{Weak containment and weak equivalence} \label{weak equiv}

All Hilbert spaces in this paper will be complex unless otherwise noted, and will have either finite or infinite countable dimensions.
Let $\Gamma$ be a finitely generated group. 
A unitary representation of $\Gamma$ is a pair $(\rho,H)$ such that $H$ is a Hilbert space, and $\rho$ is a group morphism from $\Gamma$ to the unitary group $U(H)$ of $H$. 
%Any unitary representation $(\pi,H)$ induces an orthogonal representation still denoted by $(\pi,H)$, where $H$ is viewed as the underlying real Hilbert space, and $\rho$ is viewed as a morphism from $\Gamma$ to the orthogonal group $O(H)$.

There is a natural topology on the set of all unitary representations of $\Gamma$, called the Fell topology \cite[Appendix F]{BDLHV08}. Convergence in the Fell topology can be described using the notion of weak containment \cite[Definition F.1.1]{BDLHV08}:
 \begin{defn} \label{weak equiv}
 Let $(\pi,H)$ and $(\rho,K)$ be two unitary representations of $\Gamma$. We say that $\pi$ is weakly contained in $\rho$ if for every $\xi \in H$, every finite subset $Q$ of $\Gamma$, and every $\epsilon>0$, there exist $\eta_1,...,\eta_m$ in $K$ such that for all $g\in Q$,
 $$\big| \langle \pi(g)\xi,\xi \rangle - \sum_{j=1}^m \langle \rho(g)\eta_j,\eta_j\rangle \big| <\epsilon.$$
 We write $\pi\prec\rho$ if the above holds. 
 If we have both  $\pi\prec\rho$ and $\rho\prec\pi$, then we say that $\pi$ and $\rho$ are weakly equivalent, and we write $\pi\sim\rho$.
 \end{defn}

%Let $\ell^1(\Gamma,\mathbb{C})$ be the set of $\mathbb{C}$-valued $\ell^1$-functions on $\Gamma$. 
Let $\mathbb{C}[\Gamma]$ be the set of complex, finite, linear combinations of elements of $\Gamma$. 
For any unitary representation $(\pi,H)$  of $\Gamma$ and $f\in \mathbb{C}[\Gamma]$,
%$f\in \ell^1(\Gamma,\mathbb{C})$, 
there is a well-defined continuous linear operator
$$\pi(f):H \to H,$$
see \cite[Section F.4]{BDLHV08}. Given such $\pi$ and $f$, let $\|\pi(f)\|$ denote the operator norm of $\pi(f)$.The next statement is contained in \cite[Theorem F.4.4]{BDLHV08}:
\begin{thm} \label{equivalent}
Let $(\pi,H)$ and $(\rho,K)$ be two unitary representations of $\Gamma$. The following are equivalent:
\begin{enumerate}
\item $\pi\prec\rho$
\item $\|\pi(f)\| \leq \|\rho(f)\|$ for all $f\in \mathbb{C}[\Gamma]$.
%$f\in \ell^1(\Gamma,\mathbb{C})$.
\end{enumerate}
\end{thm}
Let $\ell^1(\Gamma,\mathbb{C})$ be the set of $\mathbb{C}$-valued $\ell^1$-functions on $\Gamma$.  An elementary fact is that $\|\pi(f)\| \leq \|f\|_{\ell^1}$ and so in the theorem above, ``$f\in \mathbb{C}[\Gamma]$'' can be replaced with ``$f\in \ell^1(\Gamma,\mathbb{C})$''.
%\begin{equation}\label{l1 bound}
%\|\pi(f)\| \leq \|f\|_{\ell^1}.
%\end{equation}

If $\Gamma$ is a non-abelian surface group, then it is known \cite{DLH85} that $\Gamma$ is ``$C^*$-simple'': in particular, any unitary representation of $\Gamma$ weakly contained in $\lambda_\Gamma$ is actually weakly equivalent to $\lambda_\Gamma$. See the survey \cite{DLH07} for more details on $C^*$-simplicity.

\subsection{Strong convergence} \label{strong cv}
Recall that the regular representation $\lambda_\Gamma : \Gamma\to \End(\ell^2(\Gamma,\mathbb{C}))$ of $\Gamma$ is the following canonical representation: for all $\gamma,x\in \Gamma$, $f\in \ell^2(\Gamma,\mathbb{C})$,
$$(\lambda_\Gamma(\gamma).f) (x) := f(\gamma^{-1}x).$$
Crucial to us is the notion of strong convergence:
\begin{defn} \label{definition:strong cv}
A sequence of unitary representations $\rho_j:\Gamma\to \End(V_j)$  strongly converges to the regular representation if
% to the regular representation $\lambda_\Gamma$ of $\Gamma$ if 
$$
\forall z\in \mathbb{C}[\Gamma],\quad \lim_{j\to \infty} \|\rho_j(z)\| = \|\lambda_\Gamma(z)\|
$$
where $\lambda_\Gamma$ is the regular representation of $\Gamma$.
\end{defn}
Here $\|.\|$ denotes the operator norm as before. In this paper, we omit ``to the regular representation'' when speaking about strong convergence.
The notion of strong convergence is closely connected to the notion of ``strong asymptotic freeness''. 

%By (\ref{l1 bound}), it is simple to check that if $\rho_k$ strongly converges to $\lambda_\Gamma$, then in fact 
%$$
%\forall z\in \ell^1(\Gamma,\mathbb{C}),\quad \lim_{k\to \infty} \|\rho_k(z)\| = \|\lambda_\Gamma(z)\|.
%$$
%Moreover, if $\Gamma$ is $C^*$-simple \cite[Definition 2]{DLH07}, for instance if $\Gamma$ is a free group, then this inequality can automatically be improved to an equality (see beginning of \cite[Subsection 5.2, Proof of Theorem 1.1]{LM22}):
%\begin{equation}\label{strong cv}
%\forall z\in \ell^1(\Gamma,\mathbb{C}),\quad \lim_{k\to \infty} \|\rho_k(z)\| = \|\lambda_\Gamma(z)\|.
%\end{equation}

\subsection{Examples of strongly convergent sequences} \label{examples}
For the reader's convenience, we list the results on strong convergence which are directly applied in this paper. In what follows, let $F_k$ be a free group of rank  $k\geq 2$, generated by $x_1,...,x_k$. Let $\lambda_{F_k}$ denote its (left) regular  representation.

\subsubsection{Random Haar unitaries} 
The first examples of unitary representations of free groups which strongly converge were constructed by Haagerup-Thorbj\/{o}rnsen \cite{HT05}. These representations are defined by sending generators of the free group to random unitary matrices in $U(N)$ sampled with the probability measure which is the ``exponential'' of the GUE model \cite[Theorem 8.2 and Remark 8.3]{HT05}. 
%This measure is different from  the obvious Haar probability measure. 
The following result of Collins-Male \cite[Theorem 1.5]{CM14} confirms that the same conclusion holds when the sampling is done with the more obvious Haar measure (which is a different measure). 
For any $k$-tuple $(u_1^{(N)},...,u_k^{(N)})\in U(N)^k$, let $\rho_{(u_1^{(N)},...,u_k^{(N)})} :F_k \to U(N)$ be the unitary representation of $F_k$ where the generator $x_i$ is sent to  $u_i^{(N)}$.

\begin{thm}\label{CM}
For each $N\geq1$, let $u_1^{(N)},...,u_k^{(N)}$ be a family of independent random matrices in $U(N)$ sampled with the Haar measure.
Then, almost surely,  the sequence of unitary representations $\rho_{(u_1^{(N)},...,u_k^{(N)})}$ strongly converges.
\end{thm}

\subsubsection{Random permutations} 
For any integer $m\geq 3$, let $[m]:=\{1,...,m\}$ and $\mathbf{S}_m$ denote the permutation group of $[m]$. Let $V_m:=\ell^2([m],\mathbb{C})$ be the $m$-dimensional complex $\ell^2$-space endowed with the standard Euclidean metric, and $V^0_m\subset V_m$ the subspace of functions with average $0$. The group $\mathbf{S}_m$ acts on $V_m$ by the standard unitary representation (via $0-1$ matrices):
$$\mathbf{std}: \mathbf{S}_m\to \End(V_m),$$ 
and on $V^0_m$, which is the $(m-1)$-dimensional irreducible component, orthogonal to the diagonal direction:
$$\mathbf{std}: \mathbf{S}_m\to \End(V^0_m).$$

For any $k$-tuple $(s_1^{(m)},...,s_k^{(m)})\in \mathbf{S}_m$, let $\rho_{(s_1^{(m)},...,s_k^{(m)})} :F_k \to \End(V^0_m)$ be the unitary representation of $F_k$ 
equal to $\mathbf{std} \circ \theta_{(s_1^{(m)},...,s_k^{(m)})}$ where  $\theta_{(s_1^{(m)},...,s_k^{(m)})}:F_k\to \mathbf{S}_m$ is the homomorphism that sends
 the generator $x_i$  to  $s_i^{(m)}$.  Below is a special case of the strong convergence theorem of Bordenave-Collins \cite[Theorem 3]{BC19}:
\begin{thm}\label{BC}
For each $m>1$, let  $s_1^{(m)},...,s_k^{(m)}$ be a family of independent random permutations in $\mathbf{S}_m$ with respect to the uniform measure.
Then, the sequence of unitary representations $\rho_{(s_1^{(m)},...,s_k^{(m)})}$ strongly converges in probability, in the following sense: 
for any $\varepsilon>0$, for any $z\in \mathbb{C}[\Gamma]$,
with probability tending to $1$ as $m\to \infty$, we have
$$\|\rho_{(s_1^{(m)},...,s_k^{(m)})}(z)\|\leq \|\lambda_{F_k}(z)\|+\varepsilon.$$
\end{thm}

We will also need the following elementary lemma, just in the special case $k=2$ and $F_2=\langle a,b\rangle$. 
Let $\mathbb{S}_{V^0_m}$ be the unit sphere of $V^0_m$ endowed with the standard Euclidean metric. 
%We will also need an elementary lemma:
\begin{lem}\label{proba1}
For each $m>1$, let  $s_1^{(m)},s_2^{(m)}$ be two independent random permutations in $\mathbf{S}_m$ with respect to the uniform measure.
With probability tending to $1$ as $m\to \infty$,  the unitary representation $\rho_m:=\rho_{(s_1^{(m)},s_2^{(m)})}:   F_2\to \End(V^0_{m})$
is such that $\rho_m(a)$, $\rho_m(b)$ and $\rho_m(b^{-1}a)$ all have a respective fixed point in $\mathbb{S}_{V^0_m}$.
\end{lem}
\begin{proof}
We claim that the unitary transformation $\rho_m(a)$ has a fixed point in $\mathbb{S}_{V^0_m}$ if and only if the corresponding permutation $s_1^{(m)}$ admits a nontrivial cycle decomposition, in the sense that it admits at least two cycles. Indeed, if $x,y\in [m]$ belong to two different cycles of $s_1^{(m)}$, then the subsets 
$$X:=\{(s_1^{(m)})^l(x)\}_{l\in\mathbb{Z}},\quad Y=\{(s_1^{(m)})^l(y)\}_{l\in\mathbb{Z}}$$ are disjoint, and if $\sharp X, \sharp Y$ denote their sizes, one can define the function
$$f(z) = \left\{ \begin{array}{rcl}
\sqrt{\frac{\sharp Y}{\sharp X.\sharp Y+(\sharp X)^2}} & \mbox{for} & z\in X \\
-\frac{\sharp X}{\sharp Y} \sqrt{\frac{\sharp Y}{\sharp X.\sharp Y+(\sharp X)^2}} & \mbox{for} & z\in Y \\
0 & \mbox{for} & z\in [m]\setminus (X\cup Y)
\end{array}\right..$$
One checks that $f$ has average $0$ and $\ell^2$-norm $1$, so it is an element of $\mathbb{S}_{V^0_m}$. Moreover, it is indeed a fixed point of  $\rho_m(a) := \mathbf{std}(s_1^{(m)})$. Conversely, if $s_1^{(m)}$ is made of one cycle, then it does not fix any nontrivial function with average $0$ (via the standard representation). This shows the claim.

Next, since $s_1^{(m)}$ is chosen uniformly at random  in $\mathbf{S}_m$, the probability that it admits a nontrivial cycle decomposition converges to $1$ as $m\to \infty$ (by a simple counting argument, the probability of a random permutation of $\mathbf{S}_m$ to have a cycle of length $m$ is $\frac{1}{m}$). 
Next, $s_2^{(m)}$ is chosen independently of $s_1^{(m)}$ and so $(s_2^{(m)})^{-1}s_1^{(m)}$ is also a uniformly random element of $\mathbf{S}_m$. Thus,  $\theta_m(a)$,  $\theta_m(b)$,  $\theta_m(b^{-1}a)$ all have a nontrivial cycle decomposition with probability tending to $1$. We conclude with the claim above.

\end{proof}

\subsubsection{Representations of surface groups} 
Based on the paper of Bordenave-Collins \cite{BC19}, Louder-Magee \cite[Corollary 1.2]{LM25} constructed strongly convergent sequence of unitary representations for limit groups, in particular surface groups.

\begin{thm}\label{LM}
Let $\Gamma$ be the fundamental group of a closed oriented surface. There exists a sequence of unitary representations $\rho_j:\Gamma \to U(N_j)$ which strongly converges, and such that $\rho_j(\Gamma)$ is a finite subgroup of $U(N_j)$ for any $j\geq 1$. 

\end{thm}

\vspace{1em}

\bibliographystyle{alpha}
\bibliography{biblio_24_01_08}

\newcommand{\etalchar}[1]{$^{#1}$}
\begin{thebibliography}{CGVTvH24}

\bibitem[AB83]{AB83}
Michael~Francis Atiyah and Raoul Bott.
\newblock The yang-mills equations over riemann surfaces.
\newblock {\em Philosophical Transactions of the Royal Society of London.
  Series A, Mathematical and Physical Sciences}, 308(1505):523--615, 1983.

\bibitem[AB22]{AB22}
Mikl{\'o}s Ab{\'e}rt and Ian Biringer.
\newblock Unimodular measures on the space of all riemannian manifolds.
\newblock {\em Geometry \& Topology}, 26(5):2295--2404, 2022.

\bibitem[And02]{Andrews02}
Ben Andrews.
\newblock Notes on the isometric embedding problem and the nash-moser implicit
  function theorem.
\newblock In {\em Surveys in analysis and operator theory}, volume~40, pages
  157--209. Australian National University, Mathematical Sciences Institute,
  2002.

\bibitem[BBH{\etalchar{+}}94]{BBH94}
Fabrice Bethuel, Ha{\"\i}m Brezis, Fr{\'e}d{\'e}ric H{\'e}lein, et~al.
\newblock {\em Ginzburg-landau vortices}, volume~13.
\newblock Springer, 1994.

\bibitem[BC19]{BC19}
Charles Bordenave and Beno{\^\i}t Collins.
\newblock Eigenvalues of random lifts and polynomials of random permutation
  matrices.
\newblock {\em Annals of Mathematics}, 190(3):811--875, 2019.

\bibitem[BCG91]{BCG91}
G\'{e}rard Besson, Gilles Courtois, and Sylvestre Gallot.
\newblock Volume et entropie minimale des espaces localement sym\'{e}triques.
\newblock {\em Invent. Math.}, 103(2):417--445, 1991.

\bibitem[BCG95]{BCG95}
G\'{e}rard Besson, Gilles Courtois, and Sylvestre Gallot.
\newblock Entropies et rigidit\'{e}s des espaces localement sym\'{e}triques de
  courbure strictement n\'{e}gative.
\newblock {\em Geom. Funct. Anal.}, 5(5):731--799, 1995.

\bibitem[BCG07]{BCG07}
G{\'e}rard Besson, Gilles Courtois, and Sylvestre Gallot.
\newblock {In\'{e}galit\'{e}s} de milnor--wood {g\'{e}om\'{e}triques}.
\newblock {\em Commentarii Mathematici Helvetici}, 82(4):753--803, 2007.

\bibitem[BCHM18]{BCHM18}
Turgay Bayraktar, Dan Coman, Hendrik Herrmann, and George Marinescu.
\newblock A survey on zeros of random holomorphic sections.
\newblock {\em arXiv preprint arXiv:1807.05390}, 2018.

\bibitem[BdLHV08]{BDLHV08}
Bachir Bekka, Pierre de~La~Harpe, and Alain Valette.
\newblock Kazhdan's property (t).
\newblock {\em Kazhdan's property (T)}, 11, 2008.

\bibitem[Bry82a]{Bryant82}
Robert~L Bryant.
\newblock Conformal and minimal immersions of compact surfaces into the
  4-sphere.
\newblock {\em Journal of Differential Geometry}, 17(3):455--473, 1982.

\bibitem[Bry82b]{Bryant82b}
Robert~L Bryant.
\newblock Submanifolds and special structures on the octonians.
\newblock {\em Journal of Differential Geometry}, 17(2):185--232, 1982.

\bibitem[Bry85]{Bryant85}
Robert~L Bryant.
\newblock Minimal surfaces of constant curvature in {$S^n$}.
\newblock {\em Transactions of the American Mathematical Society},
  290(1):259--271, 1985.

\bibitem[Bry24]{Bryant24}
Robert Bryant.
\newblock Personal communication, 2024.

\bibitem[Cal67]{Calabi67}
Eugenio Calabi.
\newblock Minimal immersions of surfaces in {E}uclidean spheres.
\newblock {\em Journal of Differential Geometry}, 1(1-2):111--125, 1967.

\bibitem[CC73]{CC73}
Bang-Yen Chen and SS~Chern.
\newblock On the surface with parallel mean curvature vector.
\newblock {\em Indiana University Mathematics Journal}, 22(7):655--666, 1973.

\bibitem[CGVTvH24]{CGZTvH24}
Chi-Fang Chen, Jorge Garza-Vargas, Joel~A Tropp, and Ramon van Handel.
\newblock A new approach to strong convergence.
\newblock {\em arXiv preprint arXiv:2405.16026}, 2024.

\bibitem[Che80]{Chen80}
Bang-Yen Chen.
\newblock Surfaces with parallel normalized mean curvature vector.
\newblock {\em Monatshefte f{\"u}r Mathematik}, 90:185--194, 1980.

\bibitem[CL72]{CL72}
Bang-yen Chen and Gerald~D Ludden.
\newblock Surfaces with mean curvature vector parallel in the normal bundle1.
\newblock {\em Nagoya Mathematical Journal}, 47:161--167, 1972.

\bibitem[CLS25]{CLS25}
Riccardo Caniato, Xingzhe Li, and Antoine Song.
\newblock Area rigidity for the regular representation of surface groups.
\newblock {\em arXiv preprint arXiv:2508.19480}, 2025.

\bibitem[CM11]{CM11}
Tobias~H Colding and William~P Minicozzi.
\newblock {\em A course in minimal surfaces}, volume 121.
\newblock American Mathematical Soc., 2011.

\bibitem[CM14]{CM14}
Beno{\^\i}t Collins and Camille Male.
\newblock The strong asymptotic freeness of haar and deterministic matrices.
\newblock In {\em Annales scientifiques de l'{\'E}cole Normale Sup{\'e}rieure},
  volume~47, pages 147--163, 2014.

\bibitem[CMN22]{CMN22}
Danny Calegari, Fernando~C Marques, and Andr{\'e} Neves.
\newblock Counting minimal surfaces in negatively curved 3-manifolds.
\newblock {\em Duke Math. J}, 171:1615--1648, 2022.

\bibitem[CMP19]{CMP19}
Beno{\^\i}t Collins, Michael Magee, and Doron Puder.
\newblock Automorphism-invariant positive definite functions on free groups.
\newblock {\em arXiv preprint arXiv:1906.01518}, 2019.

\bibitem[CS15]{CS15}
Jaigyoung Choe and Marc Soret.
\newblock New minimal surfaces in {$S^3$} desingularizing the {C}lifford tori.
\newblock {\em Mathematische Annalen}, 364:763--776, 2015.

\bibitem[DCFF92]{DoCarmo92}
Manfredo~Perdigao Do~Carmo and J~Flaherty~Francis.
\newblock {\em Riemannian geometry}, volume~6.
\newblock Springer, 1992.

\bibitem[dlH85]{DLH85}
Pierre de~la Harpe.
\newblock Reduced {$C^*$}-algebras of discrete groups which are simple with a
  unique trace.
\newblock In {\em Operator algebras and their connections with topology and
  ergodic theory (Busteni, 1983)}, volume 1132 of {\em Lecture Notes in Math.},
  pages 230--253. Springer, Berlin, 1985.

\bibitem[dLH07]{DLH07}
Pierre de~La~Harpe.
\newblock On simplicity of reduced {C*}-algebras of groups.
\newblock {\em Bulletin of the London Mathematical Society}, 39(1):1--26, 2007.

\bibitem[DM23]{DM23}
Georgios Daskalopoulos and Chikako Mese.
\newblock Infinite energy harmonic maps from riemann surfaces to {CAT(0)}
  spaces.
\newblock {\em The Journal of Geometric Analysis}, 33(10):337, 2023.

\bibitem[Dun99]{Dunfield99}
Nathan~M Dunfield.
\newblock Cyclic surgery, degrees of maps of character curves, and volume
  rigidity for hyperbolic manifolds.
\newblock {\em Inventiones mathematicae}, 136(3):623--657, 1999.

\bibitem[DW07]{DW07}
Georgios Daskalopoulos and Richard~A Wentworth.
\newblock Harmonic maps and teichm{\"u}ller theory.
\newblock {\em Handbook of Teichm{\"u}ller theory}, 1:33--109, 2007.

\bibitem[ES64]{ES64}
James Eells and Joseph~H Sampson.
\newblock Harmonic mappings of riemannian manifolds.
\newblock {\em American journal of mathematics}, 86(1):109--160, 1964.

\bibitem[ESI00]{ESI00}
Ahmad El~Soufi and Sa{\"\i}d Ilias.
\newblock Riemannian manifolds admitting isometric immersions by their first
  eigenfunctions.
\newblock {\em Pacific Journal of Mathematics}, 195(1):91--99, 2000.

\bibitem[FM11]{FaMa11}
Benson Farb and Dan Margalit.
\newblock {\em A primer on mapping class groups (pms-49)}, volume~41.
\newblock Princeton university press, 2011.

\bibitem[Fra04]{Francaviglia04}
Stefano Francaviglia.
\newblock Hyperbolic volume of representations of fundamental groups of cusped
  3-manifolds.
\newblock {\em International Mathematics Research Notices}, 2004(9):425--459,
  2004.

\bibitem[GG19]{GG19}
Pedro Gaspar and Marco~AM Guaraco.
\newblock The weyl law for the phase transition spectrum and density of limit
  interfaces.
\newblock {\em Geometric and Functional Analysis}, 29:382--410, 2019.

\bibitem[Gol84]{Goldman84}
William~M Goldman.
\newblock The symplectic nature of fundamental groups of surfaces.
\newblock {\em Advances in Mathematics}, 54(2):200--225, 1984.

\bibitem[GP09]{Garcia-Prada09}
Oscar Garc{\'\i}a-Prada.
\newblock Higgs bundles and surface group representations.
\newblock {\em Moduli spaces and vector bundles}, pages 265--310, 2009.

\bibitem[GT77]{GT77}
David Gilbarg and Neil~S. Trudinger.
\newblock {\em Elliptic partial differential equations of second order}, volume
  224.
\newblock Springer, 1977.

\bibitem[ha]{Agol11}
Ian~Agol (https://mathoverflow.net/users/1345/ian agol).
\newblock Hurwitz's automorphisms theorem for infinite genus riemann surfaces.
\newblock MathOverflow.
\newblock URL:https://mathoverflow.net/q/82584 (version: 2011-12-03).

\bibitem[Ham15]{Hamenstadt15}
Ursula Hamenst{\"a}dt.
\newblock Incompressible surfaces in rank one locally symmetric spaces.
\newblock {\em Geometric and Functional Analysis}, 3(25):815--859, 2015.

\bibitem[Han96]{Hano96}
Jun-ichi Hano.
\newblock Conformal immersions of compact {R}iemann surfaces into the
  {$2n$}-sphere {$(n\geq2)$}.
\newblock {\em Nagoya Mathematical Journal}, 141:79--105, 1996.

\bibitem[Har67]{Hartman67}
Philip Hartman.
\newblock On homotopic harmonic maps.
\newblock {\em Canadian journal of mathematics}, 19:673--687, 1967.

\bibitem[Hid23]{Hide23}
Will Hide.
\newblock Effective lower bounds for spectra of random covers and random
  unitary bundles.
\newblock {\em arXiv preprint arXiv:2305.04584}, 2023.

\bibitem[HM23]{HM23}
Will Hide and Michael Magee.
\newblock Near optimal spectral gaps for hyperbolic surfaces.
\newblock {\em Annals of Mathematics}, 198(2):791--824, 2023.

\bibitem[HT05]{HT05}
Uffe Haagerup and Steen Thorbj{\o}rnsen.
\newblock A new application of random matrices: is not a group.
\newblock {\em Annals of Mathematics}, pages 711--775, 2005.

\bibitem[Hul00]{Hulin00}
Dominique Hulin.
\newblock {K\"{a}hler-Einstein} metrics and projective embeddings.
\newblock {\em The Journal of Geometric Analysis}, 10(3):525--528, 2000.

\bibitem[Jos08]{Jost08}
J{\"u}rgen Jost.
\newblock {\em Riemannian geometry and geometric analysis}, volume 42005.
\newblock Springer, 2008.

\bibitem[JY09]{JY09}
J{\"u}rgen Jost and Shing~Tung Yau.
\newblock Harmonic mappings and moduli spaces of riemann surfaces.
\newblock {\em Surveys in differential geometry}, 14(1):171--196, 2009.

\bibitem[JZ97]{JZ97}
J{\"u}rgen Jost and Kang Zuo.
\newblock Harmonic maps of infinite energy and rigidity results for
  representations of fundamental groups of quasiprojective varieties.
\newblock {\em Journal of Differential Geometry}, 47(3):469--503, 1997.

\bibitem[Kaz70]{Kazhdan70}
David Kazhdan.
\newblock Arithmetic varieties and their fields of quasi-definition.
\newblock In {\em Actes du Congres International des Math{\'e}maticiens (Nice,
  1970)}, volume~2, pages 321--325, 1970.

\bibitem[Ken76]{Kenmotsu76}
Katsuei Kenmotsu.
\newblock On minimal immersions of {$R^2$} into {$S^N$}.
\newblock {\em Journal of the Mathematical Society of Japan}, 28(1):182--191,
  1976.

\bibitem[KM12]{KM12}
Jeremy Kahn and Vladimir Markovic.
\newblock Immersing almost geodesic surfaces in a closed hyperbolic three
  manifold.
\newblock {\em Annals of Mathematics}, pages 1127--1190, 2012.

\bibitem[KMS23]{KMS23}
Jeremy Kahn, Vladimir Markovic, and Ilia Smilga.
\newblock Geometrically and topologically random surfaces in a closed
  hyperbolic three manifold.
\newblock {\em arXiv preprint arXiv:2309.02847}, 2023.

\bibitem[KPS88]{KPS88}
Hermann Karcher, Ulrich Pinkall, and Ivan Sterling.
\newblock New minimal surfaces in {$S^{3}$}.
\newblock {\em Journal of Differential Geometry}, 28(2):169--185, 1988.

\bibitem[KS20]{KS20}
Mikhail Karpukhin and Daniel~L Stern.
\newblock Min-max harmonic maps and a new characterization of conformal
  eigenvalues.
\newblock {\em arXiv preprint arXiv:2004.04086}, 2020.

\bibitem[KY10]{KY10}
Nikolaos Kapouleas and Seong-Deog Yang.
\newblock Minimal surfaces in the three-sphere by doubling the {C}lifford
  torus.
\newblock {\em American journal of mathematics}, 132(2):257--295, 2010.

\bibitem[Lab13]{Labourie13}
Fran{\c{c}}ois Labourie.
\newblock {\em Lectures on representations of surface groups}.
\newblock EMS publishing house, 2013.

\bibitem[Lab17]{Labourie17}
Fran{\c{c}}ois Labourie.
\newblock Cyclic surfaces and {H}itchin components in rank 2.
\newblock {\em Annals of Mathematics}, 185(1):1--58, 2017.

\bibitem[Lab19]{Labourie19}
Fran{\c{c}}ois Labourie.
\newblock Introduction to hitchin representations.
\newblock {UC} {B}erkeley graduate course,
  https://flab.perso.math.cnrs.fr/Hitchin/index.html, 2019.

\bibitem[Law70]{Lawson70}
H~Blaine Lawson.
\newblock Complete minimal surfaces in {$S^3$}.
\newblock {\em Annals of Mathematics}, pages 335--374, 1970.

\bibitem[Li19]{Li19}
Qiongling Li.
\newblock An introduction to higgs bundles via harmonic maps.
\newblock {\em SIGMA. Symmetry, Integrability and Geometry: Methods and
  Applications}, 15:035, 2019.

\bibitem[Li23]{Li23}
Xingzhe Li.
\newblock Generic scarring for minimal hypersurfaces in manifolds thick at
  infinity with a thin foliation at infinity.
\newblock {\em arXiv preprint arXiv:2312.03591}, 2023.

\bibitem[LMH25]{LM25}
Larsen Louder, Michael Magee, and Will Hide.
\newblock Strongly convergent unitary representations of limit groups.
\newblock {\em Journal of Functional Analysis}, 288(6):110803, 2025.

\bibitem[LN21]{LN21}
Ben Lowe and Andre Neves.
\newblock Minimal surface entropy and average area ratio.
\newblock {\em arXiv preprint arXiv:2110.09451}, 2021.

\bibitem[Loh90]{Lohkamp90}
Jochen Lohkamp.
\newblock An existence theorem for harmonic maps.
\newblock {\em manuscripta mathematica}, 67(1):21--23, 1990.

\bibitem[Lub94]{Lubotzky94}
Alex Lubotzky.
\newblock {\em Discrete groups, expanding graphs and invariant measures},
  volume 125.
\newblock Springer Science \& Business Media, 1994.

\bibitem[Mag21]{Magee21}
Michael Magee.
\newblock Random unitary representations of surface groups ii: the large $ n $
  limit.
\newblock {\em arXiv preprint arXiv:2101.03224}, 2021.

\bibitem[Mag22]{Magee22}
Michael Magee.
\newblock Random unitary representations of surface groups i: asymptotic
  expansions.
\newblock {\em Communications in mathematical physics}, 391(1):119--171, 2022.

\bibitem[Mag24]{Magee24}
Michael Magee.
\newblock Strong convergence of unitary and permutation representations of
  discrete groups.
\newblock {\em to appear}, 2024.

\bibitem[McM13]{McMullen13}
Curtis~T McMullen.
\newblock Entropy on riemann surfaces and the jacobians of finite covers.
\newblock {\em Commentarii Mathematici Helvetici}, 88(4):953--964, 2013.

\bibitem[MNS19]{MNS19}
Fernando~C Marques, Andr{\'e} Neves, and Antoine Song.
\newblock Equidistribution of minimal hypersurfaces for generic metrics.
\newblock {\em Inventiones mathematicae}, 216(2):421--443, 2019.

\bibitem[Moh22]{Mohsen22}
Jean-Paul Mohsen.
\newblock Construction of negatively curved complete intersections.
\newblock {\em Duke Mathematical Journal}, 171(9):1843--1878, 2022.

\bibitem[MRVS22]{MRVS22}
Antonin Monteil, R{\'e}my Rodiac, and Jean Van~Schaftingen.
\newblock Renormalised energies and renormalisable singular harmonic maps into
  a compact manifold on planar domains.
\newblock {\em Mathematische Annalen}, 383(3):1061--1125, 2022.

\bibitem[MS19]{MS19}
Henrik Matthiesen and Anna Siffert.
\newblock Handle attachment and the normalized first eigenvalue.
\newblock {\em arXiv preprint arXiv:1909.03105}, 2019.

\bibitem[Nad96]{Nadirashvili96}
Nikolai Nadirashvili.
\newblock Hadamard's and calabi-yau's conjectures on negatively curved and
  minimal surfaces.
\newblock {\em Inventiones mathematicae}, 126(3):457--466, 1996.

\bibitem[NS65]{NS65}
Mudumbai~S Narasimhan and Conjeeveram~S Seshadri.
\newblock Stable and unitary vector bundles on a compact riemann surface.
\newblock {\em Annals of Mathematics}, 82(3):540--567, 1965.

\bibitem[NS15]{NS15}
Nikolai~Semenovich Nadirashvili and Yannick Sire.
\newblock Conformal spectrum and harmonic maps.
\newblock {\em Moscow Mathematical Journal}, 15(1):123--140, 2015.

\bibitem[Per]{Pereira11}
Jorge~{Vit\'{o}rio} Pereira.
\newblock Geometry of complex elliptic curves.
\newblock MathOverflow.
\newblock URL:https://mathoverflow.net/q/68963 (version: 2011-06-28).

\bibitem[Pet14]{Petrides14}
Romain Petrides.
\newblock Existence and regularity of maximal metrics for the first {L}aplace
  eigenvalue on surfaces.
\newblock {\em Geometric and Functional Analysis}, 24(4):1336--1376, 2014.

\bibitem[Pow75]{Powers75}
Robert~T Powers.
\newblock Simplicity of the {$C^*$-algebra} associated with the free group on
  two generators.
\newblock 1975.

\bibitem[PS17]{PS17}
Rafael Potrie and Andr{\'e}s Sambarino.
\newblock Eigenvalues and entropy of a {H}itchin representation.
\newblock {\em Inventiones mathematicae}, 209(3):885--925, 2017.

\bibitem[Rho93]{Rhodes93}
John~A Rhodes.
\newblock Sequences of metrics on compact riemann surfaces.
\newblock {\em Duke Math. J.}, 72(3):725--738, 1993.

\bibitem[Ros22]{Ros22}
Antonio Ros.
\newblock On the first eigenvalue of the {L}aplacian on compact surfaces of
  genus three.
\newblock {\em Journal of the Mathematical Society of Japan}, 74(3):813--828,
  2022.

\bibitem[Sag23]{Sagman2023}
Nathaniel Sagman.
\newblock Infinite energy equivariant harmonic maps, domination, and anti-de
  {S}itter $3 $-manifolds.
\newblock {\em Journal of Differential Geometry}, 124(3):553--598, 2023.

\bibitem[SC91]{CS91}
T~Steger and M~Cowling.
\newblock The irreducibility of restrictions of unitary representations of
  lattices.
\newblock 1991.

\bibitem[Sen03]{Sengupta03}
Ambar~N Sengupta.
\newblock The volume measure for flat connections as limit of the yang--mills
  measure.
\newblock {\em Journal of Geometry and Physics}, 47(4):398--426, 2003.

\bibitem[Son23a]{Antoine23b}
Antoine Song.
\newblock Entropy and stability of hyperbolic manifolds.
\newblock 2023.
\newblock arXiv:2302.07422 v2 [math.DG].

\bibitem[Son23b]{Antoine23a}
Antoine Song.
\newblock Spherical volume and {s}pherical {P}lateau {p}roblem.
\newblock {\em to appear in {S}\'{e}minaire de th\'{e}orie spectrale et
  g\'{e}om\'{e}trie}, 2023.

\bibitem[Son24]{Antoine24b}
{Antoine} Song.
\newblock Hyperbolic groups and spherical minimal surfaces.
\newblock arXiv preprint, 2024.

\bibitem[Son25a]{Antoine25}
Antoine Song.
\newblock Area and energy of unitary representations.
\newblock lecture notes in preparation, 2025.

\bibitem[Son25b]{Antoine25a}
Antoine Song.
\newblock Typical minimal surfaces.
\newblock to appear in the proceedings of the 2024 International Congress of
  Basic Science, 2025.

\bibitem[Spe17]{Speicher17}
Roland Speicher.
\newblock Free probability theory: and its avatars in representation theory,
  random matrices, and operator algebras; also featuring: non-commutative
  distributions.
\newblock {\em Jahresbericht der Deutschen Mathematiker-Vereinigung},
  119(1):3--30, 2017.

\bibitem[SU81]{SU81}
Jonathan Sacks and Karen Uhlenbeck.
\newblock The existence of minimal immersions of 2-spheres.
\newblock {\em Annals of mathematics}, pages 1--24, 1981.

\bibitem[SU82]{SU82}
Jonathan Sacks and Karen Uhlenbeck.
\newblock Minimal immersions of closed riemann surfaces.
\newblock {\em Transactions of the American Mathematical Society},
  271(2):639--652, 1982.

\bibitem[SY79]{SY79}
Richard Schoen and Shing-Tung Yau.
\newblock Existence of incompressible minimal surfaces and the topology of
  three dimensional manifolds with non-negative scalar curvature.
\newblock {\em Annals of Mathematics}, 110(1):127--142, 1979.

\bibitem[SY97]{SY97}
Richard~M Schoen and Shing~Tung Yau.
\newblock Lectures on harmonic maps.
\newblock 1997.

\bibitem[SZ21]{SZ21}
Antoine Song and Xin Zhou.
\newblock Generic scarring for minimal hypersurfaces along stable
  hypersurfaces.
\newblock {\em Geometric and Functional Analysis}, 31(4):948--980, 2021.

\bibitem[SZ23]{SZ23}
Bernard Shiffman and Steve Zelditch.
\newblock Stochastic {K\"{a}hler} geometry: from random zeros to random
  metrics.
\newblock {\em arXiv preprint arXiv:2303.11559}, 2023.

\bibitem[Tak66]{Takahashi66}
Tsunero Takahashi.
\newblock Minimal immersions of {R}iemannian manifolds.
\newblock {\em Journal of the Mathematical Society of Japan}, 18(4):380--385,
  1966.

\bibitem[Tho24]{Thomas24}
Alexander Thomas.
\newblock A gentle introduction to the non-abelian hodge correspondence.
\newblock {\em L’Enseignement Math{\'e}matique}, 2024.

\bibitem[Tia90]{Tian90}
Gang Tian.
\newblock On a set of polarized {K{\"a}hler} metrics on algebraic manifolds.
\newblock {\em Journal of Differential Geometry}, 32(1):99--130, 1990.

\bibitem[Tol89]{Toledo89}
Domingo Toledo.
\newblock Representations of surface groups in complex hyperbolic space.
\newblock {\em Journal of Differential Geometry}, 29(1):125--133, 1989.

\bibitem[Wit91]{Witten91}
Edward Witten.
\newblock On quantum gauge theories in two dimensions.
\newblock {\em Communications in Mathematical Physics}, 141(1):153--209, 1991.

\bibitem[Wol91]{Wolf91}
Michael Wolf.
\newblock Infinite energy harmonic maps and degeneration of hyperbolic surfaces
  in moduli space.
\newblock {\em Journal of differential Geometry}, 33(2):487--539, 1991.

\bibitem[Wol07]{Wolpert07}
Scott~A Wolpert.
\newblock Cusps and the family hyperbolic metric.
\newblock {\em Duke Math. J.}, 138(3):423--443, 2007.

\bibitem[Yau74]{Yau74}
Shing-Tung Yau.
\newblock Submanifolds with constant mean curvature.
\newblock {\em American Journal of Mathematics}, 96(2):346--366, 1974.

\bibitem[Yau82]{Yau82}
Shing-Tung Yau.
\newblock {\em Seminar on differential geometry}.
\newblock Number 102. Princeton University Press, 1982.

\bibitem[Zar22]{Zargar22}
Masoud Zargar.
\newblock Random flat bundles and equidistribution.
\newblock {\em arXiv preprint arXiv:2210.09547}, 2022.

\end{thebibliography}

\end{document}